%% file: main.tex
\documentclass[11pt]{article}
  \usepackage[margin=0.81in]{geometry} 

\usepackage{mathrsfs} 
\usepackage{kpfonts,comment}

\usepackage{caption}
\usepackage{subcaption}
\usepackage{array}

\usepackage{kz_style}

\title{Symmetric (Optimistic)  Natural Policy Gradient for \\Multi-agent Learning with Parameter Convergence} 

\begin{document} 
\author{Sarath Pattathil \and  Kaiqing Zhang \and Asuman Ozdaglar\thanks{Authors are with the Department of Electrical Engineering and Computer Science, Massachusetts Institute of Technology, Cambridge, MA, USA. \{sarathp, kaiqing, asuman\}@mit.edu.
S.P. acknowledges support from MathWorks Engineering Fellowship. A.O and K.Z. were supported by MIT-DSTA grant 031017-00016. K.Z. also acknowledges  support  from Simons-Berkeley Research Fellowship.  
} 
} 
\date{}
\maketitle 

\begin{abstract} 
Multi-agent interactions are increasingly important in the context of  reinforcement learning,   and the theoretical foundations of policy gradient methods have attracted surging research interest. We investigate  the global convergence of natural policy gradient (NPG) algorithms in multi-agent learning. 
We first show that vanilla NPG may not have {\it parameter convergence}, i.e., the convergence of the vector  that parameterizes the policy, even when the costs are regularized (which {enabled} strong  convergence guarantees in the {\it policy space} in the literature). This non-convergence of parameters leads to stability issues in learning, which becomes especially relevant in the function approximation setting, where we can only operate on low-dimensional parameters, instead of the high-dimensional policy.  
We then propose {variants} of the NPG algorithm, for several  standard multi-agent learning scenarios: two-player zero-sum matrix and Markov games, and multi-player monotone games,  
with global last-iterate  parameter convergence guarantees. We also generalize the results to certain function approximation settings.  
Note that in our algorithms, the agents take {\it symmetric} roles. 
Our results might also be of independent interest for  solving nonconvex-nonconcave minimax optimization problems with certain structures.  Simulations are  also  provided to corroborate our theoretical findings. 
\end{abstract}

\input{Introduction}

\input{Formulation}

\input{Tabular_Matrix} 
\input{FA_Matrix}

\input{General_Monotone}

\input{Markov}

\input{Simulations}

\input{Conclusion}

\bibliographystyle{ims}
\bibliography{main,mybibfile}

\newpage
\appendix

\input{Appendix}

\end{document}

%% file: Introduction.tex

\section{Introduction}\label{sec:Introduction} 

Policy gradient (PG) methods have served as the workhorse of modern reinforcement learning (RL) \citep{schulman2015trust,schulman2017proximal,haarnoja2018soft}, and enjoy the desired properties of being scalable to large state-action spaces, stability with function approximation, as well as sample efficiency. In fact, policy gradient methods have achieved impressive empirical performance in multi-agent RL \citep{lowe2017multi,yu2021surprising}, the regime where many   RL's recent successes are  pertinent to \citep{silver2017mastering,OpenAI_dota,shalev2016safe}.

Despite the tremendous empirical successes, theoretical foundations  of PG methods, even for the single-agent setting, have  not been uncovered until recently \citep{fazel2018global,agarwal2019optimality,zhang2019global,wang2019neural,mei2020global,cen2021fastMDP}. The theoretical  understanding of PG methods for multi-agent RL remains  largely elusive, except for several recent attempts \citep{daskalakis2020independent,zhao2021provably,wei2021last,cen2021fast}. The key challenge is that in the policy parameter space, even for the basic two-player zero-sum matrix game, the problem becomes {\it nonconvex-nonconcave} and is computationally intractable in general \citep{daskalakis2021complexity}.

In this paper, we aim to fill in the gap by studying the global convergence of natural PG (NPG) \citep{kakade2002natural}, which forms the basis for many popular PG algorithms (e.g., Proximal Policy Optimization (PPO)/Trust Region Policy Optimization (TRPO)), in the parameter space and for multi-agent learning. We are interested in the setting where the  agents take {\it symmetric} roles and operate {\it independently}, as it does not require a central coordinator and it scales favorably with the number of agents. Analysis of this setting is challenging precisely because the concurrent updates of the agents makes the learning environment {\it non-stationary}  from one agent's perspective. {With {\it asymmetric}  update rules among agents, the non-stationarity issue can be mitigated, and the global convergence of PG methods has been established lately in \citep{zhang2019policyb,daskalakis2020independent,zhao2021provably}. However, though being valid as an optimization scheme, asymmetric update-rules might be hard to justify in game-theoretic multi-agent learning with symmetric players. It is thus desirable to develop provably convergent PG methods  with symmetric update rules.} 

We focus on the {\it last-iterate}   convergence of the {\it policy parameters}, which is critical to establish in order to avoid stability issues during learning. For example, if the norm of the parameters blow up, we might end up with precision issues in  computing the gradient and updating the parameters. This is particularly relevant to the setting with function approximation, where we can only operate on {\it low-dimensional} parameters of the policy, instead of the high-dimensional policy per se. Indeed, we aim to explore the convergence property in the function approximation setting to handle  large state-action spaces. Finally, our results  are also motivated from the study of  {\it nonconvex-nonconcave minimax} optimization problems, especially  those with certain structures that yield global convergence of gradient-based methods. We aim to explore such structures in multi-agent learning with {\it parameterized}  policies. We summarize our contributions as follows.  




\paragraph{Contributions.} Our contributions are three-fold. First, we identify the non-convergence issue in the policy parameter space of natural PG methods for RL. {We show that this issue persists even with entropy regularized rewards}. Second, we {develop  symmetric variants of the natural PG method, i.e., both without and with the optimistic updates (\cite{rakhlin2013optimization})} 
and  establish the  last-iterate global convergence  to the Nash equilibrium in the policy parameter space. Third, we generalize the scope of symmetric PG methods in game-theoretic multi-agent learning, including two-player zero-sum matrix and Markov games (MGs), multi-player monotone games, and the corresponding  linear function approximation settings under certain assumptions, in order  to handle enormously large state-action spaces, all with last-iterate parameter convergence rate guarantees. We have  also provided numerical experiments to validate the effectiveness of our algorithms. 


 

\subsection{Related work}\label{sec:related_work}



\paragraph{Policy gradient RL methods for games.}
Gradient-descent-ascent (GDA) with projection on simplexes  can be viewed as symmetric policy gradient methods for solving matrix games with {\it direct policy parameterization} \citep{agarwal2019optimality}, which enjoys an average-iterate  convergence \citep{cesa2006prediction}. Such a guarantee is shared with multiplicative weight update (MWU), which is especially  suitable for repeated matrix games, and   equivalent to natural PG method with tabular softmax policy parameterization \citep{agarwal2019optimality}. To the best of our knowledge, however, neither  parameter convergence nor function approximation has been studied in this context. For Markov games, \cite{zhang2019policyb,bu2019global,hambly2021policy} have studied global convergence of PG methods for those with a linear quadratic structure; for zero-sum Markov games,  \cite{daskalakis2020independent} established global convergence of independent PG with two-timescale stepsizes for the tabular setting;  \cite{zhao2021provably} studied  a double-loop natural PG algorithm with function approximation; {more recently, \cite{alacaoglu2022natural} proposed a framework of natural actor-critic algorithms}.  No last-iterate convergence to the Nash equilibrium was established in these works, and these update-rules  were all {\it asymmetric}. \cite{qiu2021provably} developed  symmetric policy optimization methods for certain zero-sum Markov games with structured transitions. 
Concurrently,  
\cite{zhang2022policy} proposed a policy optimization framework with fast average-iterate convergence guarantees for finite-horizon Markov games. Finally, \cite{leonardos2021global,zhang2021gradient,ding2022independent} have studied global convergence of symmetric PG methods in Markov potential games recently, not focused on  
last-iterate or parameter convergence. 

\paragraph{Last-iterate convergence in constrained multi-agent learning.}
\input{last_iterate.tex}

\paragraph{Nonconvex-nonconcave minimax optimization.} 
It is shown that for general nonconvex-nonconcave minimax optimization, even {\it local}  solution concepts \citep{jin2020local} may not exist,  and finding them can be intractable \citep{daskalakis2021complexity}. Thus, specific structural properties have to be  exploited to design efficient algorithms with {\it global} convergence. \cite{lin2019gradient, thekumparampil2019efficient, nouiehed2019solving,yang2020global} have studied the nonconvex-(strongly)-concave or the nonconvex-Polyak-Lojasiewicz (PL)  or PL-PL settings, with global convergence rate guarantees.  
 The algorithms in these papers are all {\it asymmetric} in that they run the inner loop (which solves the maximization problem) multiple times (or on a faster timescale with larger stepsizes) to reach an approximate solution of the inner optimization problem, and then run one step of descent on the outer problem (or on a slower timescale with smaller stepsizes). Closely related to one motivation of our work,  \cite{vlatakis2019poincare,flokas2021solving,mladenovic2021generalized} studied nonconvex-nonconcave minimax problem with {\it hidden convexity}  structures, and show that GDA can fail to converge globally even so. {Interestingly, the benefit of regularization (more generally, strict convexity), and natural gradient flow under  Fisher information geometry, were also examined in \cite{flokas2021solving,mladenovic2021generalized} to establish some positive  convergence results.}  Different from our work, the dynamics there are in {\it continuous-time}, and the parameterization in \cite{flokas2021solving} is decoupled by dimension, and the convergence rate in the perturbed game is not global. These conditions prevent the application of their results and proof techniques to our setting directly. {Also, last-iterate  finite-time  rates of the iterates were not established in  \cite{mladenovic2021generalized}.} 

\paragraph{Notation.}
For vector $v\in\RR^d$, we use $[v]_a$ with $a\in\{1,2,\cdots,d\}$ to denote the $a$-th element of $v$. {We use $\|v\|$ to denote the $\ell_2$-Euclidean norm of a vector $v$ and $\| Q\|$ to denote the $\ell_2$-induced norm of matrix $Q$. We also use $\| Q\|_\infty$ to denote the infinity norm and $\|Q\|_F$ to denote the Frobenius norm of matrix $Q$.}  
For a finite-set $\cS$, we use $\Delta(\cS)$ to denote the simplex over $\cS$. 
We use $\mathbbm{1}$ to denote  the matrix  of all ones of appropriate dimension. For any positive integer $n$, we use $[n]$ to denote the set $\{1,\cdots,n\}$. We use the  subscript $-i$ to denote the quantities  of all players other than player $i$. $\KL(p\| q)$ denotes the KL divergence between two probability distributions $p$ and $q$. 
For a matrix $C$, we use  $C=[A \ |\ B]$  to denote the concatenation of the component matrices $A$ and $B$. {For two vectors $x,y\in\RR^d$, $x\cdot y$ denotes their inner-product, i.e., $x^\top y$.}  We use $I_d$ to denote an identity matrix of dimension $d$.

%% file: last_iterate.tex

Several papers including \cite{tseng1995linear, azizian2020tight, ibrahim2020linear, fallah2020optimal} and references therein studied the {\it last-iterate}  behavior of strongly monotone games. Furthermore, \cite{golowich2020tight, golowich2020last} extended this analysis to the monotone game setting. However, these papers did not consider {\it parametrized policies}. More specifically, in the matrix game  setting with a simplex constraint, papers including \cite{daskalakis2019last, wei2020linear,cen2021fast}   showed the last-iterate policy convergence of optimistic methods. However, these papers did not consider policy parameterization or the function approximation settings, and some papers required  the assumption that the NE is unique \citep{daskalakis2019last,wei2020linear}.  
For Markov games, \cite{zhao2021provably}  established last-iterate convergence, but not to NE due to asymmetric update; \cite{wei2021last,cen2021fast} were, to the best of our knowledge, the only last-iterate policy convergence results in Markov games with symmetric updates. {However, these works did not study the function approximation setting, or monotone games beyond the two-player zero-sum case. Also, though having greatly inspired our work, }the regularization idea in \cite{cen2021fast} alone cannot  prevent the non-convergence issue of the policy parameters from happening (see \S\ref{sec:prelim}). Our goal,  in contrast, is to study the (last-iterate) convergence of the actual policy parameters, and for more general multi-agent learning settings beyond the tabular zero-sum one. 

%% file: Formulation.tex

%

\section{Motivation \& Background} \label{sec:prelim}


In this section, we introduce the background  of the  natural PG methods we study, with two-player\footnote{Hereafter, we use {\it player} and {\it agent} interchangeably.} zero-sum matrix games  being a motivating  example.  



\paragraph{Zero-sum matrix games.} Two-player zero-sum matrix games are characterized by a tuple $(\cA,\cB,Q)$, where
 $Q\in\RR^{n\times n}$ denotes the cost\footnote{Note that we can also model it as a payoff, with a negative sign.} matrix,   $\cA$ and $\cB$ denote the action spaces of players $1$ and $2$, respectively. For notational simplicity, we assume both action spaces have cardinality $n$, i.e., $|\cA|=|\cB|=n$. Note that our results can be readily generalized to the setting with different action-space cardinalities. For convenience, we use {\it indices} of the actions to denote the actions, i.e., $\cA=\cB=\{1,2,\cdots,n\}$, without loss of generality. Note that the {\it actual actions} of both players for the same index need not to be the same{, and the cost  matrix $Q$ needs not to be symmetric}.  
 The problem can thus be formulated as a minimax (i.e., saddle-point optimization) problem
 \#\label{equ:zs_matrix_original}
 \min_{g\in\Delta(\cA)}\max_{h\in\Delta(\cB)}~~f(g,h):=g^\top Q h,
 \#
 where $g$ and $h$ are referred to as the policies/strategies of the players. 
 By Minimax Theorem \citep{neumann1928theorie}, the $\min$ and $\max$ operators in \eqref{equ:zs_matrix_original} can be interchanged, and the solution concept of {\it Nash equilibrium} (NE), which is defined as a pair of policies $(g^\star,h^\star)$ such that 
 \$
 f(g,h^\star)\geq f(g^\star,h^\star)\geq f(g^\star,h), \text{~~for any~}(g,h)\in\Delta(\cA)\times\Delta(\cB)
 \$
 always holds. In particular, at the Nash equilibrium, the players execute the best-response policies of each other, and have no  incentive to deviate from it.

 
 \paragraph{Policy parameterization.} To develop policy gradient methods for multi-player learning, the policies $(g,h)\in\Delta(\cA)\times\Delta(\cB)$ are parameterized by some parameters $\theta$ and $\nu$. Specifically, consider the following softmax parameterization that is common in practice: for any $a\in\cA$ and $b\in\cB$ 
 \begin{align}\label{equ:policy_para_matrix}
g_{\theta}(a) = \frac{e^{p_{\theta}(a)}}{ \sum_{a' \in\cA} e^{p_{\theta}(a')} }, \qquad h_{\nu}(b) = \frac{e^{q_{\nu}(b)}}{ \sum_{b' \in\cB} e^{q_{\nu}(b')}},
\end{align}
where $\theta,\nu\in\RR^d$ for some integer $d>0$, $p_{\theta},q_{\nu}:\RR^d\to\RR^n$ are two differentiable functions. Note that $g_{\theta}(a),h_{\nu}(b)>0$ for any bounded  $p_{\theta},q_{\nu}$, and $\sum_{a\in\cA}g_{\theta}(a)=\sum_{b\in\cB}h_{\nu}(b)=1$. This parameterization gives the following minimax problem for the zero-sum matrix game:
\#\label{eq:hidden_bilinear_problem_og}
 \min_{\theta\in\RR^n}\max_{\nu\in\RR^n}~~~~f(\theta,\nu):=g^\top_\theta Q h_\nu, 
 \#
 where by a slight abuse of notation, we use $f(\theta,\nu)$ to denote $f(g_\theta,h_\nu)$. 
 In this section and \S\ref{sec:ONPG_Static_tabular}, we consider the tabular softmax parameterization where $p_{\theta}=\theta\in\RR^n$ and $q_\nu=\nu\in\RR^n$. In \S\ref{sec:ONPG_Matrix_FA}, we consider the setting with function approximation where $d<n$.

The benefits of softmax parameterization are that: 1) it transforms a {\it constrained} problem over simplexes  to an {\it unconstrained} one, making it easier to implement; 2) it readily incorporates function approximation to deal with large spaces (see \S\ref{sec:ONPG_Matrix_FA}). On the other hand, this policy parameterization makes the optimization problem \eqref{eq:hidden_bilinear_problem_og} more challenging to solve. Indeed, the minimax problem \eqref{eq:hidden_bilinear_problem_og} becomes  a nonconvex-nonconcave problem in $\theta$ and $\nu$, even with the tabular parameterization as we will show later in Lemma \ref{lemma:ncnc}.

\begin{remark}[Hidden bilinear problem]
	Note that Problem \eqref{eq:hidden_bilinear_problem_og}, which resembles a bilinear zero-sum  game, in fact falls into the class of {\it hidden bilinear} minimax problems discussed in \cite{vlatakis2019poincare} (or more generally the {\it hidden convex-concave} games  studied in \cite{flokas2021solving,mladenovic2021generalized}). It was shown in \cite{flokas2021solving}  that for general smooth functions of $g_\theta$ and $h_\nu$,  vanilla gradient descent-ascent exhibits a variety of behaviors antithetical to convergence to the  solutions. We here instead, show that for the specific  {\it softmax parameterization}, and for certain variants of the vanilla gradient-descent-ascent  method, the last-iterate convergence rate of the parameters  $\theta$ and $\nu$ can be established. 
\end{remark}
 
 
 
\paragraph{Natural PG \& Non-convergence pitfall.} Before proceeding further, we first introduce the {\it regularized} game: 
\begin{align}
\label{prob:entr_reg_matrix}
\min_{\theta\in\RR^n}\max_{\nu\in\RR^n}~~~\ f_\tau(\theta,\nu):= g_{\theta} ^{\top} Q h_{\nu} - \tau \mathcal{H}(g_{\theta}) + \tau \mathcal{H}(h_{\nu}), 
\end{align}
where the cost of both players is  regularized by the Shannon entropy of the policies, with $\tau>0$ being the regularization parameter, and $\mathcal{H}(\pi)=-\sum_{a\in\cA}\pi(a)\log(\pi(a))$ for $\pi$ on a simplex. 
The entropy regularization, which is commonly used in single-player RL, enjoys the benefits of both  encouraging exploration and accelerating convergence \citep{neu2017unified,mei2020global}. Our hope is also to exploit the benefits of entropy regularization in the multi-player setting. Indeed, the regularized cost traces its source in the game theory literature \citep{mckelvey1995quantal}, to model the imperfect knowledge of the cost matrix $Q$. 
{In the next lemma, we show that the problem in Equation \ref{prob:entr_reg_matrix} can be of the nonconvex-nonconcave type: 
\begin{lemma}
\label{lemma:ncnc}
The minimax problem \eqref{prob:entr_reg_matrix} is nonconvex in $\theta$ and nonconcave in $\nu$,  even if $p_{\theta}=\theta$ and $q_\nu=\nu$. 
\end{lemma}
Note that the nonconvexity in the parameters remains even when we regularize with the entropy of the policy, i.e., $\tau>0$.} 

Motivated by the successes of {\it natural policy gradient} \citep{kakade2002natural} and its variants, as PPO/TRPO \citep{schulman2015trust,schulman2017proximal}, in RL practice, we consider the natural PG  descent-ascent update for \eqref{prob:entr_reg_matrix}, which is given by 
\#
{\theta}_{t+1} &={\theta}_{t}-\eta \cdot F^{\dagger}_\theta(\theta_t)\cdot\frac{\partial f_\tau(\theta_t,\nu_t)}{\partial\theta} = (1 - \eta \tau) \theta_t - \eta Q h_{{\nu}_t}+\eta\tau\bigg(\log \sum_{a' \in\cA} e^{{\theta_t}(a')}-1\bigg),\label{equ:vanilla_NPG_1}\\ 
{\nu}_{t+1} &={\nu}_{t}+\eta \cdot F^{\dagger}_\nu(\nu_t)\cdot\frac{\partial f_\tau(\theta_t,\nu_t)}{\partial\nu}= (1 - \eta \tau) \nu_t + \eta Q^{\top} g_{{\theta}_t}-\eta\tau\bigg(\log \sum_{b' \in\cB} e^{{\nu_t}(b')}-1\bigg), \label{equ:vanilla_NPG_2}
\#
where $F_\theta(\theta) = \EE_{a\sim g_\theta}[(\nabla_\theta\log g_\theta(a))(\nabla_\theta\log g_\theta(a))^\top]$ and $F_\nu(\nu)=\EE_{b\sim h_\nu}[(\nabla_\nu\log h_\nu(b))(\nabla_\nu\log h_\nu(b))^\top]$ are the Fisher information matrices, $M^\dagger$ denotes the pseudo-inverse of the matrix $M$, and $\eta>0$ is the stepsize. The  derivations for natural policy gradient can be found in \S\ref{sec:vanilla_NPG_deriv} for completeness. 

Unfortunately, the vanilla NPG update  \eqref{equ:vanilla_NPG_1}-\eqref{equ:vanilla_NPG_2} may fail to converge in the parameter space for \emph{any} stepsize $\eta>0$. The key reason for the failure is that the mappings represented in \eqref{equ:vanilla_NPG_1}-\eqref{equ:vanilla_NPG_2} may not have a fixed point for a general $Q$ and $\tau$ (which could be the only limit point for this dynamics). {In fact, this issue persists even when the regularization parameter $\tau = 0$}. We formalize this  pitfall in the following lemma, with its proof deferred to appendix.  

\begin{lemma}[Pitfall of vanilla NPG]\label{lemma:pitfall_NPG}
	There exists a game \eqref{prob:entr_reg_matrix} with $\tau \geq 0$ and $|\cB|=1$, such that the updates \eqref{equ:vanilla_NPG_1}-\eqref{equ:vanilla_NPG_2} 
	do not converge for any $\eta>0$. 
\end{lemma}
Remarkably, we emphasize that our Lemma \ref{lemma:pitfall_NPG}, by construction, also even applies to the single-agent setting, with a regularized cost  and the NPG update, as studied recently in \cite{cen2021fastMDP,lan2021policy,zhan2021policy}. These works only focus on the convergence in the policy space, which does not imply the desired convergence in the policy parameter space. The later becomes especially relevant in the function approximation setting, as we will study later. Finally, we remark that, the non-convergence here also should not be confused with the {\it last-iterate non-convergence} of no-regret learning algorithms for solving bilinear zero-sum games \citep{daskalakis2018limit,bailey2018multiplicative}, as our example is essentially a single-agent case. We summarize the importance and motivation of establishing parameter convergence as follows.

\paragraph{Importance of Parameter Convergence:} \hfill \\
\vspace{-7pt}

\noindent\textbf{\textit{Numerical instability:}} 
Prior works were only able to  show the convergence of {\it values} and/or {\it policies}, and the convergence behavior of policy parameters was {\it unclear} (or overlooked). Arguably, having (last-iterate) parameter convergence is the {\it strongest}  type of convergence among the three. In practice, having parameters blow-up to infinity can cause numerical issues. For example, once the size of the parameter crosses a threshold (say $2^{64}$ for an integer in a 64-bit operating system), there would be overflow issues, and the stored parameter would be void, and NaN (not a number) would be returned by the program. This blow-up would then cause trouble in recovering the policy, or approximating the policy with arbitrary accuracy. In order to circumvent this issue, a common practice in Neural Network training is to do {\it Clipping/Projection}. In fact, ensuring the {\it stability}  of the model is very important in deep learning. Specifically, 
consider  $\max_{\theta \in \mathbb{R}^n} q^{\top}g_{\theta} + \tau \mathcal{H}(g_\theta)$, where we know that under the NPG updates, $g_\theta \rightarrow g^\star$ while $\theta$ could blow up to infinity. 
One could clip the parameter $\theta$ to some large constant $\theta_{\max}$, i.e., solving $\max_{ \| \theta \|_{\infty} \leq \theta_{\max}} q^{\top}g_{\theta} + \tau \mathcal{H}(g_\theta)$ instead. For concreteness, let $n = 2$, $\tau = 1$, $\theta_{\max} = 80$ and $q = [-2, -3]$. The optimal solution is then given by $g^{\star}_i \propto \exp(q_i)$ for $i = 1,2$. On running the vanilla NPG algorithm, since we do weight clipping, the algorithm converges  to $\theta = [\theta_{\max}, \theta_{\max}]$  corresponding to the distribution $[1/2, 1/2] \neq g^{\star}$. Meanwhile, the modified NPG we propose converges to $\theta = [-2, -3]$ (see Theorem \ref{thm:param_con}) which exactly corresponds to the optimal solution $g^{\star}$. Hence, in practice where the norm of the (neural network) parameters is bounded, one might not obtain policy convergence using vanilla NPG as desired, while  our proposed algorithm works. \\
\vspace{-8pt}

\noindent\textbf{\textit{Nonconvex-nonconcave minimax optimization:}} The second reason   comes from a minimax optimization perspective of solving \eqref{eq:hidden_bilinear_problem_og}. We view optimization over  the parameter space as an interesting {\it nonconvex-nonconcave} minimax optimization problem with a {\it hidden structure} (See Lemma \ref{lemma:ncnc}). To  the best of our knowledge, this is the first paper to provide a {\it symmetric} discrete-time algorithm to solve certain nonconvex-nonconcave  problem (and more generally non-monotone variational inequalities) with last-iterate convergence rates, even including the specialized settings (like the ones with    Polyak-\L ojasiewicz condition \citep{nouiehed2019solving,yang2020global}).  \\ 
\vspace{-8pt}

\noindent\textbf{\textit{Function approximation (FA): }} Parameter convergence becomes crucial in FA settings (used in practice).  Here, the policy lies in a high-dimensional space (or even an infinite-dimensional space if the actions are {\it continuous}), which we simply do not have access to and/or cannot operate on. The way practitioners run PG methods is to just operate on the  low-dimensional  policy parameter space. Thus, parameter convergence is necessary  to design meaningful stopping criteria for optimization algorithms. If parameters explode to infinity, we cannot decide on how close we are to convergence, and the numerical issue mentioned before would cause trouble in recovering the policy.






%% file: Tabular_Matrix.tex

\section{{Warm-up: (Optimistic) NPG for Matrix Games}}\label{sec:ONPG_Static_tabular}





{

{To address the pitfall above about parameter convergence, we introduce two variants of the vanilla NPG \eqref{equ:vanilla_NPG_1}-\eqref{equ:vanilla_NPG_2}, and show their convergence for solving matrix games.} 

\subsection{NPG for Matrix Games}
We first introduce the {following variant of the vanilla NPG update}: 
\#
{\theta}_{t+1} &=(1 - \eta \tau) \theta_t - \eta Q h_{{\nu}_t},\label{equ:new_NPG_1}\\ 
{\nu}_{t+1} &= (1 - \eta \tau) \nu_t + \eta Q^{\top} g_{{\theta}_t},\label{equ:new_NPG_2}
\#
{where we removed the last term in \eqref{equ:vanilla_NPG_1}-\eqref{equ:vanilla_NPG_2}, respectively.} 
Note that these updates correspond to the popular Multiplicative Weights Update (MWU) for the regularized game in policy space   (we succinctly represent $g_{\theta_t}$ and $h_{\nu_t}$ as $g_t$ and  $h_t$, respectively), i.e., 
\begin{align}
\label{eq:policy_evol_MWU}
g_{t+1}(a)  \propto {g}_t(a)^{(1 - \eta \tau)} e^{-\eta [Q h_t]_a }, \nonumber \\
h_{t+1}(b)  \propto {h}_t(b)^{(1 - \eta \tau)} e^{\eta [Q^{\top} g_t]_b }. 
\end{align}

First, we provide a convergence result for the updates in Equations \eqref{equ:new_NPG_1}-\eqref{equ:new_NPG_2}, the non-optimistic version, both in terms of policy as well as parameters. In order to do so, we need to first show that the iterates of the regularized MWU in Equations \eqref{equ:new_NPG_1}-\eqref{equ:new_NPG_2} ensure  that the policies stay bounded away from the boundary of the simplex. We show this in the following lemma: 
\begin{lemma}
\label{lemma:MWU_c_main}
The polic{ies} corresponding to the iterates of regularized MWU in Equations \eqref{equ:new_NPG_1}-\eqref{equ:new_NPG_2} with stepsize $\eta < 1/\tau$ stay within a set $\Delta' \subset \Delta$ which is bounded away from the boundary of the simplex, i.e., {$\forall x \in \Delta'$ with $x=(x_1,\cdots,x_n)^\top$, and for all $i\in[n]$,} $x_i \geq \delta > 0$ for some $\delta$. 
\end{lemma} 

Since the iterates {of the policies} lie within $\Delta'$, {a closed and bounded set, and the regularized cost   is continuously differentiable with respect to the policies,} we let $L$ denote the {smoothness constant} of the {regularized cost}  in the policy space, i.e.,
\begin{align}
\label{eq:MWU_L_m}
\| (Mz_1 + \tau \nabla \cH(z_1)) - (Mz_2 + \tau \nabla \cH(z_2)) \| \leq L \|z_1 - z_2 \| , \qquad \forall z_1, z_2 \in \cZ',
\end{align}
where $z = [g; h]$, $\cZ' \in \Delta' \times \Delta'$ and with a slight abuse of notation, we define $\nabla \cH(z) = [\nabla_g \cH(g); \nabla_h \cH(h)]$, and also, we  define the matrix $M = \begin{pmatrix}
0 & Q \\ -Q^{\top} & 0
\end{pmatrix}$.

Finally, the entropy regularized optimization problem in the policy space can be formulated as:
\begin{align}
\label{prob:entr_reg_matrix_policy}
\min_{g \in \Delta(\cA)}\max_{h \in \Delta(\cB)} = g ^{\top} Q h - \tau \mathcal{H}(g) + \tau \mathcal{H}(h). 
\end{align}

Now, we use this to derive the policy and parameter convergence of the MWU updates in Equations \eqref{equ:new_NPG_1}-\eqref{equ:new_NPG_2}
 in the following theorem.  
 \begin{theorem}
 \label{thm:con_MWU}
 The solution ($g^\star, h^\star$) to the problem \eqref{prob:entr_reg_matrix_policy} is unique. Furthermore, let $\theta^\star = \frac{-Qh^\star}{\tau}$ and $\nu^\star = \frac{Q^{\top}g^\star}{\tau}$. Then on running Equations \eqref{equ:new_NPG_1}-\eqref{equ:new_NPG_2} with the stepsize satisfying $0 < \eta \leq \tau/L^2$, 
we have:
\begin{align}
\KL (z^* \| z_{t+1}) \leq \left( 1 - \frac{\eta \tau}{2} \right) \KL (z^* \| z_{t}), 
\end{align}
and
\begin{align}
\| \theta_{t+1} - \theta^\star \|^2 + \| \nu_{t+1} - \nu^\star \|^2  \leq (1 - \eta \tau/4)^t \Bigg(\| \theta_{0} - \theta^\star \|^2 + \| \nu_{0} - \nu^\star \|^2 +  \frac{4C}{\eta\tau}\Bigg), 
\end{align}
where $C =  \left(1 + \frac{1}{\eta\tau} (1 - \eta \tau)^2 \right)4 \eta^2 \|Q \|_\infty^2\KL(z^\star \| z_0)$. Here $z_0 = (g_{\theta_0}, h_{\nu_0})$ and $z^\star = (g^\star, h^\star)$.
 \end{theorem}
 
 To the best of our knowledge, this is the first policy  convergence guarantee for NPG (without optimism) for regularized matrix games. As the theorem above shows, parameter convergence can be established as well. In other words, though the vanilla NPG {descent-ascent} diverges for Problem \eqref{prob:entr_reg_matrix}, the variant we propose (Equations \eqref{equ:new_NPG_1}-\eqref{equ:new_NPG_2}) {\it implicitly regularizes} the parameter iterates to converge to a particular solution, in last-iterate. Recall from Lemma \ref{lemma:ncnc} that Problem \eqref{prob:entr_reg_matrix} is  a nonconvex-nonconcave minimax problem
 and there have been no convergence guarantees{, to the best of our knowledge,} of any symmetric and simultaneous-update algorithms in general (see \cite{yang2020global, rafique2018non, lin2019gradient} and references therein {as examples} for some structured nonconvex-nonconcave problems).  From Theorem \ref{thm:con_MWU}, the rate we can hope to achieve with NPG is $\mathcal{O}(\kappa^2 \log (1/\epsilon))$ (number of steps to reach a point $\epsilon$ close to the solution), where $\kappa$ is the condition number $L/\tau$ of the problem. In the next subsection, we see how adding optimism (\cite{rakhlin2013optimization}) will improve this rate of convergence. 
}


\subsection{Optimistic NPG (ONPG) for Matrix Games}

{In this subsection, we study the variant of the NPG updates in Equations \eqref{equ:new_NPG_1}-\eqref{equ:new_NPG_2} along with optimism} 
\citep{popov1980modification, rakhlin2013optimization,daskalakis2018limit,daskalakis2019last, mokhtari2020convergence}. In particular, we introduce the intermediate iterates $(\bar\theta_t,\bar\nu_t)$, and the following optimistic variant of the NPG (here ($\bar{\theta}_0, \bar{\nu}_0$) is initialized as ($\theta_0, \nu_0$)):
\$
&\bar{\theta}_{t+1} = (1 - \eta \tau) \theta_t - \eta Q h_{\bar{\nu}_t},~~{\theta}_{t+1} = (1 - \eta \tau) {\theta}_t - \eta Q h_{\bar{\nu}_{t+1}} \\
&\bar{\nu}_{t+1} = (1 - \eta \tau) \nu_t + \eta Q^{\top} g_{\bar{\theta}_t},~~
{\nu}_{t+1} = (1 - \eta \tau) {\nu}_t + \eta Q^{\top} g_{\bar{\theta}_{t+1}}.
\$
This optimistic update is motivated by the success of optimistic gradient methods in saddle-point problems recently analyzed in several papers including \cite{hsieh2019convergence, mokhtari2020unified, mokhtari2020convergence}.
Note that our algorithm is symmetric in $\theta$ and $\nu$ updates and both players update simultaneously.  
The update rules are also tabulated in Algorithm \ref{alg:optimistic_npg}. 

\begin{remark}[Connections to the literature]\label{remark:connections}
	Note that the natural PG update rule in Equations  \eqref{equ:new_NPG_1}-\eqref{equ:new_NPG_2} has a close relationship to the multiplicative weight update rule \citep{freund1997decision,arora2012multiplicative} in the policy space $(g_\theta,h_\nu)$, see Section C.3 in   \cite{agarwal2019optimality} for a detailed discussion. Similarly, the optimistic NPG update in Algorithm \ref{alg:optimistic_npg} also relates to the optimistic MWU update \citep{daskalakis2019last}. In fact, recent works \cite{wei2020linear,cen2021fast} have shown the last-iterate {\it policy convergence}  of OMWU for zero-sum matrix games (with \cite{wei2020linear} relying on the uniqueness assumption of the NE). Our goal, in contrast, is to study the (last-iterate) convergence behavior of the actual {\it policy parameters}  $(\theta,\nu)$, and go beyond the tabular zero-sum setting. 
\end{remark}

{ Inspired by the results which show that adding optimistic updates improves convergence rates \citep{rakhlin2013optimization}, we next explore our modified NPG updates with optimism, and show that the convergence rate does in fact improves 
 (in line with the comparison of the  performance of GDA  and optimistic GDA,  see e.g., \cite{fallah2020optimal}). }


In the next theorem, we show that our  Optimistic NPG Algorithm (Algorithm \ref{alg:optimistic_npg}) in fact converges linearly in last-iterate to a unique point in the set of NE in the parameter space, at a faster rate\footnote{ {Note that we say that the optimistic version achieves a faster rate, since the range of stepsizes which permits convergence is much larger for the optimistic variant. This can be noted from the fact that $L$ in equation \eqref{eq:MWU_L_m} will be larger than $\|Q\|_\infty + \tau$.}} than {the non-optimistic counterpart}  
in Equations  \eqref{equ:new_NPG_1}-\eqref{equ:new_NPG_2}.  The results are formally stated below. 
 \begin{theorem}
 \label{thm:param_con}
Let $\theta^\star = \frac{-Qh^\star}{\tau}$ and $\nu^\star = \frac{Q^{\top}g^\star}{\tau}$, {where $g^\star$ and $h^\star$ are the solutions to the regularized game \eqref{prob:entr_reg_matrix}}. Then on running Algorithm \ref{alg:optimistic_npg} with the stepsize satisfying $0 < \eta \leq \min \left\{ \frac{1}{2\tau + 2 \| Q \|_{\infty}}, \frac{1}{4 \| Q \|_{\infty}} \right\}$,
we have:
\begin{align}
\label{eq:lin_conv_ONPG_tab}
\| \theta_{t+1} - \theta^\star \|^2 + \| \nu_{t+1} - \nu^\star \|^2 &\leq (1 - \eta \tau/2)^t V_0,
\end{align}
where $V_0 = \| \theta_{0} - \theta^\star \|^2 +  \| \nu_{0} - \nu^\star \|^2 + \frac{2C}{\eta\tau}$ and $C = \left(1 + \frac{1}{\eta\tau} (1 - \eta \tau)^2 \right)4 \eta^2 \|Q \|_\infty^2\KL(z^\star \| z_0)$. Here $z_0 = (g_{\theta_0}, h_{\nu_0})$ and $z^\star = (g^\star, h^\star)$.
 \end{theorem}
 
This result shows that the specific hidden bilinear minimax problem we are dealing with does not fall into the spurious categories  discussed in \cite{vlatakis2019poincare}, if we resort to the (optimistic) natural PG update. Note that achieving parameter convergence is a non-trivial task since we are dealing with a nonconvex-nonconcave minimax problem (see Lemma \ref{lemma:ncnc})
The proof relies on the specific structure of the softmax policy parametrization and the construction of a novel Lyapunov function (see \S \ref{sec:ONPG_Static_tabular_app} for more details). 
Next, we show how the optimistic  NPG algorithm solves the original matrix game without regularization.


\begin{corollary}
\label{cor:tab_matrix_connect_unreg}
If we run the {non-optimistic variant of NPG in Equations \eqref{equ:new_NPG_1}-\eqref{equ:new_NPG_2} or the optimistic version} in Algorithm \ref{alg:optimistic_npg} for time $T = \cO \left( \frac{\log n}{\eta \epsilon} \log \left( \frac{1}{\epsilon} \right) \right)$
and set $\tau = \epsilon/(8\log n)$, we have that the output $(\theta_T, \nu_T)$ is an $\epsilon$-NE of the original unregularized Problem \eqref{eq:hidden_bilinear_problem_og}.
\end{corollary}

%
%


We extend these results to simple function approximation settings in the next section.

%% file: FA_Matrix.tex

\section{Matrix Games with  Function Approximation}
\label{sec:ONPG_Matrix_FA}

%
%
 
 
To handle games with excessively  large action spaces,  we resort to policy parameterization with 
function approximation. In particular, consider the following problem: 
\begin{align}
\label{eq:objective_func_approx}
\min_{\theta \in \mathbb{R}^d} \ \max_{\nu \in  \mathbb{R}^d} ~~~ \ g_{\theta} ^{\top} Q h_{\nu},
\end{align}
where $g_{\theta}$ and $h_{\nu}$ are both parameterized in a softmax way as in \eqref{equ:policy_para_matrix}, with the linear function class  $p_\theta(a)=\phi_a^{\top} \theta$ (also called {\it log-linear} policy in \cite{agarwal2019optimality}),  
where $\phi_a \in \mathbb{R}^d$ is a low-dimensional feature representation of the action (see \cite{branavan2009reinforcement, agarwal2019optimality}) (Note that usually $d <n$). We define: 
\begin{align}
\label{eq:Phi_def}
\Phi = [\phi_1, \phi_2, \cdots , \phi_n]  ~\in \mathbb{R}^{d \times n}.
\end{align}

\begin{assumption}
\label{ass:full_rank}
$\Phi$ is a full rank matrix. In particular, assume that $\Phi = [M \ | \ 0]$, where $M \in \mathbb{R}^{d\times d}$ is an invertible $d \times d$ square matrix.
\end{assumption}

Note that {the full-rankness of $\Phi$ is a standard  assumption (see Assumption 6.2 in \cite{agarwal2019optimality}). It essentially requires the features to be the bases of some low-dimensional space. 
Furthermore, the results also extend to the case where the matrix $\Phi$ is of the from $[M \ | \ c \mathbbm{1}]$ where $\mathbbm{1}$ is the matrix of all 1s of appropriate dimension,  and $c$ is any constant.} 
{This  particular structure of the feature matrix, though being restrictive,   ensures that the constraint set of policies is convex, as shown next, otherwise the minimax theorem of $\min\max=\max\min$  might not hold, i.e., the Nash equilibrium for the parameterized game does not exist. Moreover, the assumption is also not as restrictive as it seems. For example, in applications of self-driving car and robotics, only a subset of actions (steering angles) is essential in controlling the agent, with other actions being insignificant/redundant. Our feature matrix encodes patterns like these. Moreover, as a first step studying  policy optimization in multi-agent learning with function approximation, we start with this simpler setting. Extending the ideas to more general FA settings is an interesting direction worth exploring.}


 \begin{remark}
 Note that the results in this section are presented for the case where the feature matrix is identical for both players purely for simplification of notation. The results continue to hold for the case with asymmetric features  as well (as long as the feature matrix also satisfies Assumption \ref{ass:full_rank}. 
 \end{remark}

As motivated in the previous section, we study the following regularized problem in order to solve \eqref{eq:objective_func_approx} efficiently:
\begin{align}
\label{eq:func_approx_reg}
\min_{\theta \in \mathbb{R}^d} \ \max_{\nu \in  \mathbb{R}^d}  \ ~~~ g_{\theta} ^{\top} Q h_{\nu} - \tau \mathcal{H}(g_{\theta}) + \tau \mathcal{H}(h_{\nu}), 
\end{align}
where $\mathcal{H}$ denotes the entropy function and $\tau>0$ is the regularization parameter. Note that this problem can still be nonconvex-nonconcave in general, given the example in \S\ref{sec:ONPG_Static_tabular} as a special case.

{We define the solution to this problem next, the Nash equilibrium in the parameterized policy classes, i.e., in-class NE. 

\begin{definition}[$\epsilon$-in-class Nash equilibrium]\label{def:para_eps_NE_FA_matrix}
	The policy parameter $(\tilde{\theta}, \tilde{\nu})$ is an $\epsilon$-{\it Nash equilibrium}  of the matrix game with function approximation (or  $\epsilon$-{\it in-class} NE), if it satisfies that for all  $i\in[N]$, 
\begin{align}
g_{\tilde{\theta}}^{\top}Q h_{\nu} - \epsilon \ \leq \ g_{\tilde{\theta}}^{\top}Q h_{\tilde{\nu}} \ \leq \ g_{{\theta}}^{\top}Q h_{\tilde{\nu}} + \epsilon , \qquad \forall \theta,~\nu~ \in \mathbb{R}^d.
\end{align}
Furthermore, when $\epsilon = 0$, we refer to it as the \textit{in-class Nash Equilibrium}.
\end{definition}

}

\subsection{Equivalent problem characterization}
 
In this subsection, we study the regularized problem \eqref{eq:func_approx_reg} under a log-linear parametrization and find an equivalent problem in the tabular case. 

First, in the following lemma, we characterize the set of distributions  covered by this parametrization, and study the equivalent problem in the space of probability vectors.

\begin{lemma}
\label{lemma:char_func_approx}
Under Assumption \ref{ass:full_rank}, the log-linear parametrization in Equation \eqref{equ:policy_para_matrix} covers all distributions in the following \textit{convex} set:
\begin{align}
\label{eq:Delta_tilde_def}
\tilde{\Delta} = \{ \mu: \mu \in \Delta, \mu_{d+1} = \mu_{d+2} = \cdots = \mu_n  \},
\end{align} 
and Problem \eqref{eq:func_approx_reg} is equivalent to
\begin{align}
\min_{g_{\theta} \in \tilde{\Delta}} \ \max_{h_{\nu}  \in \tilde{\Delta} } ~~~\ g_{\theta} ^{\top} Q h_{\nu} - \tau \mathcal{H}(g_{\theta}) + \tau \mathcal{H}(h_{\nu}). 
\label{eq:prob_contraint}
\end{align}
\end{lemma}
Lemma \ref{lemma:char_func_approx} characterizes the set of distributions which can be represented by log-linear parametrization. Therefore, when we try to solve the matrix game with such function approximation, the best we can hope for is to find an equilibrium within the set $\tilde{\Delta}$.

Next, we characterize the Nash equilibrium of the regularized Problem \eqref{eq:func_approx_reg} in the function approximation setting, and show its equivalence to another problem in the tabular softmax setting.  

\begin{theorem}
\label{thm:Nash_char_func_approx}
{An in-class Nash equilibrium $(\theta^\star, \nu^\star)$ of Problem \eqref{eq:func_approx_reg} under the function approximation setting 
{exists}, and any such in-class NE satisfies:}
\begin{align}
g_{\theta^\star}(a) = \frac{e^{-\frac{[\Psi^{\top} Q \Psi h_{\nu^\star}]_a}{\tau}}}{\sum_{a'} e^{-\frac{[\Psi^{\top} Q \Psi h_{\nu^\star}]_{a'}}{\tau}}}{,}~~~ 
h_{\nu^\star}(a) =  \frac{e^{\frac{[\Psi^{\top} Q^\top \Psi g_{\theta^\star}]_a}{\tau}}}{\sum_{a'} e^{\frac{[\Psi^{\top} Q^\top \Psi g_{\theta^\star}]_{a'}}{\tau}}}, \nonumber
\end{align}
where $\Psi \in \mathbb{R}^{n \times n}$ is defined as:
\begin{align}
\label{eq:psi_def}
\Psi = 
\begin{pmatrix}
I_d & \zero \\ \zero & \begin{matrix} \frac{1}{n-d} &\cdots & \frac{1}{n-d} \\ \cdots &\cdots & \cdots \\ \frac{1}{n-d} &\cdots & \frac{1}{n-d} \end{matrix}
\end{pmatrix}.
\end{align}
Furthermore, Problem \eqref{eq:func_approx_reg} is equivalent to\footnote{By equivalent, we mean that the two problems have the same value at the NE. We will also show the relationship between the solutions in Proposition \ref{eq:lemma:algo_lin_func_approx}} 
\begin{align}
\min_{g_{\theta} \in {\Delta}} \ \max_{h_{\nu}  \in {\Delta} } ~~~\ g_{\theta} ^{\top} \Psi^{\top} Q \Psi h_{\nu} - \tau \mathcal{H}(g_{\theta}) + \tau \mathcal{H}(h_{\nu}).
\end{align}
\end{theorem}

Note that the matrix $\Psi$ defined in Theorem \ref{thm:Nash_char_func_approx} is the invariant matrix for the set $\tilde{\Delta}$, i.e., $\Psi \mu = \mu$, $\forall \mu \in \tilde{\Delta}$.


\subsection{Optimistic NPG algorithm}


{From Section \ref{sec:ONPG_Static_tabular}, we see that the optimistic version of NPG leads to faster convergence (of both policy and parameters). This motivates us to focus on the optimistic version of the  methods. In this subsection, (and the ones that follow), we focus on optimistic methods instead of their non-optimistic counterparts. Note that similar to Section \ref{sec:ONPG_Static_tabular}, we can derive convergence rates for the non-optimistic versions as well, which would be slower than the corresponding optimistic versions.
}

As Theorem \ref{thm:Nash_char_func_approx} characterizes the solution to the function approximation setting to that of the tabular softmax setting, we modify the algorithm for function approximation setting as follows:
\begin{align}
\bar{\theta}_{t+1} = (1 - \eta \tau) \theta_t - \eta  [(M^\top)^{-1} \ |\ 0]  \tilde{P} Q h_{\bar{\nu}_t},\qquad 
{\theta}_{t+1} = (1 - \eta \tau) {\theta}_t - \eta  [(M^\top)^{-1} \ |\ 0] \tilde{P} Q h_{\bar{\nu}_{t+1}},
\end{align}
and a similar update for $\nu$ to reach the solution of the regularized problem under a log-linear parametrization. Here, the matrix $\tilde{P}$ is defined as
\begin{align*}
\tilde{P} = 
\begin{pmatrix}
I_d & \begin{matrix} \frac{-1}{n-d} &\cdots & \frac{-1}{n-d} \\ \cdots &\cdots & \cdots \\ \frac{-1}{n-d} &\cdots & \frac{-1}{n-d} \end{matrix} \\ \zero & \zero
\end{pmatrix}.
\end{align*}
The additional term involving the inverse of the feature matrix arises due to the nature of the log-linear function approximation. We make this formal in the following proposition. 

\begin{proposition}
\label{eq:lemma:algo_lin_func_approx}
Consider Algorithm \ref{alg:optimistic_npg_func} used to solve Problem \eqref{eq:func_approx_reg} under the log-linear parametrization in Equation \eqref{equ:policy_para_matrix} under Assumption \ref{ass:full_rank}. Then the iterates of Algorithm \ref{alg:optimistic_npg_func} have the same guarantees provided in Theorem \ref{thm:param_con}. Here, the NE parameter value to which the algorithm converges to is given by:
\begin{align}
\theta^\star = \frac{- [(M^\top)^{-1} \ |\ 0]  \tilde{P} Q h_{{\nu}^\star}}{\tau}, \ \ \ \ \nu^\star = \frac{[(M^\top)^{-1} \ | \ 0] \tilde{P} Q^\top g_{{\theta}^\star}}{\tau}. \nonumber 
\end{align}
\end{proposition}


Next, we show how the optimistic  NPG algorithm solves the original matrix game without regularization.

\begin{corollary}
If we run Algorithm \ref{alg:optimistic_npg_func} for time $T = \cO \left( \frac{\log n}{\eta \epsilon} \log \left(\frac{1}{\epsilon} \right) \right)$
and set the regularization parameter $\tau =  \epsilon/(8\log n)$, we have that the output $(\theta_T, \nu_T)$ is an $\epsilon$-in-class NE (Definition \ref{def:para_eps_NE_FA_matrix}) 
of the unregularized Problem \eqref{eq:objective_func_approx}. 
\end{corollary}

%% file: General_Monotone.tex
\section{{Multi}-player Monotone Games}
\label{sec:Gen_Monotone}


 \paragraph{Monotone games.} Consider a  multi-player continuous game over simplexes, which strictly generalizes the zero-sum matrix game in \S\ref{sec:prelim}. The game is characterized by  $(\cN,\{\cA_i\}_{i\in[N]},\{f_i\}_{i\in[N]})$, where $\cN=[N]$ is the set of players. Without loss of generality, we assume $|\cA_i|=n$ for all $i\in[N]$. For notational convenience, let $\Delta$ denote the simplex over $\cA_i$, and $z:=(g_1,g_2,\cdots,g_N)\in\Delta^N$ denote the strategy profile of all $N$ players, with each $g_i\in\Delta$. We define the {\it pseudo-gradient} operator $F:\Delta^N\to \RR^{nN}$ as 
 $
 F(z) := [\nabla_{g_{i}} f_i (g_i, g_{-i})]_{i=1}^N.
 $
 To make the $N$-player game tractable, we make the following  standard assumptions on $F$ \citep{rosen1965existence,nemirovski2004prox,facchinei2007finite}. 
 
 \begin{assumption} [Monotonicity \& Smoothness]
\label{ass:monotone_smooth_F}
The operator $F$ is monotone and smooth, i.e., $\forall z, z' \in \Delta^N$ 
{\small\begin{align*}
\langle F(z) - F(z'), z - z' \rangle &\geq 0,  \qquad~~\| F(z) - F(z') \| \leq L \cdot\|z - z'\|,
\end{align*}}
where $L>0$ is the Lipschitz constant of the operator $F$. 
\end{assumption}

The goal is to find the NE, given by strategy $z^\star$ such that 
$
f_i(z_i^\star,z_{-i}^\star)\leq f_i(z_i,z_{-i}^\star),\ \forall~z_i\in\Delta,~i\in[N]. 
$
Under Assumption \ref{ass:monotone_smooth_F}, it is known that the NE exists \citep{rosen1965existence}. 

\paragraph{Policy parameterization \& regularized game.} To develop policy gradient methods, we parameterize each policy $g_i\in\Delta$ by $g_{\theta_i}$ in the softmax form as before, i.e., for any $a_i\in\cA_i$, $g_{\theta_i}(a_i)={e^{p_{\theta_i}(a_i)}}\cdot\big({\sum_{a_i' \in\cA_i} e^{p_{\theta_i}(a_i')}}\big)^{-1}$, where $\theta_i\in\RR^d$, and we consider both the tabular case with $p_{\theta_i}=\theta_i$ and the linear function approximation case with  $p_{\theta_i}(a_i)=\phi_{a_i}^{\top} \theta_i$. This parameterization leads to the following set of optimization problems: 
\#
\label{eq:N-payer_monotone_game_unreg}
\min_{\theta_i \in \mathbb{R}^n}~~ f_i(g_{\theta_i}, g_{\theta_{-i}}),\qquad \forall~i\in[N],
\#
whose solution $(\theta_1^\star,\theta_2^\star,\cdots,\theta_N^\star)$, if exists, corresponds to the Nash equilibrium under this parameterization. Note that \eqref{eq:N-payer_monotone_game_unreg} can also be viewed  as a nonconvex game (Lemma \ref{lemma:ncnc} is a special case) with a ``hidden'' {\it monotone variational inequality}  structure, which generalizes the class of hidden convex-concave problems discussed in \cite{flokas2021solving,mladenovic2021generalized}. 


Motivated by \S\ref{sec:prelim}, we also consider the regularized game in hope of stronger convergence guarantees for solving \eqref{eq:N-payer_monotone_game_unreg}. Specifically, the players solve
\#
\label{eq:N-payer_monotone_game_reg}
\min_{\theta_i \in \mathbb{R}^n}~~ f_i(g_{\theta_i}, g_{\theta_{-i}})- \tau \mathcal{H}(g_{\theta_i}),\qquad \forall~i\in[N],
\# 
where $\tau>0$ and $\mathcal{H}$ is the entropy function. 
With a small enough $\tau$, the solution to \eqref{eq:N-payer_monotone_game_reg} approximates that to \eqref{eq:N-payer_monotone_game_unreg}.

\subsection{Softmax parameterization}\label{sec:monotone_tabualr}

We first consider the tabular softmax parameterization  with $p_{\theta_i}=\theta_i\in\RR^{n}$ for all $i\in[N]$. In this case, the Nash equilibrium $\theta^\star = (\theta_1^\star, \theta_2^\star, \cdots, \theta_N^\star)$ of the   regularized monotone game   \eqref{eq:N-payer_monotone_game_reg} satisfies the following property. 


\begin{lemma}
\label{lemma:optimal_solution_N_player}
The NE of the game \eqref{eq:N-payer_monotone_game_reg}  exists. A  vector $\theta^\star = (\theta_1^\star, \theta_2^\star, \cdots, \theta_N^\star)$ is a NE of  \eqref{eq:N-payer_monotone_game_reg} if and only if:  $g_{\theta^\star_i}(a) \propto \exp \big( \frac{- [\nabla_{g_{\theta_i}} f_i(g_{\theta^\star_i}, g_{\theta^\star_{-i}})]_a}{\tau} \big)$ 
and the vector $(g_{\theta^\star_1}, g_{\theta^\star_2}, \cdots, g_{\theta^\star_N})$ is unique. We denote $g_i^\star := g_{\theta^\star_i}$.
\end{lemma}


Note that although the NE policy $(g_{\theta^\star_1}, g_{\theta^\star_2}, \cdots, g_{\theta^\star_N})$ is unique, the NE parameter $({\theta^\star_1}, {\theta^\star_2}, \cdots, {\theta^\star_N})$ is not necessarily the case. Motivated by \S\ref{sec:prelim} and \S\ref{sec:ONPG_Static_tabular}, we propose the following  update-rule for solving \eqref{eq:N-payer_monotone_game_reg}: 
$\forall~\text{players}~i \in [N]$, 
\$
\bar{\theta}_i^{t+1} &= (1 - \eta \tau)\theta_i^t -\eta \nabla_{g_{\theta_i}} f_i(g_{\bar{\theta}^t_i}, g_{\bar{\theta}^t_{-i}}),\qquad\quad 
\theta_i^{t+1} = (1 - \eta \tau)\theta_i^t -\eta \nabla_{g_{\theta_i}} f_i(g_{\bar{\theta}^{t+1}_i}, g_{\bar{\theta}^{t+1}_{-i}}).
\$
We refer to the update-rule as {\it optimistic NPG} (as summarized in Algorithm \ref{alg:ogda_monotone}), as it corresponds to the optimistic version of the ({specific instance of}) natural PG direction for the regularized objective \eqref{eq:N-payer_monotone_game_reg}.{ We choose this specific instance of}  NPG due to the pitfall discussed in \S\ref{sec:prelim}; and the optimistic update is meant to obtain fast last-iterate convergence. 
See \S\ref{sec:append:ONPG_update} for a  detailed derivation of the update rule. 




As shown in \S\ref{sec:prelim}, the problem \eqref{eq:N-payer_monotone_game_unreg} is nonconvex in the policy parameter space, and can be  challenging in general. Our strategy is to show that our algorithm solves the regularized problem \eqref{eq:N-payer_monotone_game_reg} fast, with last-iterate  parameter convergence (see Theorem \ref{thm:eg_og_conv_mono}), which, with small enough $\tau$, also solves the  nonconvex game \eqref{eq:N-payer_monotone_game_unreg} (see Corollary \ref{cor:small_tau_N_player}).


\begin{theorem}
\label{thm:eg_og_conv_mono}
Let $z^\star = (g_i^\star)_{i=1}^N$ be the unique Nash equilibrium given in Lemma \ref{lemma:optimal_solution_N_player}. Also, we denote $z_t = (g_{\theta^t_i})_{i=1}^N$. Then for  Algorithm \ref{alg:ogda_monotone} with stepsize  $0 < \eta <  \frac{1}{2(N+4)L + 2\tau}$, we have:
\begin{align}
\max \left\{ \KL(z^\star \| z_{t} ), \KL(z^\star \| \bar{z}_{t+1} ) \right\} &\leq (1 - \eta \tau)^t 2 \KL (z^\star \| z_0),\qquad\quad  
\| \theta_{t+1} - \theta^\star \|^2  \leq (1 - \eta \tau/2)^t V_0,
\end{align}
where $\theta^\star_i = \frac{-\nabla_{g_{\theta_i}} f_i(g_{i}^\star, g_{{-i}}^\star)}{\tau}$, $V_0 =\| \theta^{t} - \theta^\star \|^2 + \frac{2NC}{\eta\tau}$, and $C = 4\eta^2L^2\left(1 + \frac{1}{\eta\tau} (1 - \eta \tau)^2 \right) \KL(z^\star \| z_{0})$. 
\end{theorem}

The proof  follows by  first showing  convergence in the policy space, in which we are dealing with a strongly convex problem under convex constraints. We then use this result, along with a novel  Lyapunov function to demonstrate the convergence in the parameter space, in which it is a nonconvex problem. The proof technique might of independent interest, and might be generalized to showing convergence in other nonconvex games with a hidden monotonicity  structure. 



\begin{remark}
The proof for Theorem \ref{thm:eg_og_conv_mono} follows by first showing the convergence rate of the Proximal Point method, and then observing that Optimistic methods approximate this method and could potentially achieve the same convergence rates  (see \cite{mokhtari2020convergence} for a unified analysis). We provide a convergence analysis for the Proximal Point and Extragradient methods in \S \ref{app:sec_EG_PP_proofs}. 
\end{remark}

Now, we present the convergence of Algorithm \ref{alg:ogda_monotone} to an $\epsilon$-NE of the un-regularized problem \eqref{eq:N-payer_monotone_game_unreg}. 
\begin{corollary}
\label{cor:small_tau_N_player}
If we run Algorithm \ref{alg:ogda_monotone} for time $T = \mathcal{O}\left( \frac{N\log n}{\eta \epsilon} \log \left( \frac{1}{\epsilon} \right) \right)$
and set the regularization parameter $\tau =  \epsilon/(4N\log n)$, we have that ${\theta^\top} = [{\theta_1^\top}, {\theta_2^\top}, \cdots, {\theta_N^\top}]$, the iterate at time $T$, is an $\epsilon$-NE of Problem \eqref{eq:N-payer_monotone_game_unreg}.
\end{corollary}

We extend the results to certain function approximation settings in \S\ref{appendix:monotone_FA} in the Appendix.

%% file: Markov.tex
\section{Optimistic NPG for Markov Games}\label{sec:ONPG_Markov}



%
%

{We now generalize our results   to the sequential  decision-making case of Markov games.}  


\paragraph{Model.}  A two-player zero-sum Markov game is characterized by the tuple $(\cS,\cA,\cB,P,r,\gamma)$, where $\cS$ is the state space; $\cA,\cB$ are the action spaces of players $1$ and $2$, respectively; $P:\cS\times\cA\times\cB\to \Delta(\cS)$  denotes the transition probability of states; $r: \cS\times\cA\times \cB\to [0,1]$ denotes the bounded reward function 
 of player $1$ (thus $-r$ is the reward function of player $2$); and $\gamma\in[0,1)$ is the discount  factor. The goal of player $1$ (player $2$) is to minimize (maximize) the long-term accumulated  discounted reward.  

Specifically, at each time $t$, player $1$ (player $2$) has a {Markov}  stationary policy $g:\cS\to\Delta(\cA)$ ($h:\cS\to\Delta(\cB)$). 
The state makes a transition   from $s_t$ to $s_{t+1}$ following the probability distribution  $P(\cdot\given s_t,a_t,b_t)$, given 
$(a_t,b_t)$. 
As in the Markov decision process model, one can define the  \emph{state-value function} under a pair of joint policies $(g,h)$ as  
\$
V^{g,h}(s):=\EE_{a_t\sim g(\cdot\given s_t),b_t\sim h(\cdot\given s_t)}\bigg[\sum_{t\geq 0}\gamma^tr(s_t,a_t,b_t)\bigggiven s_0=s\bigg]. 
\$
Also, the \emph{state-action/Q-value function} 
under $(g,h)$ is  defined as 
\small 
\$
Q^{g,h}(s,a,b) &:=\EE_{a_t\sim g(\cdot\given s_t),b_t\sim h(\cdot\given s_t)}\bigg[\sum_{t\geq 0}\gamma^tr(s_t,a_t,b_t)\bigggiven s_0=s,a_0=a,b_0=b\bigg]. 
\$
\normalsize
Similar as the matrix game setting, a common  solution concept in Markov game is also  the  (Markov perfect) \emph{Nash equilibrium} policy pair $(g^\star,h^\star)$, which satisfies the following saddle-point inequality: 
\#\label{equ:def_NE_immed}
 V^{g,h^\star}(s)\leq V^{g^\star,h^\star}(s)\leq V^{g^\star,h}(s),\qquad \forall \ s\in\cS.  
 \#
 It follows from  \cite{shapley1953stochastic,filar2012competitive} that  there  exists a Nash equilibrium  $(g^\star,h^\star)\in\Delta(\cA)^{|\cS|}\times \Delta(\cB)^{|\cS|}$ for finite two-player  discounted zero-sum MGs. The state-value $V^\star:=V^{g^\star,h^\star}$ is referred to as the \emph{value of the game}. The corresponding Q-value function is denoted by $Q^\star$.

We focus on the softmax parameterization $g_\theta$ and $h_\nu$ of the policies $g$ and $h$, respectively.  

 \paragraph{Policy parameterization.} Following the matrix game setting, we also use the softmax parameterization of the policies. Specifically, for any $\theta,\nu\in\RR^d$, $(s,a,b)\in\cS\times\cA\times\cB$, 
  \begin{align}\label{equ:policy_para_markov}
g_{\theta}(a\given s) = \frac{e^{p_{\theta}(s,a)}}{ \sum\limits_{a' \in\cA} e^{p_{\theta}(s,a')} }, ~~~~ \qquad h_{\nu}(b\given s) = \frac{e^{q_{\nu}(s,b)}}{ \sum\limits_{b' \in\cB} e^{q_{\nu}(s,b')}}. 
\end{align}
Note that for any $s\in\cS$, $\sum_{a}g_\theta(a\given s)=\sum_{b}h_\nu(b\given s)=1$. We will consider both 

1) the tabular case with $p_{\theta}=\theta$ and $q_{\nu}=\nu$, where $\theta\in\RR^{|\cA|\times |\cS|}$ and $\nu\in\RR^{|\cB|\times|\cS|}$. We will use $\theta_s$ and $\nu_s$ to denote the parameters corresponding to state $s$, i.e., $\theta_s = [\theta_{s,1}, \theta_{s,2}, \cdots, \theta_{s, |A|}]$ and similarly $\nu_s$; 

2) the linear function approximation case with  $p_{\theta}(s,a)=\theta\cdot\phi(s,a)$ and $q_{\nu}(s,a)=\nu\cdot\phi(s,a)$ (See \S\ref{sec:appendix_func_approx_markov_results} for more details)\footnote{Once again, note that we use the same features for both players for notational convenience. Our results continue to hold when the two players have different features $\phi$.}. 

The parameterization thus leads to the following definition of the solution concept. 



\begin{definition}[$\epsilon$-in-class-NE for Markov games]\label{def:para_eps_NE}
	The policy parameter pair $(\tilde{\theta},\tilde{\nu})$ is an $\epsilon$-{\it in-class Nash equilibrium} 
	if it satisfies that for all  $s\in\cS$, 
\begin{align}
V^{\tilde{\theta}, \nu}(s) - \epsilon \geq V^{\tilde{\theta},\tilde{\nu}}(s) \geq V^{\theta,\tilde{\nu}}(s) + \epsilon, \qquad \forall \theta, \nu \in \mathbb{R}^d,
\end{align}
where $V^{\theta,\nu}(s)=V^{g_\theta,h_\nu}(s)$ denotes  the value of the parameterized policy pair $(g_\theta,h_\nu)$. Note that if we are in the tabular setting, we will have $d = n$, and the definition covers  that of the standard $\epsilon$-NE for Markov games. We also define the $\epsilon$-in-class NE $Q$-value  accordingly.
\end{definition}

Given that matrix games considered in \S\ref{sec:prelim} is a special case of the Markov games with $|\cS|=1$ and $\gamma=0$,  
Lemma \ref{lemma:ncnc} implies that finding the NE in Definition \ref{def:para_eps_NE} is  nonconvex-nonconcave in general, and can be challenging to solve. 
We show in the following lemma that, for the tabular and linear function approximation settings we consider, such a parameterized NE exists. 


\begin{lemma}[Existence of parameterized/in-class NE]\label{lemma:exist_para_NE}
	Under policy parameterization \eqref{equ:policy_para_markov} with tabular parameterization, the  in-class NE defined in Definition \ref{def:para_eps_NE} exists.
\end{lemma}

Motivated by \S\ref{sec:ONPG_Static_tabular} and \S\ref{sec:ONPG_Matrix_FA}, we consider the modified version of NPG with optimism to solve this problem.







 \paragraph{Optimistic NPG.} 
 Following \S\ref{sec:prelim}, we also consider the  {\it regularized} Markov games  \citep{geist2019theory,zhang2020model,cen2021fast}, in hope of  favorable convergence guarantees. Define the regularized value functions as  
 \small
 \begin{align}
 \label{eq:reg:V_func}
 V_{\tau}^{\theta, \nu} (s) &:= \mathbb{E} _{a_t\sim g_\theta(\cdot\given s_t),b_t\sim h_\nu(\cdot\given s_t)}\left[ \sum_{t=0}^{\infty} \gamma^t \big( r_t - \tau \log g_\theta(a_t | s_t) + \tau \log h_\nu(b_t|s_t) \big) \Biggiven s_0 = s \right],
 \end{align}
 \normalsize
where $r_t=r(s_t, a_t, b_t)$, $\tau < 1$, 
 and 
\#\label{eq:def_optimal_Q_func}
Q_{\tau}^{\theta, \nu} (s, a, b) := r(s, a, b) + \gamma \mathbb{E}_{s' \sim \mathbb{P}(\cdot|s, a, b)}[ V_{\tau}^{\theta, \nu} (s')].
 \#
 {We denote by $V_\tau^\star$ and $Q_\tau^\star$, the NE value and Q-functions respectively,  for the regularized Markov game (``regularized NE''), i.e., $V_\tau^\star = \min_\theta \max_\nu \ V_{\tau}^{\theta, \nu}$ and $Q_\tau^\star$ is the corresponding Q-function. Note that their existence follows along similar lines as Lemma \ref{lemma:exist_para_NE}.}
As a generalization of regularized matrix games in \S\ref{sec:prelim}, the non-convergence pitfall of vanilla NPG also occurs. 
We also define the following notation 
\begin{align}\label{equ:f_tau}
f_{\tau} \big(Q(s); g_\theta(\cdot\given  s),h_\nu(\cdot\given s)\big):= -\ g_\theta(\cdot\given s) ^{\top} Q(s) h_\nu(\cdot\given s)- \tau \mathcal{H}(g_\theta(\cdot\given s)) + \tau \mathcal{H}(h_\nu(\cdot\given s)).
\end{align}
\normalsize

\subsection{Convergence guarantees}

To stabilize the algorithm, we propose the update rule where  the parameters $(\theta, \nu)$ for all states are updated at a faster time scale,  and the $Q$ matrix is updated at a slower time scale. To be more precise, at every time $t$ of the outer loop, we solve the matrix game \footnote{Note that here we use the fact that  $\min_{g_\theta(\cdot\given s)}\max_{h_\nu(\cdot\given s)} [f_\tau \cdots]$ is equivalent to $\min_{\theta_s\in\RR^{|A|}}\max_{\nu_s\in\RR^{|A|}} [f_\tau \cdots]$ for each $s$.} 
\begin{align}\label{equ:tabular_minimax_each_s}
 \min_{{\theta_s} \in \mathbb{R}^n} \ \max_{{\nu_s}  \in \mathbb{R}^n } ~~~~f_{\tau} (Q(s); g_\theta(\cdot\given s),h_\nu(\cdot\given s)),
\end{align} 
for each state $s \in \cS$ by running $T_{inner}$ iterations of the Optimistic NPG algorithm (Algorithm \ref{alg:optimistic_npg}). 
At the end of each inner loop, the outer loop updates the $Q$ matrix for each state $s \in \cS$ as $Q_{t+1}(s, a, b) = r(s, a, b)  + \gamma \mathbb{E}_{s' \sim \mathbb{P}(\cdot|s, a, b)} [ f_{\tau} (Q_t(s'); g_{\theta_{T_{inner}}}(\cdot\given s'), h_{\nu_{T_{inner}}}(\cdot\given s'))]$. 

The complete algorithm is presented in Algorithm \ref{alg:markov_game}. Note that we use the name ONPG for Markov games because the inner matrix game is solved using the ONPG 
updates. The two-timescale-type update rule (between the policy and value updates) for solving infinite-horizon  Markov games has also been used before in \cite{sayin2021decentralized,cen2021fast,wei2021last}.

Next, we provide a convergence result for the performance of Algorithm \ref{alg:markov_game} for the regularized Markov game.  




\begin{theorem}
\label{thm:Markov_game_convergence}
Let $Q^\star_{\tau}$ be the NE $Q$-value of the regularized Markov Game under the tabular parametrization. Choose the stepsize $\eta = \frac{1-\gamma}{2(1 + \tau(\log n + 1 - \gamma))}$ for the inner loop in Algorithm \ref{alg:markov_game}. Let $T$ denote the total number of iterations $(T_{outer}\cdot T_{inner})$. Then, after
\begin{align}
T_{inner} &= \mathcal{O} \left( \frac{1}{\eta \tau} \left( \log \frac{1}{\epsilon} + \log \frac{1}{1 - \gamma} + \log \log n + \log \frac{1}{\eta} \right) \right), \nonumber \\
T_{outer} &= \mathcal{O} \left( \frac{1}{1 - \gamma} \left( \log \frac{d}{\epsilon} + \log \left( \frac{8}{\tau}  \left(1 + C^2 \|Q^\star\|^2_F \right) \right)  +\log \frac{1 + \tau \log n}{1 - \gamma} \right) \right),
\end{align}
iterations, we have $\| Q_T - Q^\star_{\tau} \|_\infty \leq \epsilon$ and $\max\{ \| \theta_T - \theta^\star \| , \| \nu_T - \nu^\star \| \} \leq \epsilon$
where $(Q_T, \theta_T, \nu_T)$ is the output of Algorithm \ref{alg:markov_game} after $T$ iterations, and $(\theta^\star, \nu^\star)$ are defined in Equation \eqref{def:optimal_para_MG_tab}. 
\end{theorem}

Finally, we show how the optimistic  NPG algorithm solves the original Markov game without regularization.
\begin{corollary}
\label{cor:MG_relation_unreg}
If we run Algorithm \ref{alg:markov_game}  for time $T_{inner} =  \mathcal{O} \left(  \frac{\log n}{(1 -\gamma)^2\epsilon} \log \left( \frac{1}{\epsilon} \right) \right)$, $T_{outer} =  \mathcal{O} \left( \frac{1}{(1 -\gamma)} \log \left( \frac{1}{\epsilon} \right) \right)$, and setting $\tau = \mathcal{O} ((1-\gamma) \epsilon/\log n)$, the output $(\theta_T, \nu_T)$ will be an $\epsilon$-in-class NE (Definition \ref{def:para_eps_NE}) of the original unregularized Markov game.
\end{corollary}

\begin{remark}
We remark that Theorem \ref{thm:Markov_game_convergence}, to the best of our knowledge, is the first to show policy parameter convergence in Markov games with policy parametrization, and is different from the policy  convergence results of several recent works \cite{zhao2021provably,wei2021last,cen2021fast} (see \S\ref{sec:related_work} for a detailed comparison). 
\end{remark}

We extend these results to simple function approximation settings in \S\ref{sec:appendix_func_approx_markov_results} in the Appendix.

%% file: Simulations.tex

\section{Simulations}\label{sec:simulations}

We now provide simulation results to corroborate our theoretical results. 
First, we study matrix games under the tabular setting in Figure \ref{fig:compare_vanilla_mod_NPG}. Here, we show the behavior of vanilla NPG and our proposed variant   (Equations \eqref{equ:new_NPG_1}-\eqref{equ:new_NPG_2}). We plot the first element of the iterate, i.e., $\theta(1)$ on the y-axis. It is shown that  even for vanishingly small stepsizes, vanilla NPG diverges, whereas the proposed variant converges  even with reasonable step-size choices. The cost matrix $Q$ is taken to be an identity matrix of dimension $5$.
 
Next, we confirm the convergence of our  variant of NPG in Figure \ref{fig:sub22}. In this figure, we compare the behavior of ONPG and NPG, and show that ONPG admits convergence for larger stepsizes. Smaller stepsizes that enables NPG convergence would lead to a slower  convergence rate than ONPG. This is in line with our results in Theorems \ref{thm:con_MWU} and \ref{thm:param_con}.


\begin{figure}[H]
\centering
\begin{subfigure}{.45\textwidth}
  \centering
  \includegraphics[width=1\linewidth]{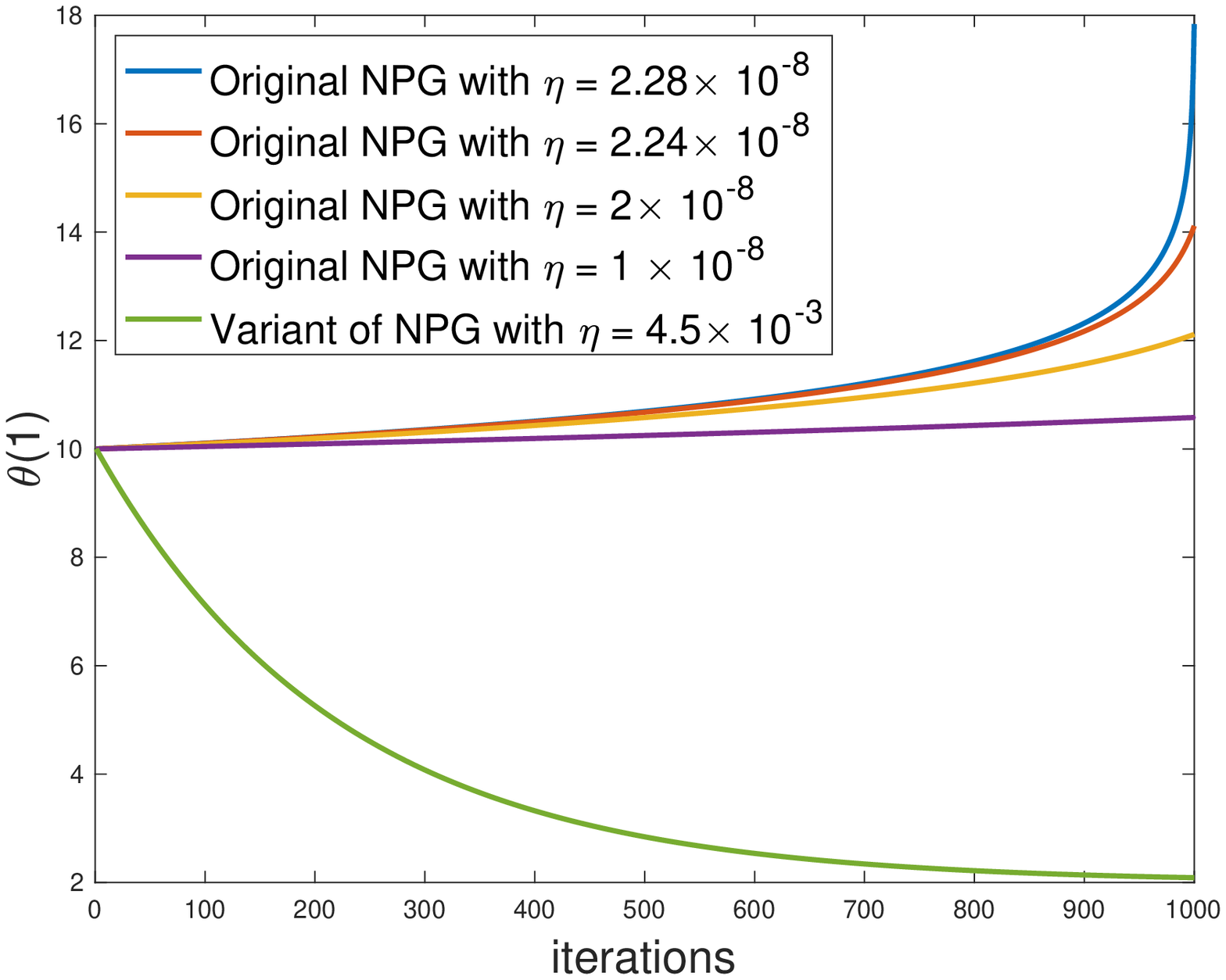}
  \caption{Vanilla NPG vs proposed variant of NPG.}
  \label{fig:compare_vanilla_mod_NPG}
\end{subfigure}%
\begin{subfigure}{.45\textwidth}
  \centering
  \includegraphics[width=1\linewidth]{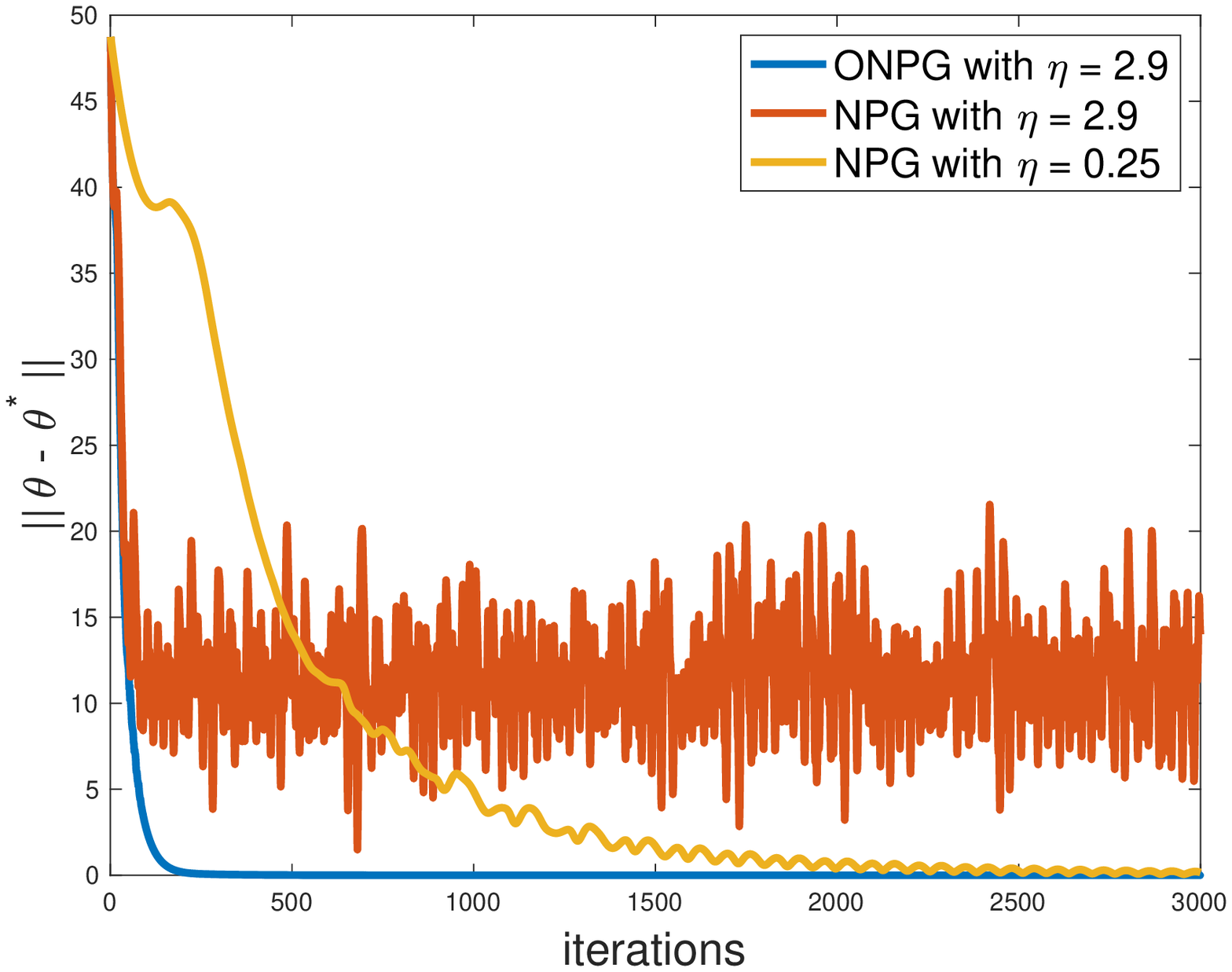}
  \caption{NPG  vs ONPG.}
  \label{fig:sub22}
\end{subfigure}
\caption{Comparison of vanilla NPG, proposed variant of NPG and ONPG in matrix game under the tabular setting, in terms of parameter convergence.}
\label{fig:test}
\end{figure}

 \subsection{ONPG in Markov games with function approximation}

Figures \ref{fig:sub1} and \ref{fig:sub2} study the behavior of Algorithm \ref{alg:markov_game_func} in  Markov games with log-linear function approximation, and corroborate the results of Theorem \ref{thm:Markov_game_convergence_func}. Here we take the feature matrix $\Phi \in \mathbb{R}^{10 \times 100}$, and $|\cS| = 10$, i.e.,  there are $10$  states. The first $10$ columns correspond to the first action of each of the $10$ states. 
This means that $\Phi_{(s, 1)} = e_s$ for $s = \{ 1, 2, \cdots, 10\}$ where $e_s$ is a standard basis vector with element $1$ at position $s$.  We take the discount factor $\gamma = 0.8$. We take the transition probability to be uniform for each state action pair, i.e., $\mathbb{P}(\cdot | s, a, b) = 1/10$ for all $(s, a, b)$, i.e., $\mathbb{P}(s' | s, a, b) = 1/10$ for all state $s' \in \cS$. Finally, we take the regularization parameter $\tau = 0.1$. 


\begin{figure}[H]
\centering
\begin{subfigure}{.45\textwidth}
  \centering
  \includegraphics[width=1\linewidth]{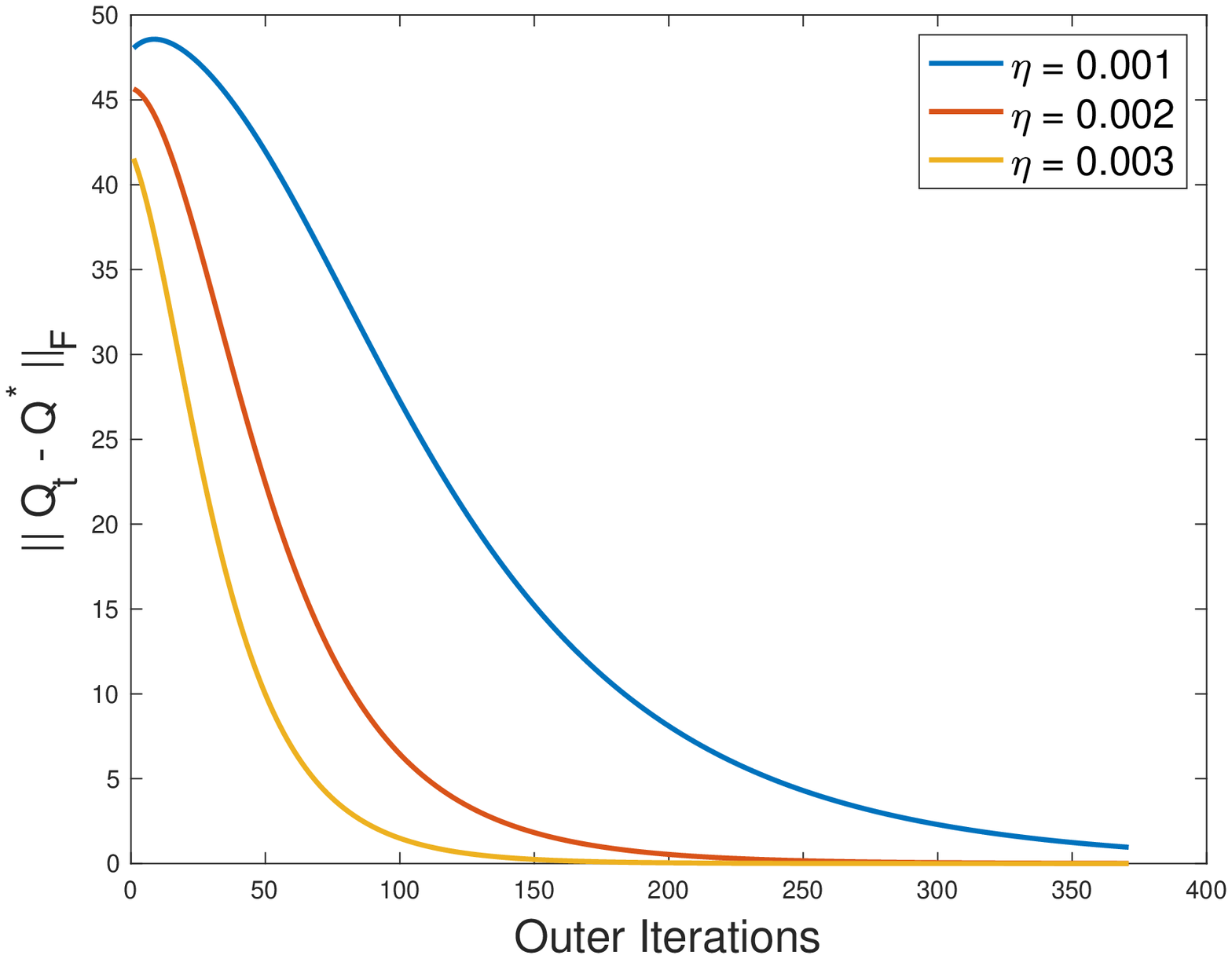}
  \caption{$Q$-matrix to in-class NE $Q$-matrix.}
  \label{fig:sub1}
\end{subfigure}%
\begin{subfigure}{.45\textwidth}
  \centering
  \includegraphics[width=1\linewidth]{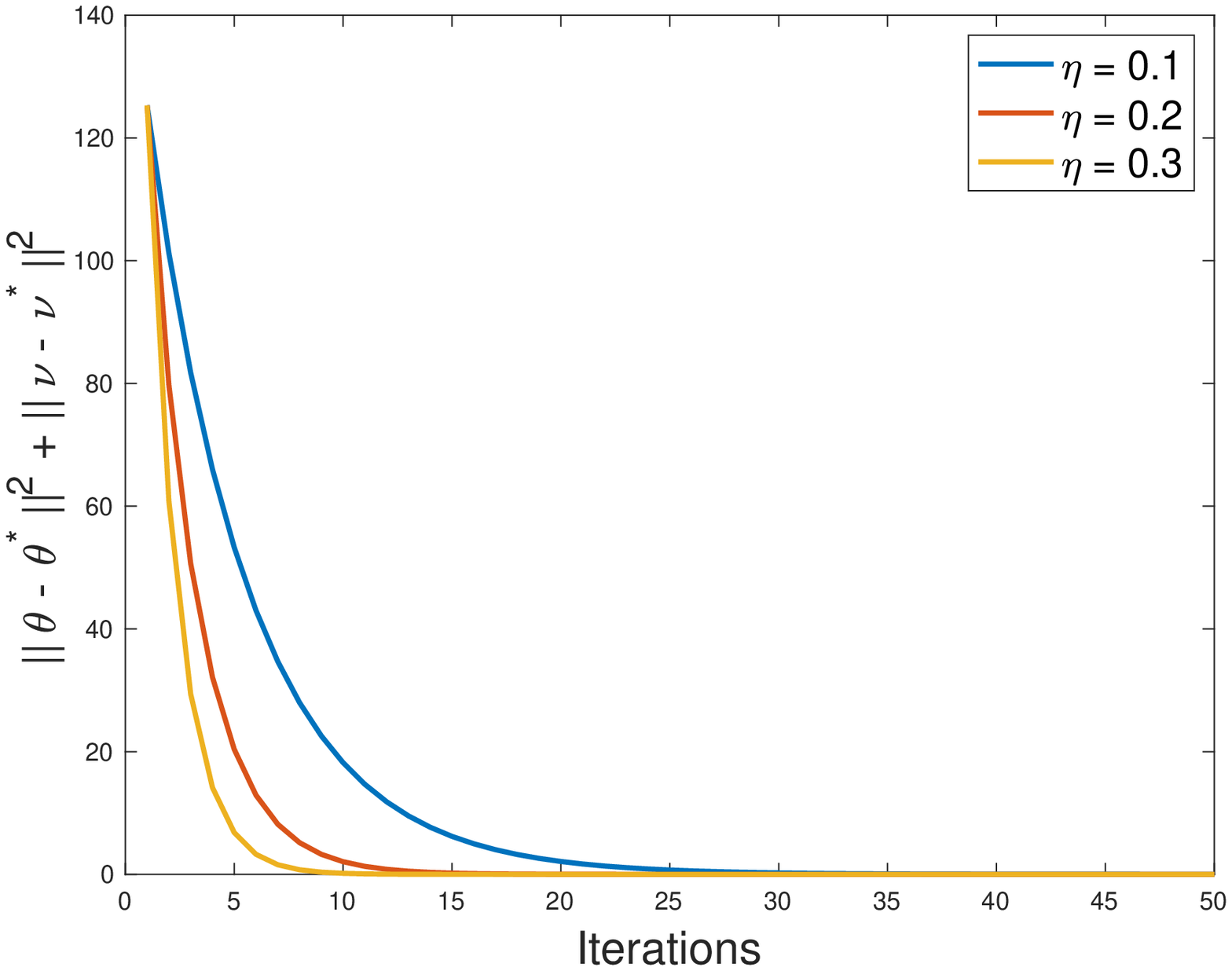}
  \caption{Parameters to the in-class NE parameters.}
  \label{fig:sub2}
\end{subfigure}
\caption{Convergence in Markov Games with linear function approximation.}
\label{fig:test}
\end{figure}

%% file: Conclusion.tex

\section{Concluding Remarks}\label{sec:conclusions}

In this paper, we study the global last-iterate parameter convergence of symmetric policy gradient methods for multi-agent learning. We identified the non-convergence issue of vanilla natural PG in policy parameters, even in presence of regularized reward function,  and developed    variants of natural PG methods  that enjoy last-iterate parameter convergence. We have then expanded the scope of the symmetric PG methods for multi-agent learning, and incorporated function approximation to handle large state-action spaces. Future work includes embracing more general function approximation in policy parameterization, and exploring the power of our approach in  nonconvex-nonconcave minimax optimization with other hidden convex structures. 


%% file: Appendix.tex

%
%



\input{NPG_background}

\section{Missing Details and Proofs in  \S\ref{sec:ONPG_Static_tabular}}
\label{sec:ONPG_Static_tabular_app}

\begin{remark}
We note that all results presented in this section also follow for the case where the action spaces for both players are asymmetric. However, we stick to the case where the number of actions is the same for both players, for ease of exposition. 
\end{remark}

\begin{algorithm}[tb]
   \caption{Optimistic NPG}
   \label{alg:optimistic_npg}
\begin{algorithmic}
   \STATE {\bfseries Initialize:} $\theta_0 = 0$ and $\nu_0 = 0$.
   \FOR{$t=1, 2, \cdots$}
   \STATE $\bar{\theta}_{t+1} = (1 - \eta \tau) \theta_t - \eta Q h_{\bar{\nu}_t} $
   \STATE $\bar{\nu}_{t+1} = (1 - \eta \tau) \nu_t + \eta Q^{\top} g_{\bar{\theta}_t}$
   \STATE 
   \STATE ${\theta}_{t+1} = (1 - \eta \tau) {\theta}_t - \eta Q h_{\bar{\nu}_{t+1}}$
   \STATE ${\nu}_{t+1} = (1 - \eta \tau) {\nu}_t + \eta Q^{\top} g_{\bar{\theta}_{t+1}}$
   \ENDFOR
\end{algorithmic}
\end{algorithm}

\subsection{Proof of Theorem \ref{thm:con_MWU}}
\label{subsec:MWU_app}

\subsubsection{Policy convergence} 

Consider the following modified NPG updates for the regularized game:
\begin{align}
{\theta}_{t+1} &=(1 - \eta \tau) \theta_t - \eta Q h_{{\nu}_t},\label{equ:new_NPG_1_MWU}\\ 
{\nu}_{t+1} &= (1 - \eta \tau) \nu_t + \eta Q^{\top} g_{{\theta}_t}.\label{equ:new_NPG_2_MWU}
\end{align}
Note that these updates correspond to the popular Multiplicative Weights Update  \citep{freund1997decision,arora2012multiplicative}  for the regularized game in policy space   (we succinctly represent $g_{\theta_t}$ and $h_{\nu_t}$ as $g_t$ and $h_t$,  respectively), i.e., 
\begin{align}
g_{t+1}(a)  \propto {g}_t(a)^{(1 - \eta \tau)} e^{-\eta [Q h_t]_a }, \nonumber \\
h_{t+1}(b)  \propto {h}_t(b)^{(1 - \eta \tau)} e^{\eta [Q^{\top} g_t]_b }. 
\end{align}

We can write these updates as a mirror descent update with Bregman function given by the negative entropy (i.e., the corresponding Bregman distance is the KL divergence) as follows: 
\begin{align}
g_{t+1} &= \argmin_{g \in \Delta} \ \{ \langle Qh_t + \tau \nabla_g \cH(g_t), g \rangle + \KL (g \| g_t) \}, \nonumber \\
h_{t+1} &= \argmax_{h \in \Delta} \ \{ \langle Q^{\top} g_t - \tau \nabla_h \cH(h_t), h \rangle - \KL (h \| h_t) \}.
\end{align}

Note that we can write these updates succinctly as one Mirror Descent update in the following form:
\begin{align}
z_{t+1} &= \argmin_{z \in \cZ} \ \{ \langle Mz_t + \tau \nabla \cH(z_t), z \rangle+ \KL (z \| z_t) \}, 
\end{align}
where $z = [g; h]$, $\cZ \in \Delta \times \Delta$ and with slight abuse of notation, we define $\nabla \cH(z) = [\nabla_g \cH(g); \nabla_h \cH(h)]$. 
Also, we define the matrix 
\begin{align}
M = \begin{pmatrix}
0 & Q \\ -Q^{\top} & 0
\end{pmatrix}.
\end{align}

We can now use properties of mirror decent to analyze the iterates of MWU.

First, we have the following two lemmas which follow from \cite{bauschke2003bregman}, Proposition 2.3, and Lemma D.4 in \cite{sokota2022unified}  
which will be used to derive the final convergence rate:
\begin{lemma}
\label{lemma:MWU_a}
For all $z \in \cZ$, we have
\begin{align}
\eta \langle Mz_t  + \tau \nabla \cH(z_t), z_{t+1} - z \rangle \leq \KL (z \| z_t) - \KL (z \| z_{t+1}) - \KL (z_{t+1} \| z_t). 
\end{align}
\end{lemma}

\begin{lemma}
\label{lemma:MWU_b}
For all $z \in \cZ$, we have
\begin{align}
\eta \langle Mz + \tau \nabla \cH(z), z^* - z \rangle \leq -\eta \tau (\KL (z \| z^*) + \KL (z^* \| z) ).
\end{align}
\end{lemma}

In the next lemma, we show that the iterates of MWU on the regularized problem will be bounded away from the boundary of the simplex.
\begin{lemma}
\label{lemma:MWU_c}For $\eta < 1/\tau$, the iterates of regularized MWU stay within a set $\Delta' \subset \Delta$ which is bounded away from the boundary of the simplex, i.e., $x_i \geq \delta > 0$ 
for some $\delta>0$, $\forall x \in \Delta'$. 
\end{lemma}
\begin{proof}
Consider the update of $g$.
We have the following property from Mirror Descent (see \cite{beck2003mirror}) 
\begin{align}
\KL (g^* \| g_{t+1} ) &\leq (1- \eta \tau) \KL(g^* \| g_t) - (1 - \eta \tau) \KL (g_{t+1} \| g_t) - \eta \tau \cH (g_{t+1}) + \eta \tau \cH(g^*) \nonumber \\
& \qquad \qquad - \eta \langle g_{t+1} - g^*, Qh_t \rangle \nonumber \\
&\leq (1- \eta \tau) \KL(g^* \| g_t) + 2\eta \tau \log n + 2\eta \| Q \|_{\infty}.
\end{align}
This implies that 
\begin{align}
\label{eq:KL_UB}
\KL (g^* \| g_{t} ) \leq 2 \log n + \frac{2 \| Q \|_{\infty}}{\tau} + \KL (g^* \| g_{0} ). 
\end{align}
From the definition of the $\KL$ divergence and Equation \eqref{eq:KL_UB}, we have:
\begin{align}
g_{t,i} \geq \exp \left( \frac{-1}{g^*_{\min}} \left(2 \log n + \frac{2 \| Q \|_{\infty}}{\tau}+ \KL (g^* \| g_{0} ) + \cH(g^*) \right)  \right) > 0, \qquad \forall  \ i{,~~ \text{and}~~  \forall\ t}.
\end{align}
Here $g_{\min}^*$ is the smallest value of the Nash Equilibrium policy (which is greater than $0$ since the NE policy of the regularized game is in the interior of the simplex). This completes the proof of the Lemma.
\end{proof}

Since the iterates lie within $\Delta'$, we let $L$ denote the Lipschitz constant of $Mz + \tau \nabla \cH(z)$ {(a continuous  function over $\Delta\times\Delta$ whose norm  approaches infinity as $z$ approaches the boundary)} in the set $\cZ'$ (where $\cZ' = \Delta' \times \Delta'$), i.e., 
\begin{align}
\label{eq:MWU_L}
\| (Mz_1 + \tau \nabla \cH(z_1)) - (Mz_2 + \tau \nabla \cH(z_2)) \| \leq L \|z_1 - z_2 \| , \qquad \forall z_1, z_2 \in \cZ'.
\end{align}

Now, we use these lemmas to derive the convergence rate of MWU for the regularized problem.

\begin{theorem}
\label{thm:Policy_conv_MWU}
Consider the modified NPG updates in Equation \eqref{equ:new_NPG_1_MWU}-\eqref{equ:new_NPG_2_MWU}   with stepsize satisfying $0 \leq \eta \leq \tau/L^2$. We then have:
\begin{align}
\KL (z^* \| z_{t+1}) \leq \left( 1 - \frac{\eta \tau}{2} \right) \KL (z^* \| z_{t}).
\end{align}
\end{theorem}
\begin{proof}
Note that the constraint $\eta \leq \tau/L^2$ will automatically satisfy $\eta < 1/\tau$ (as needed by Lemma \ref{lemma:MWU_c}) since $\tau \leq L$.

We have the following string of inequalities:

\begin{align}
\KL (z^* \| z_{t+1}) &\leq^{*1}  KL (z^* \| z_t) - \KL (z_{t+1} \| z_t ) + \eta \langle F(z_t) + \tau \nabla \cH (z_t) , z^* - z_{t+1} \rangle \nonumber \\
&= \KL (z^* \| z_t) - \KL (z_{t+1} \| z_t ) + \eta \langle F(z_{t+1}) + \tau \nabla g(z_{t+1}) , z^* - z_{t+1} \rangle \nonumber \\
& \qquad + \eta \langle  (F(z_t) + \tau \nabla \cH(z_t))  - (F(z_{t+1}) + \tau \nabla \cH (z_{t+1})), z^* - z_{t+1} \rangle \nonumber \\
&\leq^{*2} \KL (z^* \| z_t) - \KL (z_{t+1} \| z_t ) -\eta \tau \left( \KL(z_{t+1} \| z^*) + \KL(z^* \| z_{t+1}) \right)  \nonumber \\
& \qquad + \eta \langle  (F(z_t) + \tau \nabla \cH(z_t))  - (F(z_{t+1}) + \tau \nabla \cH (z_{t+1})), z^* - z_{t+1} \rangle \nonumber \\
&\leq^{*3}  \KL (z^* \| z_t) - \KL (z_{t+1} \| z_t ) -\eta \tau \left( \KL(z_{t+1} \| z^*) + \KL(z^* \| z_{t+1}) \right)  \nonumber \\
& \qquad \qquad + \eta L \| z_{t+1} - z_t \| \| z^* - z_{t+1} \| \nonumber \\
&\leq^{*4} \KL (z^* \| z_t) - \KL (z_{t+1} \| z_t ) -\eta \tau \left( \KL(z_{t+1} \| z^*) + \KL(z^* \| z_{t+1}) \right)  \nonumber \\
& \qquad \qquad + \frac{1}{2} \| z_{t+1} - z_t \|^2 + \frac{\eta^2 L^2}{2}  \| z^* - z_{t+1} \|^2 \nonumber \\
& \leq^{*5} \KL (z^* \| z_t) - \KL (z_{t+1} \| z_t ) - \KL(z_{t+1} \| z^*) + \eta^2L^2 \KL(z_{t+1} \| z^*) \nonumber \\
& \qquad \qquad \qquad - \eta \tau \KL(z_{t+1} \| z^*) - \eta \tau \KL (z^* \| z_{t+1} ) \nonumber \\
& \leq^{*6} \KL (z^* \| z_t) - \eta \tau \KL (z^* \| z_{t+1}).
\end{align}
Here $(*1)$ follows from Lemma \ref{lemma:MWU_a}, $(*2)$ follows from Lemma  \ref{lemma:MWU_b}, $(*3)$ follows from Equation \eqref{eq:MWU_L}, $(*4)$ follows from Young's inequality, $(*5)$ follows from Pinskers inequality and  $(*6)$ follows from $\eta \leq \tau/L^2$.
Therefore, we have 
\begin{align}
\KL (z^* \| z_{t+1}) \leq \frac{1}{ 1 + \eta \tau} \KL (z^* \| z_{t}) \leq \left( 1 - \frac{\eta \tau}{2} \right) \KL (z^* \| z_{t}),
\end{align}
which completes the proof.
\end{proof}

\subsubsection{Parameter convergence} 

In this section, we show convergence of the policy parameters. We have the following theorem. 

\begin{theorem}
\label{thm:Para_conv_MWU}
Consider the modified NPG updates in Equation \eqref{equ:new_NPG_1_MWU}-\eqref{equ:new_NPG_2_MWU}  with stepsize satisfying $0 \leq \eta \leq \tau/L^2$. We then have:
\begin{align}
\| \theta_{t+1} - \theta^\star \|^2 + \| \nu_{t+1} - \nu^\star \|^2  \leq (1 - \eta \tau/4)^t \Bigg(\| \theta_{0} - \theta^\star \|^2 + \| \nu_{0} - \nu^\star \|^2 +  \frac{4C}{\eta\tau} \Bigg),
\end{align}
where $\theta^\star = \frac{-Qh^\star}{\tau}$, $\nu^\star = \frac{Q^{\top}g^\star}{\tau}$ and $C =  \left(1 + \frac{1}{\eta\tau} (1 - \eta \tau)^2 \right)4 \eta^2 \|Q \|_\infty^2\KL(z^\star \| z_0)$. 
\end{theorem}

\begin{proof}

We begin by first providing the intuition of the proof, when the opponent is playing the NE strategy. We denote the NE strategies of the players as $g^\star$ and $h^\star$. Then, the NPG update has the following form: 
\begin{align}\label{equ:psedo_update_theta_MWU}
\theta_{t+1} &= (1-\eta \tau) \theta_t - \eta Q h^\star. 
\end{align}

We know that the NE satisfy:
\begin{align}
g^\star(a) = \frac{e^{-[Qh^\star]_a/\tau}}{\sum_{a' \in \mathcal{A}} e^{-[Qh^\star]_{a'}/\tau}} =  \frac{e^{-[Qh^\star]_a/\tau}}{K},
\end{align}
where we define $K:=\sum_{a' \in \mathcal{A}} e^{-[Qh^\star]_{a'}/\tau}$. 
Taking $\log$ on both sides, we have:
\begin{align}
 -\eta Q h^\star = \eta \tau \log g^\star + \eta \tau \log K.
\end{align}
Substituting this back into the $\theta$ update in \eqref{equ:psedo_update_theta_MWU}, we have:
\begin{align}
\theta_{t+1} &= (1-\eta \tau) \theta_t +  \eta \tau \log g^\star + \eta \tau \log K \nonumber \\
&=  \theta_t - \eta \tau \theta_t +  \eta \tau \log g^\star +  \eta \tau \log K- \eta \tau \log Z_{\theta_t} + \eta \tau \log Z_{\theta_t} \nonumber \\
&= \theta_t   +  \eta \tau \log g^\star - \eta \tau \log \left( \frac{e^{\theta_t}}{Z_{\theta_t}} \right) +  \eta \tau \log \left(  \frac{K}{Z_{\theta_t}} \right) \nonumber \\
&= \theta_t   +  \eta \tau \log g^\star - \eta \tau \log g_{\theta_t}  +  \eta \tau \log \left(  \frac{K}{Z_{\theta_t}} \right) = \theta_t   + \eta \tau \log \left( \frac{g^\star}{g_{\theta_t}} \right)  +  \eta \tau \log \left(  \frac{K}{Z_{\theta_t}} \right) \nonumber \\
&=  \theta_t   +  \eta \tau \log \left(  \frac{ g^\star K }{ g_{\theta_t} Z_{\theta_t} } \right).
\end{align}
We can further simplify this as:
\begin{align}
\theta_{t+1} &= \theta_t + \eta \tau \log \left(  \frac{e^{- [Qh^\star]/\tau}}{e^{ \theta_t}} \right),
\end{align}
which is nothing but:
\begin{align}
\theta_{t+1} &= \theta_t + \eta \tau \left( \frac{- Qh^\star}{\tau} - \theta_t \right).
\end{align}
Note that this is the Gradient Descent update on the strongly convex function $\frac{1}{2} \big\| \frac{- Qh^\star}{\tau} - \theta\big\|^2$ with stepsize $\eta \tau$. This update leads to the following convergence guarantees: 
\begin{align}
\| \theta_{t+1} - \theta^\star \|^2 \leq (1 - \eta \tau)^2 \| \theta_{t} - \theta^\star \|^2,
\end{align}
where $\theta^\star = \frac{- Qh^\star}{\tau}$. \\

The analysis above shows that if one of the players is already at  the NE strategy, the parameters of the second player converges to the NE at a linear rate. However, the original NPG update for $\theta$ is given by
\begin{align}
\theta_{t+1} &= (1 - \eta \tau) \theta_t -  \eta Q h^\star  + \eta Q ( h^\star - h_{{\nu}_{t}}). 
\end{align}
Since $h_{\bar{\nu}_{t+1}}$ converges to $h^\star$ at a linear rate (from Theorem \ref{thm:Policy_conv_MWU}), we expect the term 
\begin{align}
\varepsilon_t = \eta Q ( h^\star - h_{{\nu}_{t}}),
\end{align}
to be small, and goes to 0. This is formalized in what follows.

The NPG update for $\theta$ can be re-written using $\varepsilon_t$ as:
\begin{align}
\theta_{t+1} &= \theta_t + \eta \tau \left( \frac{-Qh^\star}{\tau} - \theta_t \right) + \varepsilon_t .
\end{align}
Once again, defining $\theta^\star = \frac{-Qh^\star}{\tau}$, we have:
\begin{align}
\| \theta_{t+1} - \theta^\star \|^2 &= \| \theta_t + \eta \tau \left( \theta^\star - \theta_t \right) + \varepsilon_t  - \theta^\star \|^2 = \| (1 - \eta \tau) (\theta_t - \theta^\star) + \varepsilon_t \|^2 \nonumber \\
&= (1 - \eta \tau)^2 \| \theta_t - \theta^\star \|^2 + 2(1 - \eta \tau) (\theta_t - \theta^\star)^{\top} \varepsilon_t + \| \varepsilon_t \|^2 \nonumber \\
&\leq^{*1}  (1 - \eta \tau)^2 \| \theta_t - \theta^\star\|^2 + \eta \tau \| \theta_t - \theta^\star\|^2 + \frac{1}{\eta\tau} (1 - \eta \tau)^2  \| \varepsilon_t \|^2 +  \| \varepsilon_t \|^2 \nonumber \\
&= (1 - \eta \tau + \eta^2 \tau^2)  \| \theta_t - \theta^\star\|^2 + \left(1 + \frac{1}{\eta\tau} (1 - \eta \tau)^2 \right)\| \varepsilon_t \|^2,
\label{eq:call_back_para_conv}
\end{align}
where $*1$ follows from Young's inequality. 

Next, we analyze the error term $\|\varepsilon_t\|$. We have:
\begin{align}
\|\varepsilon_t\|^2 &= \| \eta Q ( h^\star - h_{{\nu}_{t}}) \|^2 \leq \eta^2 \|Q \|^2_\infty \|( h^\star - h_{{\nu}_{t}}) \|^2_1 \leq^{*1} 2 \eta^2 \|Q \|^2_\infty (\KL(h^\star \| h_{{\nu}_{t}})  ) .
\end{align}
Here $*1$ follows from Pinsker's Inequality. Now, writing the same inequality for $\nu$, we have:
\begin{align}
&\| \theta_{t+1} - \theta^\star \|^2 + \| \nu_{t+1} - \nu^\star \|^2 \leq (1 - \eta \tau + \eta^2 \tau^2) (\| \theta_{t} - \theta^\star \|^2 + \| \nu_{t} - \nu^\star \|^2) \nonumber \\
& \qquad\qquad\qquad\qquad\qquad\qquad\qquad + \left(1 + \frac{1}{\eta\tau} (1 - \eta \tau)^2 \right) 2 \eta^2  \|Q \|^2_\infty \KL(z^\star \| {z}_{t} )  \\ 
&\quad\leq (1 - \eta \tau + \eta^2 \tau^2) (\| \theta_{t} - \theta^\star \|^2 + \| \nu_{t} - \nu^\star \|^2)  + \left(1 + \frac{1}{\eta\tau} (1 - \eta \tau)^2 \right) 4 \eta^2 C \|Q \|^2_\infty \left( 1 - \frac{\eta \tau}{2} \right)^t \KL(z^\star \| z_0).\nonumber  
\end{align}
Define:
\begin{align}
C = \left(1 + \frac{1}{\eta\tau} (1 - \eta \tau)^2 \right)4 \eta^2 \|Q \|_\infty^2\KL(z^\star \| z_0) .
\end{align}
Substituting back in Equation \eqref{eq:call_back_para_conv}, along with the corresponding expression for $\nu$, we have:
\begin{align}
\| \theta_{t+1} - \theta^\star \|^2 + \| \nu_{t+1} - \nu^\star \|^2 \leq (1 - \eta \tau + \eta^2 \tau^2) (\| \theta_{t} - \theta^\star \|^2 + \| \nu_{t} - \nu^\star \|^2) + C \left( 1 - \frac{\eta \tau}{2} \right)^t.
\end{align}
For $\eta \tau < 1/2$ we have:
\begin{align}
\| \theta_{t+1} - \theta^\star \|^2 + \| \nu_{t+1} - \nu^\star \|^2 \leq (1 - \eta \tau/4) (\| \theta_{t} - \theta^\star \|^2 + \| \nu_{t} - \nu^\star \|^2 )+ C(1 - \eta \tau/2)^t.
\end{align}
Consider the Lyapunov function:
\begin{align}
V_{t+1} = \| \theta_{t+1} - \theta^\star \|^2 + \| \nu_{t+1} - \nu^\star \|^2 + \frac{4C}{\eta\tau} (1 - \eta \tau/2)^{t+1}.
\end{align}
We have:
\begin{align}
V_{t+1} &\leq (1 - \eta \tau/4)  (\| \theta_{t} - \theta^\star \|^2 + \| \nu_{t} - \nu^\star \|^2) + C (1 - \eta \tau/2)^t + \frac{4C}{\eta\tau} (1 - \eta \tau/2) (1 - \eta \tau/2)^t \nonumber \\
& = (1 - \eta \tau/4) (\| \theta_{t} - \theta^\star \|^2 + \| \nu_{t} - \nu^\star \|^2)   +  \frac{4C}{\eta\tau} (1 - \eta \tau)^t (1 - \eta \tau/4) \nonumber \\
&= (1 - \eta \tau/4) \left( \| \theta_{t} - \theta^\star \|^2 + \| \nu_{t} - \nu^\star \|^2 + \frac{4C}{\eta\tau} (1 - \eta \tau)^t \right) = (1 - \eta\tau/4) V_t.
\end{align}
This shows linear convergence of the parameter $\theta$ to $\theta^\star$ since:
\begin{align}
\| \theta_{t+1} - \theta^\star \|^2 + \| \nu_{t+1} - \nu^\star \|^2 \leq V_{t+1} \leq (1 - \eta \tau/4)^t V_0.
\end{align}
This completes the proof. 
\end{proof}

\subsection{Proof of Theorem \ref{thm:param_con}}

We first prove the following result which follows from Theorem 1 in \cite{cen2021fast}. 

 \begin{lemma}
 \label{lemma:distribution_con_cen}
 [Theorem 1, \cite{cen2021fast}]
Consider Algorithm \ref{alg:optimistic_npg}. Suppose that the learning rates satisfy:
\begin{align}
0 &< \eta \leq \min \left\{ \frac{1}{2\tau + 2 \| Q \|_{\infty}}, \frac{1}{4 \| Q \|_{\infty}} \right\}.
\end{align}
Let $z_t = (g_{\theta_t}, h_{\nu_t})$ and  $\bar{z}_t = (g_{\bar{\theta}_t}, h_{\bar{\nu}_t})$. Then:
{\small\begin{align}
& \max \left\{ \KL(z^\star \| z_t), \frac{1}{2} \KL(z^\star \| \bar{z}_{t+1} ) \right\} 
\leq (1 - \eta \tau)^t \KL(z^\star \| z_0).
\end{align}}
\end{lemma}

\begin{proof}
Let $g_t$ and $\bar{g}_t$ denote $g_{\theta_t}$ and $g_{\bar{\theta}_t}$ respectively. Also, for any parameter $\theta$, we denote $Z_\theta$ to be the normalizing constant $\sum_{a' \in \mathcal{A}} e^{\theta(a')}$ (Define $Z_\nu$ similarly). We have:
 \begin{align}
 \bar{g}_{t+1}(a)  &=  \frac{e^{\bar{\theta}_{t+1}(a)}}{ \sum_{a' \in \mathcal{A}} e^{\bar{\theta}_{t+1}(a')} } \nonumber \\
 &\propto e^{\bar{\theta}_{t+1}(a)} \ = \ e^{(1 - \eta \tau) {\theta}_t(a) - \eta [Q h_{\bar{\nu}_t}]_a } \ = \ e^{(1 - \eta \tau) {\theta}_t(a) + \log Z_{{\theta}_t} - \log Z_{{\theta}_t} - \eta [Q h_{\bar{\nu}_t}]_a } \nonumber \\
 &= e^{(1 - \eta \tau) \log e^{{\theta}_t(a)} + \log Z_{{\theta}_t} - \log Z_{{\theta}_t} - \eta [Q h_{\bar{\nu}_t}]_a } \ = \ e^{(1 - \eta \tau) \log \left( \frac{e^{{\theta}_t(a)}}{Z_{{\theta}_t} } \right)  + \log Z_{{\theta}_t} - \eta [Q h_{\bar{\nu}_t}]_a } \nonumber \\
& \propto e^{(1 - \eta \tau) \log \left( \frac{e^{{\theta}_t(a)}}{Z_{{\theta}_t} } \right) - \eta [Q h_{\bar{\nu}_t}]_a } \ = \ e^{(1 - \eta \tau) \log {g}_t(a) - \eta [Q h_{\bar{\nu}_t}]_a } \ = \ e^{\log {g}_t(a)^{(1 - \eta \tau)} - \eta [Q h_{\bar{\nu}_t}]_a } \nonumber \\
&={g}_t(a)^{(1 - \eta \tau)} e^{-\eta [Q h_{\bar{\nu}_t}]_a }. 
 \end{align} 
Therefore:
\begin{align}
 \bar{g}_{t+1}(a)  \propto {g}_t(a)^{(1 - \eta \tau)} e^{-\eta [Q h_{\bar{\nu}_t}]_a }. 
\end{align}
Similarly, we have:
\begin{align}
{g}_{t+1}(a)  &\propto {g}_t(a)^{(1 - \eta \tau)} e^{-\eta [Q h_{\bar{\nu}_{t+1}}]_a }, \nonumber \\
\bar{h}_{t+1}(a)  &\propto {h}_t(a)^{(1 - \eta \tau)} e^{\eta [Q^{\top} g_{\bar{\theta}_t}]_a }, \nonumber \\
{h}_{t+1}(a)  &\propto {h}_t(a)^{(1 - \eta \tau)} e^{\eta [Q^{\top} g_{\bar{\theta}_{t+1}}]_a },
\end{align}
which is the same as the OMW updates for the regularized problem in \cite{cen2021fast}. Therefore, by Theorem 1 in  \cite{cen2021fast}, we have convergence of $g_\theta$ and $h_\nu$ to the solution of the regularized min-max problem.  
\end{proof}

We begin by first providing the intuition of the proof, when the opponent is playing the NE strategy. We denote the NE strategies of the players as $g^\star$ and $h^\star$. Then, the optimistic NPG update has the following form: 
\begin{align}\label{equ:psedo_update_theta}
\theta_{t+1} &= (1-\eta \tau) \theta_t - \eta Q h^\star. 
\end{align}

From \cite{mertikopoulos2016learning}, we know that the NE satisfy:
\begin{align}
g^\star(a) = \frac{e^{-[Qh^\star]_a/\tau}}{\sum_{a' \in \mathcal{A}} e^{-[Qh^\star]_{a'}/\tau}} =  \frac{e^{-[Qh^\star]_a/\tau}}{K},
\end{align}
where we define $K:=\sum_{a' \in \mathcal{A}} e^{-[Qh^\star]_{a'}/\tau}$. 
Taking $\log$ on both sides, we have:
\begin{align}
 -\eta Q h^\star = \eta \tau \log g^\star + \eta \tau \log K.
\end{align}
Substituting this back into the $\theta$ update in \eqref{equ:psedo_update_theta}, we have:
\begin{align}
\theta_{t+1} &= (1-\eta \tau) \theta_t +  \eta \tau \log g^\star + \eta \tau \log K \nonumber \\
&=  \theta_t - \eta \tau \theta_t +  \eta \tau \log g^\star +  \eta \tau \log K- \eta \tau \log Z_{\theta_t} + \eta \tau \log Z_{\theta_t} \nonumber \\
&= \theta_t   +  \eta \tau \log g^\star - \eta \tau \log \left( \frac{e^{\theta_t}}{Z_{\theta_t}} \right) +  \eta \tau \log \left(  \frac{K}{Z_{\theta_t}} \right) \nonumber \\
&= \theta_t   +  \eta \tau \log g^\star - \eta \tau \log g_{\theta_t}  +  \eta \tau \log \left(  \frac{K}{Z_{\theta_t}} \right) = \theta_t   + \eta \tau \log \left( \frac{g^\star}{g_{\theta_t}} \right)  +  \eta \tau \log \left(  \frac{K}{Z_{\theta_t}} \right) \nonumber \\
&=  \theta_t   +  \eta \tau \log \left(  \frac{ g^\star K }{ g_{\theta_t} Z_{\theta_t} } \right).
\end{align}
We can further simplify this as:
\begin{align}
\theta_{t+1} &= \theta_t + \eta \tau \log \left(  \frac{e^{- [Qh^\star]/\tau}}{e^{ \theta_t}} \right),
\end{align}
which is nothing but:
\begin{align}
\theta_{t+1} &= \theta_t + \eta \tau \left( \frac{- Qh^\star}{\tau} - \theta_t \right).
\end{align}
Note that this is the Gradient Descent update on the strongly convex function $\frac{1}{2} \big\| \frac{- Qh^\star}{\tau} - \theta\big\|^2$ with stepsize $\eta \tau$. Note that this update leads to the following convergence guarantees: 
\begin{align}
\| \theta_{t+1} - \theta^\star \|^2 \leq (1 - \eta \tau)^2 \| \theta_{t} - \theta^\star \|^2,
\end{align}
where $\theta^\star = \frac{- Qh^\star}{\tau}$. \\

The analysis above shows that if one of the players is already at  the NE strategy, the parameters of the second player converges to the NE at a linear rate. However, the original OGDA update for $\theta$ is given by
\begin{align}
\theta_{t+1} &= (1 - \eta \tau) \theta_t -  \eta Q h^\star  + \eta Q ( h^\star - h_{\bar{\nu}_{t+1}}). 
\end{align}
Since $h_{\bar{\nu}_{t+1}}$ converges to $h^\star$ at a linear rate (from Lemma \ref{lemma:distribution_con_cen}), we expect the term 
\begin{align}
\varepsilon_t = \eta Q ( h^\star - h_{\bar{\nu}_{t+1}}),
\end{align}
to be small, and goes to 0. This is formalized in what follows.

The OGDA update for $\theta$ can be re-written using $\varepsilon_t$ as:
\begin{align}
\theta_{t+1} &= \theta_t + \eta \tau \left( \frac{-Qh^\star}{\tau} - \theta_t \right) + \varepsilon_t .
\end{align}
Once again, defining $\theta^\star = \frac{-Qh^\star}{\tau}$, we have:
\begin{align}
\| \theta_{t+1} - \theta^\star \|^2 &= \| \theta_t + \eta \tau \left( \theta^\star - \theta_t \right) + \varepsilon_t  - \theta^\star \|^2 = \| (1 - \eta \tau) (\theta_t - \theta^\star) + \varepsilon_t \|^2 \nonumber \\
&= (1 - \eta \tau)^2 \| \theta_t - \theta^\star \|^2 + 2(1 - \eta \tau) (\theta_t - \theta^\star)^{\top} \varepsilon_t + \| \varepsilon_t \|^2 \nonumber \\
&\leq^{*1}  (1 - \eta \tau)^2 \| \theta_t - \theta^\star\|^2 + \eta \tau \| \theta_t - \theta^\star\|^2 + \frac{1}{\eta\tau} (1 - \eta \tau)^2  \| \varepsilon_t \|^2 +  \| \varepsilon_t \|^2 \nonumber \\
&= (1 - \eta \tau + \eta^2 \tau^2)  \| \theta_t - \theta^\star\|^2 + \left(1 + \frac{1}{\eta\tau} (1 - \eta \tau)^2 \right)\| \varepsilon_t \|^2,
\label{eq:call_back_para_conv}
\end{align}
where $*1$ follows from Young's inequality. 

Next, we analyze the error term $\|\varepsilon_t\|$. We have:
\begin{align}
\|\varepsilon_t\|^2 &= \| \eta Q ( h^\star - h_{\bar{\nu}_{t+1}}) \|^2 \leq \eta^2 \|Q \|^2_\infty \|( h^\star - h_{\bar{\nu}_{t+1}}) \|^2_1 \leq^{*1} 2 \eta^2 \|Q \|^2_\infty (\KL(h^\star \| h_{\bar{\nu}_{t+1}})  ) .
\end{align}
Here $*1$ follows from Pinsker's Inequality. Now, writing the same inequality for $\nu$, we have:
\begin{align}
&\| \theta_{t+1} - \theta^\star \|^2 + \| \nu_{t+1} - \nu^\star \|^2 \leq (1 - \eta \tau + \eta^2 \tau^2) (\| \theta_{t} - \theta^\star \|^2 + \| \nu_{t} - \nu^\star \|^2) \nonumber \\
& \qquad\qquad \qquad \qquad \qquad  + \left(1 + \frac{1}{\eta\tau} (1 - \eta \tau)^2 \right) 2 \eta^2  \|Q \|^2_\infty \KL(z^\star \| \bar{z}_{t+1} ) \\
&\quad\leq (1 - \eta \tau + \eta^2 \tau^2) (\| \theta_{t} - \theta^\star \|^2 + \| \nu_{t} - \nu^\star \|^2)  + \left(1 + \frac{1}{\eta\tau} (1 - \eta \tau)^2 \right) 4 \eta^2 C \|Q \|^2_\infty (1 - \eta \tau)^t \KL(z^\star \| z_0).\notag
\end{align}
Define:
\begin{align}
C = \left(1 + \frac{1}{\eta\tau} (1 - \eta \tau)^2 \right)4 \eta^2 \|Q \|_\infty^2\KL(z^\star \| z_0) .
\end{align}
This gives us (Using Lemma \ref{lemma:distribution_con_cen}):
\begin{align}
\| \theta_{t+1} - \theta^\star \|^2 + \| \nu_{t+1} - \nu^\star \|^2 \leq (1 - \eta \tau + \eta^2 \tau^2) (\| \theta_{t} - \theta^\star \|^2 + \| \nu_{t} - \nu^\star \|^2) + C (1 - \eta \tau)^t.
\end{align}
For $\eta \tau < 1/2$ we have:
\begin{align}
\| \theta_{t+1} - \theta^\star \|^2 + \| \nu_{t+1} - \nu^\star \|^2 \leq (1 - \eta \tau/2) (\| \theta_{t} - \theta^\star \|^2 + \| \nu_{t} - \nu^\star \|^2 )+ C(1 - \eta \tau)^t.
\end{align}
Consider the Lyapunov function:
\begin{align}
V_{t+1} = \| \theta_{t+1} - \theta^\star \|^2 + \| \nu_{t+1} - \nu^\star \|^2 + \frac{2C}{\eta\tau} (1 - \eta \tau)^{t+1}.
\end{align}
We have:
\begin{align}
V_{t+1} &\leq (1 - \eta \tau/2)  (\| \theta_{t} - \theta^\star \|^2 + \| \nu_{t} - \nu^\star \|^2) + C (1 - \eta \tau)^t + \frac{2C}{\eta\tau} (1 - \eta \tau) (1 - \eta \tau)^t \nonumber \\
& = (1 - \eta \tau/2) (\| \theta_{t} - \theta^\star \|^2 + \| \nu_{t} - \nu^\star \|^2)   +  \frac{2C}{\eta\tau} (1 - \eta \tau)^t (1 - \eta \tau/2) \nonumber \\
&= (1 - \eta \tau/2) \left( \| \theta_{t} - \theta^\star \|^2 + \| \nu_{t} - \nu^\star \|^2 + \frac{2C}{\eta\tau} (1 - \eta \tau)^t \right) \nonumber \\
&= (1 - \eta\tau/2) V_t.
\end{align}
This shows linear convergence of the parameter $\theta$ to $\theta^\star$ since:
\begin{align}
\| \theta_{t+1} - \theta^\star \|^2 + \| \nu_{t+1} - \nu^\star \|^2 \leq V_{t+1} \leq (1 - \eta \tau/2)^t V_0.
\end{align}
This completes the proof. \hfil \qed

Note that proof of Corollary \ref{cor:tab_matrix_connect_unreg} follows from Remark 4 and the preceding discussion in \cite{cen2021fast}.



\section{Missing Details and Proofs in \S\ref{sec:ONPG_Matrix_FA}}

\begin{remark}
We note that all results present in this section also follow for the case where the cardinality of the action spaces for both players are unequal. However, we stick to the case where the number of actions is the same for both players for ease of exposition.
\end{remark}


\begin{algorithm}[tb]
   \caption{Optimistic NPG (Function Approximation)}
   \label{alg:optimistic_npg_func}
\begin{algorithmic}
   \STATE {\bfseries Initialize:} $\theta_0 = 0$ and $\nu_0 = 0$.
   \FOR{$t=1, 2, \cdots$}
   \STATE $\bar{\theta}_{t+1} = (1 - \eta \tau) \theta_t - \eta  [(M^\top)^{-1} | 0]  \tilde{P} Q h_{\bar{\nu}_t} $
   \STATE $\bar{\nu}_{t+1} = (1 - \eta \tau) \nu_t + \eta  [(M^\top)^{-1} | 0]  \tilde{P} Q^{\top} g_{\bar{\theta}_t}$
   \STATE 
   \STATE ${\theta}_{t+1} = (1 - \eta \tau) {\theta}_t - \eta  [(M^\top)^{-1} | 0] \tilde{P} Q h_{\bar{\nu}_{t+1}}$
   \STATE ${\nu}_{t+1} = (1 - \eta \tau) {\nu}_t + \eta  [(M^\top)^{-1} | 0] \tilde{P} Q^{\top} g_{\bar{\theta}_{t+1}}$
   \ENDFOR
\end{algorithmic}
\end{algorithm}

\subsection{Proof of Lemma \ref{lemma:char_func_approx}}
\label{appendix:lemma:char_func_approx}

The first part of the lemma describes the set of distributions in the $m$-dimensional simplex covered by this parametrization. Since the set of distributions covered by the log-linear parametrization would be the same for all invertible $M$\footnote{To see this, consider $\tilde{\theta} = M\theta$. Since $M$ is invertible, there is a 1-1 correspondence between $\tilde{\theta}$ and $\theta$. The distribution parametrized by $\theta$ under the function approximation matrix $M$, is the same as the distribution parametrized by $\tilde{\theta}$ under the function approximation matrix $I$. Therefore it is enough to consider the special case of $M = \Ib$.}, for simplicity, we study the case where $M = \Ib$. Note that this would imply:
\begin{align}
g_{\theta}(a) &\propto e^{\theta(a)} \qquad \forall a \in \{ 1, 2, \cdots, d\}, \nonumber \\
g_{\theta}(a) &\propto 1 \qquad \quad \    \forall a \in \{ d+1, d+2, \cdots, n \}.
\end{align}
Similarly for $h_{\nu}$. Therefore, according to this parametrization the first $d$ elements can be chosen freely, and the rest $n-d$ parameters have to be equal. In other words, this parametrization covers the following set of distributions: 
\begin{align}
\label{eq:Delta_tilde_def_1}
\tilde{\Delta} = \{ \mu: \mu \in \Delta, \mu_{d+1} = \mu_{d+2} = \cdots = \mu_n  \},
\end{align}
which is a closed convex subset of the $n$-dimensional simplex. To see that any element of $\tilde{\Delta}$ can be represented by the log-linear parametrization, we can take the parameters $\theta(a) = \log \mu(a)$ for $a = 1, 2, \cdots, d$\footnote{Note that if $\mu(a) = 0$, the corresponding paramter would be $-\infty$.}. This would be a valid parametrization under the log-linear function approximation setting, and therefore all elements of $\tilde{\Delta}$ can be represented in this manner. Therefore, since the two sets are equivalent, we can rewrite the problem in terms of the policy vectors lying in the constraint set $\tilde{\Delta}$ which completes the proof of the lemma.

\subsection{Proof of Theorem \ref{thm:Nash_char_func_approx}}
\label{appendix:thm:Nash_char_func_approx}

For the regularized game  \eqref{eq:prob_contraint},  let player 2 play the NE strategy $h_\nu^\star$. Then player 1's optimization problem is given by:
\begin{align}
\min_{g_{\theta} \in \tilde{\Delta}}   \ g_{\theta} ^{\top} Q h_{\nu}^\star - \tau \mathcal{H}(g_{\theta}).
\end{align}
Define the following Lagrange multipliers (and associated constraints):
\begin{align}
\lambda \ &: \ \sum_{a = 1}^{n} g_\theta(a) = 1, \nonumber \\
\beta_1 \ &: \ g_\theta(d+1) = g_\theta(d+2), \nonumber \\
\beta_2 \ &: \ g_\theta(d+2) = g_\theta(d+3), \nonumber \\
&\cdots \nonumber \\
\beta_{n-d-1} \ &: \ g_\theta(n-1) = g_\theta(n).
\end{align}
Therefore, for the optimal Lagrange multipliers, taking the first-order optimality conditions (with respect to the variable $g_\theta$), we have:
\begin{align}
[Q h_{\nu}^\star]_a + \tau(\log g_\theta(a) + 1) &+ \lambda = 0 \qquad\qquad\qquad \qquad \qquad \forall a \in \{ 1, 2, \cdots, d\}, \nonumber \\
[Q h_{\nu}^\star]_a + \tau(\log g_\theta(a) + 1) &+ \lambda + \beta_{a - d} - \beta_{a-d-1}= 0 \qquad\quad \ \ \forall a \in \{ d+1, d+2, \cdots, n\} ,
\end{align}
where $\beta_0$ and $\beta_{n-d}$ are defined to be equal to $0$. This gives us:
\begin{align}
g^\star_\theta(a) \propto e^{-[Q h_{\nu}^\star]_a / \tau} \qquad \forall a \in \{ 1, 2, \cdots, d\}{.}  
\end{align}
For actions with indices $a > d$, the equality constraints give us:
\begin{align}
[Q h_{\nu}^\star]_a +  \beta_{a - d} - \beta_{a-d-1}= C \qquad\qquad\qquad \forall a \in \{ d+1, d+2, \cdots, n\}, 
\end{align}
for some constant $C$. On solving these equations  for $C$, we have: 
\begin{align}
C = \frac{1}{n-d} \sum_{a = d+1}^{n} [Q h_{\nu}^\star]_a,
\end{align}
which gives us
\begin{align}
g^\star_\theta(a) \propto e^{-C / \tau} \qquad \forall a \in \{ d + 1, \cdots, n\}.
\end{align}
Now, consider the symmetric matrix $\Psi \in \mathbb{R}^{n \times n}$ defined as:
\begin{align}
\Psi = 
\begin{pmatrix}
I_d & \zero \\ \zero & \begin{matrix} \frac{1}{n-d} &\cdots & \frac{1}{n-d} \\ \cdots &\cdots & \cdots \\ \frac{1}{n-d} &\cdots & \frac{1}{n-d} \end{matrix}
\end{pmatrix}.
\end{align}
Note that
\begin{align}
[\Psi^{\top} Q h_{\nu}^\star]_a &= [Q h_{\nu}^\star]_a \ \ \ \forall a \in \{ 1, 2, \cdots, d\}, \nonumber \\
[\Psi^{\top} Q h_{\nu}^\star]_a &= C \quad \qquad \forall a \in \{ d+1, d+2, \cdots, n\}.
\end{align}
Therefore, we can succinctly write the optimal distribution as: 
\begin{align}
g^\star_\theta(a) \propto e^{-[\Psi^{\top} Q h_{\nu}^\star]_a / \tau} \qquad \forall a \in \{ 1, 2, \cdots, n\}.
\end{align}
Also, we have $\Psi \mu = \mu$, ~$\forall \mu \in \tilde{\Delta}.$

Now, since $h_\nu^\star$ is the optimal distribution for Player $2$, we must have $h_\nu^\star \in \tilde{\Delta}$, which implies:
\begin{align}
\Psi h_{\nu}^\star = h_\nu^\star{.}
\end{align}
Doing a similar calculation for $g_{\nu}^\star$, we have that the NE satisfy: 
\begin{align}\label{equ:optimal_solution_linear_FA}
g^\star_\theta(a) = \frac{e^{-[\Psi^{\top} Q \Psi h^\star_\nu]_a/\tau}}{\sum_{a'} e^{-[\Psi^{\top} Q \Psi h^\star_\nu]_{a'}/\tau}}{,}\qquad\qquad\qquad
h^\star_\nu(a) =  \frac{e^{[\Psi^{\top} Q^\top \Psi g^\star_\theta]_a/\tau}}{\sum_{a'} e^{[\Psi^{\top} Q^\top \Psi g^\star_\theta]_{a'}/\tau}}.
\end{align}
Now, from this characterization of the NE, we see that these solutions are also the same as those of the following problem:

\begin{align}
\min_{g\in\Delta} \ \max_{h\in\Delta}  \ g ^{\top} \Psi^{\top} Q \Psi h - \tau \mathcal{H}(g) + \tau \mathcal{H}(h). 
\end{align}


Note that in the analysis above we have found a solution in the (relative) interior  of the constraint set. For example, if for some action $a$, we have $f_\theta(a) = 0$, then the term $\log f_\theta(a)$ is not defined and the Lagrangian would be different. However, since this is a strongly convex strongly concave minimax problem over a convex compact set, there is a unique solution (see \cite{facchinei2007finite}).  As we have already found a solution in the interior of the constraint set, by the argument above we have that this is the unique solution. This allows us to solve the first-order KKT optimality conditions (see \cite{boyd2004convex}) to find the solution. Also, since all terms satisfy $f_\theta(a) > 0$ (similarly for $g_\nu(a)$) we do not need to explicitly write down the Lagrange multipliers for the non-negativity constraint. This completes the proof. 
\hfil \qed

\subsection{Proof of Proposition \ref{eq:lemma:algo_lin_func_approx}}

The algorithm is similar to the one proposed for the tabular case. However, since the exponents are elements of $\Phi^{\top} \theta$ instead of just $\theta$ as was in the case of tabular softmax, we have the updates modified as well (we write the updates here for the combined update instead of the two step update for ease of presentation): 
\begin{align}
\Phi^\top \theta_{t+1} &= (1 - \eta \tau) \Phi^\top {\theta}_t - (2 - \eta \tau) \eta \Psi^{\top} \tilde{P} Q \Psi h_{\nu_t} + (1 - \eta \tau) \eta \Psi^{\top} \tilde{P} Q \Psi h_{\nu_{t-1}}, \nonumber \\
\Phi^\top \nu_{t+1} &= (1 - \eta \tau) \Phi^\top {\nu}_t + (2 - \eta \tau) \eta \Psi^{\top} \tilde{P} Q^\top \Psi g_{\theta_t} - (1 - \eta \tau) \eta \Psi^{\top} \tilde{P} Q^\top \Psi g_{\theta_{t-1}},
\end{align}
where $\Phi$ is a full-rank feature matrix. Note that the additional matrix $\tilde{P}$ is to ensure that the probability vector at each step of the algorithm satisfies the function approximation constraint\footnote{This is based on the fact that for a probability vector $x_i \propto e^{y_i}$, we will also have $x_i \propto e^{y_i - k}$ for any constant k independent of the index $i$.}. Since $\Phi$ is full-rank, we can explicitly write this update for $\theta$ and $\nu$ as follows:
\$
 \theta_{t+1} &= [\Phi \Phi^{\top}]^{-1}\Phi \left( (1 - \eta \tau) \Phi^\top {\theta}_t - (2 - \eta \tau) \eta \Psi^{\top} \tilde{P} Q \Psi h_{\nu_t} + (1 - \eta \tau) \eta \Psi^{\top} \tilde{P} Q \Psi h_{\nu_{t-1}} \right), \nonumber \\
 \nu_{t+1} &=  [\Phi \Phi^{\top}]^{-1}\Phi \left( (1 - \eta \tau) \Phi^\top {\nu}_t + (2 - \eta \tau) \eta \Psi^{\top} \tilde{P} Q^\top \Psi g_{\theta_t} - (1 - \eta \tau) \eta \Psi^{\top} \tilde{P} Q^\top \Psi g_{\theta_{t-1}} \right),
\$
which can be simplified to:
\begin{align}
 \theta_{t+1} &= (1 - \eta \tau) {\theta}_t - (2 - \eta \tau) \eta  [\Phi \Phi^{\top}]^{-1}\Phi  \Psi^{\top} \tilde{P} Q \Psi h_{\nu_t} + (1 - \eta \tau) \eta  [\Phi \Phi^{\top}]^{-1}\Phi  \Psi^{\top} \tilde{P} Q \Psi h_{\nu_{t-1}} , \nonumber \\
 \nu_{t+1} &= (1 - \eta \tau) {\nu}_t + (2 - \eta \tau) \eta  [\Phi \Phi^{\top}]^{-1}\Phi  \Psi^{\top} \tilde{P} Q^\top \Psi g_{\theta_t} - (1 - \eta \tau) \eta  [\Phi \Phi^{\top}]^{-1}\Phi  \Psi^{\top} \tilde{P} Q^\top \Psi g_{\theta_{t-1}} .
\end{align}

Note that for $\Phi = [M \ | \ 0]$, we have 
\begin{align}
[\Phi \Phi^{\top}]^{-1}\Phi &= ([M | 0] [M | 0]^{\top} )^{-1} [M | 0] = (M M^{\top})^{-1} [M|0] = [(M^{\top})^{-1} | 0].
\end{align}

From the previous discussion, since all the terms $g_{\theta_t}$ and $h_{\nu_t}$ lie in the set $\tilde{\Delta}$, we have $\Psi g_{\theta_t} = g_{\theta_t}$ and $\Psi h_{\nu_t} = h_{\nu_t}$. Also, we have:
\begin{align}
[\Phi \Phi^{\top}]^{-1}\Phi  \Psi^{\top} &= [(M^{\top})^{-1} | 0]  \Psi^{\top} = [(M^\top)^{-1} | 0], 
\end{align}
from the structure of $\Psi$. Therefore, the update can be written as:
\begin{align}
 \theta_{t+1} &= (1 - \eta \tau) {\theta}_t - (2 - \eta \tau) \eta   [(M^\top)^{-1} | 0]  \tilde{P} Q  h_{\nu_t} + (1 - \eta \tau) \eta   [(M^\top)^{-1} | 0]  \tilde{P} Q h_{\nu_{t-1}} , \nonumber \\
 \nu_{t+1} &= (1 - \eta \tau) {\nu}_t + (2 - \eta \tau) \eta  [(M^\top)^{-1} | 0]  \tilde{P} Q^\top  g_{\theta_t} - (1 - \eta \tau) \eta  [(M^\top)^{-1} | 0]  \tilde{P} Q^\top g_{\theta_{t-1}},
\end{align}
which completes the proof.
\hfil \qed

\subsection{Simulation to show convergence of ONPG in the function approximation setting}


In this section, we show the performance of the ONPG algorithm under Function Approximation (Algorithm \ref{alg:optimistic_npg_func}). The policies have a log linear parametrization, with the feature matrix  $\Phi = [I \ | \ 0] \in \mathbb{R}^{10 \times 100}$, and the cost matrix $Q$ is chosen to be a random matrix of dimension $100 \times 100$.
 
\begin{figure}[H]
  \centering 
\includegraphics[width=0.6\columnwidth]{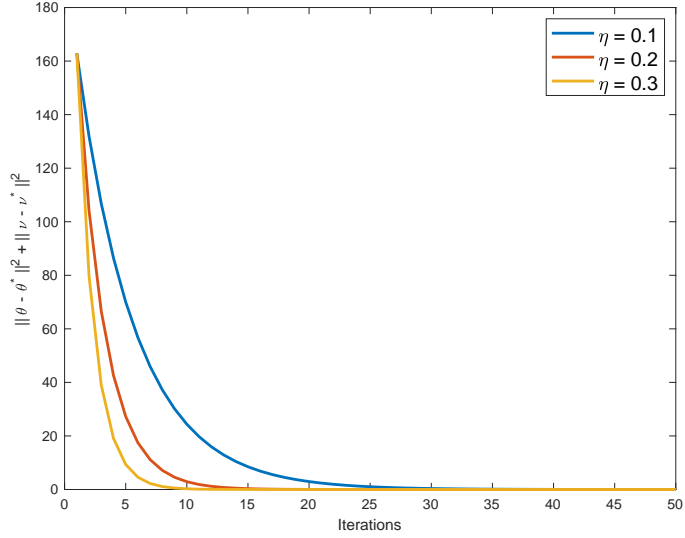}
    \caption{Behavior of Algorithm \ref{alg:optimistic_npg_func} in the matrix game under the log-linear function approximation setting.}
 \label{fig:NPG_FA_Matrix}
 \end{figure}
 

\section{Missing Definitions and Proofs in \S\ref{sec:Gen_Monotone}} 

\begin{remark}
We note that all results presented in this section also follow for the case where 
the number of possible actions for each player can be different.  However, we stick to the case where the number of action is the same for both players for ease of exposition. {Note that the actual action spaces need not be identical, but only their cardinalities.}
\end{remark}

\begin{definition}[In-class Nash equilibrium for a monotone game]\label{def:para_NE_mono_in_class}
	The policy parameter $\theta^\star = [\theta_1^\star, \theta_2^\star, \cdots, \theta_N^\star]$ is an {NE under function approximation, i.e., (in-class NE)} 
	of the monotone game, if it satisfies that for all  $i\in[N]$, 
\begin{align}
f_i( {g_{{\theta}_i^\star}, g_{{\theta}_{-i}^\star}})  \leq f_i( {g_{{\theta}_i}, g_{{\theta}_{-i}^\star}}) , \qquad \forall \theta_i \in \mathbb{R}^d. 
\end{align}
Note that if we are in the tabular setting, we will have $d = n$.
\end{definition}

\begin{definition}[$\epsilon$-in-class Nash equilibrium for a monotone game]\label{def:para_eps_NE_mono}
	The policy parameter $(\tilde{\theta}_1,\cdots,\tilde{\theta}_N)$ is an $\epsilon$-{\it Nash equilibrium}  {under function approximation} (or {\it in-class} $\epsilon$-NE) of the monotone game if it satisfies that for all  $i\in[N]$, 
\begin{align}
f_i( {g_{\tilde{\theta}_i}, g_{\tilde{\theta}_{-i}}}) - \epsilon \leq f_i( {g_{{\theta}_i}, g_{\tilde{\theta}_{-i}}}) , \qquad \forall \theta_i \in \mathbb{R}^d. 
\end{align}
Note that if we are in the tabular setting, we will have $d = n$.
\end{definition}

\begin{algorithm}[tb]
   \caption{Optimistic NPG for monotone games}
   \label{alg:ogda_monotone}
\begin{algorithmic}
   \STATE {\bfseries Initialize:} $\theta^0_i = 0$ for all players $i$. 
   \FOR{$t=1, 2, \cdots$}
   \STATE $\bar{\theta}_i^{t+1} = (1 - \eta \tau)\theta_i^t -\eta \nabla_{g_{\theta_i}} f_i(g_{\bar{\theta}^t_i}, g_{\bar{\theta}^t_{-i}}) \ \ \forall i \in [N]$.
   \STATE $\theta_i^{t+1} = (1 - \eta \tau)\theta_i^t -\eta \nabla_{g_{\theta_i}} f_i(g_{\bar{\theta}^{t+1}_i}, g_{\bar{\theta}^{t+1}_{-i}}) \ \ \forall i \in [N].$
   \ENDFOR
\end{algorithmic}
\end{algorithm}

We will also use the notation $f_i^\tau(g_i, g_{-i}) = f_i(g_i, g_{-i}) - \tau \mathcal{H} (g_i)$. 
\subsection{Proof of Lemma \ref{lemma:optimal_solution_N_player}}

We can follow the analysis for a two player game from  \cite{mertikopoulos2016learning} to write down the solution form for the N-player monotone setting. We let $g_i$ and $g_{-i}$ denote $g_{\theta_i}$ and $g_{\theta_{-i}}$ respectively. {First, note that the solutions in the policy space exists for the unregularized game, since we are solving a monotone VI over a convex compact set, and this solution is unique if we regularize the game, since in this case we are solving a stringly monotone VI over a convex compact set. See \cite{facchinei2007finite}.}

Consider player $i's$ optimization problem when other players play the equilibrium strategies:  
\begin{align}
\min_{g_i \in \Delta} f_i(g_i, g^\star_{-i})  - \tau \mathcal{H}(g_i).
\end{align}
Since we are in the monotone setting with a strongly convex regularizer, the first order Karush–Kuhn–Tucker (KKT) conditions are both necessary and sufficient. The first order KKT conditions  are:
\begin{align}
[\nabla_{g_i} f_i(g_i^\star, g_{-i}^\star)]_a + \tau(\log g_i^\star + 1) &+ \lambda = 0, \nonumber 
\end{align}
where $\lambda$ is the Lagrange multiplier corresponding to the simplex constraint. This implies:
\begin{align}
g_i^\star(a) &\propto e^{\frac{- [\nabla_{g_i} f_i(g_i^\star, g_{-i}^\star)]_a}{\tau}}, 
\end{align}
which shows the NE in the policy space. Now, to complete the proof of the lemma, we need to find parameters $\theta^\star$ which leads to this distribution. This can be easily seen by setting $\theta_i^\star = \frac{- [\nabla_{g_i} f_i(g_i^\star, g_{-i}^\star)]}{\tau}$, thereby completing the proof of the lemma.
\hfil \qed

\subsection{(Optimistic) NPG for monotone games}\label{sec:append:ONPG_update} 

As noted in \S\ref{sec:vanilla_NPG_deriv}, we have that the Fisher Information matrix $F_\theta(\theta) = \nabla_\theta g_{\theta}$. Therefore, the NPG update for player $i$ can be simplified as: 
\begin{align}
{\theta}^{t+1}_i &={\theta}^{t}_i-\eta \cdot F^{\dagger}_\theta(\theta^t_i)\cdot\frac{\partial f^\tau_i(g_{\theta^t_i},g_{\theta^t_{-i}})}{\partial\theta} ={\theta}_{t}-\eta\frac{\partial f^\tau_i(g_{\theta^t_i},g_{\theta^t_{-i}})}{\partial g_{\theta_i}} \nonumber \\
&= \theta_t - \eta \left( \nabla_{g_{\theta_i}} f_i(g_{\theta_i^t}, g_{\theta_{-i}^t}) + \tau (\mathbbm{1} + \log g_{\theta^t_i}) \right).
\end{align}

This can be simplified as: 
\begin{align}
{\theta}^{t+1}_i(a) = (1 - \eta \tau) {\theta}^{t}_i(a) - \eta [\nabla_{g_{\theta_i}} f_i(g_{\theta_i^t}, g_{\theta_{-i}^t})]_a + \eta \tau (\log Z_{\theta_i^t} - 1),
\end{align}
where $Z_{\theta_i^t} = \sum_{a' \in \cA}e^{\theta_i^t(a')}$. Note that this update will have the same pitfall of parameter divergence as the NPG update for the matrix game, since a matrix game is a special case of the monotone game. Therefore, we propose the following modified version of NPG, as done for the matrix game:
\begin{align}
{\theta}^{t+1}_i(a) = (1 - \eta \tau) {\theta}^{t}_i(a) - \eta [\nabla_{g_{\theta_i}} f_i(g_{\theta_i^t}, g_{\theta_{-i}^t})]_a.
\end{align}
This leads to the modified NPG dynamics for the monotone game. Now, similar to the matrix game, we analyze the optimistic version of this algorithm with updates:
\$
\bar{\theta}_i^{t+1} = (1 - \eta \tau)\theta_i^t -\eta \nabla_{g_{\theta_i}} f_i(g_{\bar{\theta}^t_i}, g_{\bar{\theta}^t_{-i}}), \qquad \theta_i^{t+1} = (1 - \eta \tau)\theta_i^t -\eta \nabla_{g_{\theta_i}} f_i(g_{\bar{\theta}^{t+1}_i}, g_{\bar{\theta}^{t+1}_{-i}}),
\$
in \S \ref{sec:Gen_Monotone}.

\subsection{Proof of Theorem \ref{thm:eg_og_conv_mono}}
\label{appendix:thm:eg_og_conv_mono}

We prove the following Lemma first:
\begin{lemma}
\label{Lemma_1}
For any $z = (g_{\theta_i}, g_{\theta_{-i}})  \in \Delta^N $, consider an update of the form:
\begin{align}
\label{eq:general_update_form}
\theta_i^{t+1} &= (1 - \eta \tau) \theta_i^{t} -\eta \nabla_{g_{\theta_i}} f_i(g_{\theta_i}, g_{\theta_{-i}}) \ \ \  \forall i \in [N].
\end{align}
We have:
\begin{align}
\langle \log z_{t+1} - (1 - \eta \tau) \log z_t - \eta \tau \log z^\star, z - z^\star  \rangle \leq 0,
\end{align}
where $z^\star = (g_i^\star)_{i=1}^N$.
\end{lemma}
\begin{proof}
In the proof below, we define $g_i := g_{\theta_i}$ and $g_i^t := g_{\theta_i^t}$.
From the update sequence in Equation \eqref{eq:general_update_form}, we have
\begin{align}
\log g_i^{t+1} = (1- \eta \tau) \log g_i^t - \eta \nabla_{g_i} f_i(g_i, g_{-i}) + c\cdot\mathbbm{1},
\end{align}
where $c$ is the normalization constant. This implies:
\begin{align}
\langle \log g_i^{t+1}  - (1- \eta \tau) \log g_i^t, g_i - g_i^\star  \rangle &= \langle  - \eta \nabla_{g_i} f_i(g_i, g_{-i})+ c\cdot \mathbbm{1}, g_i - g_i^\star  \rangle \nonumber \\
&= \langle  - \eta \nabla_{g_i} f_i(g_i, g_{-i}),g_i - g_i^\star  \rangle. 
\end{align}
Note that $ \langle c\cdot \mathbbm{1}, g_i - g_i^\star  \rangle = 0$, since $g_i, g_i^\star \in \Delta$. Since this is true for all players $i$, we have:
\begin{align}
\langle \log z_{t+1} - (1 - \eta \tau) \log z_t, z - z^\star  \rangle &=  -\eta \sum_i \langle  \nabla_{g_i} f_i(g_i, g_{-i}),g_i - g_i^\star  \rangle = -\eta \langle F(z), z - z^\star \rangle.
\label{lemma1_main1}
\end{align}
Now, from the properties of the NE, we have:
\begin{align}
\eta \tau \log g_i^\star= -\eta \nabla_{g_i} f_i(g_i^\star, g_{-i}^\star) +  c\cdot \mathbbm{1},
\end{align}
which gives us:
\begin{align}
\langle \eta \tau \log g_i^\star ,  g_i - g_i^\star \rangle = \langle   -\eta \nabla_{g_i} f_i(g_i^\star, g_{-i}^\star) , g_i - g_i^\star \rangle.
\end{align}
Since this is true for all players $i$, we have
\begin{align}
\label{lemma1_main2}
\langle \eta \tau \log z^\star , z - z^\star\rangle = - \eta \langle  F(z^\star), z - z^\star \rangle.
\end{align}
From Equations \eqref{lemma1_main1} and \eqref{lemma1_main2}, we have:
\begin{align}
\langle \log z_{t+1} - (1 - \eta \tau) \log z_t - \eta \tau \log z^\star, z - z^\star  \rangle &=  -\eta \langle F(z) - F(z^\star), z - z^\star \rangle \leq 0,
\end{align}
where the last step uses the monotonicity assumption of $F$. This completes the proof of the lemma.
\end{proof}

\begin{lemma}
For updates of the form:
\begin{align}
\theta_i^{t+1} &= (1 - \eta \tau) \theta_i^{t} -\eta \nabla_{g_{\theta_i}} f_i(g_{\bar{\theta}^{t+1}_i}, g_{\bar{\theta}^{t+1}_{-i}}), \ \ \  \forall i \in [N],
\end{align}
we have
\begin{align}
\label{eq:common_1}
(1 - \eta \tau)\KL(z^\star \| z_t) \geq (1 - \eta \tau) \KL & (\bar{z}_{t+1} \| z_t )  + \eta \tau \KL (\bar{z}_{t+1} \| z^\star) + \KL (z_{t+1} \| \bar{z}_{t+1} ) \nonumber \\
& \qquad -\langle \log \bar{z}_{t+1} - \log z_{t+1}, \bar{z}_{t+1} - z_{t+1} \rangle + \KL(z^\star \| z_{t+1}),
\end{align}
and
\begin{align}
\label{eq:common_2}
\KL(z^\star \| \bar{z}_{t+1}) = \KL(z^\star \| z_{t+1}) - \KL(\bar{z}_{t+1} \| z_{t+1} ) - \langle z^\star - \bar{z}_{t+1}, \log \bar{z}_{t+1} - \log z_{t+1} \rangle.
\end{align}
\end{lemma}

\begin{proof}
From the definition of $\KL$ divergence we have:
\begin{align}
\label{eq:sub1}
- \langle \log z_{t+1} - (1 - \eta \tau) \log z_t - \eta \tau \log z^\star, z^\star  \rangle = - (1 - \eta \tau) \KL (z^\star \| z_t) + \KL(z^\star \| z_{t+1} ). 
\end{align}
Next, note that:
\begin{align}
\label{eq:sub2}
&\langle \log z_{t+1} - (1 - \eta \tau) \log z_t - \eta \tau \log z^\star, \bar{z}_{t+1} \rangle \nonumber \\
&= \langle \log \bar{z}_{t+1} - (1 - \eta \tau) \log z_t - \eta \tau \log z^\star, \bar{z}_{t+1} \rangle + \langle \log \bar{z}_{t+1} - \log z_{t+1} , z_{t+1} \rangle    - \langle \log \bar{z}_{t+1} - \log z_{t+1}, \bar{z}_{t+1} - z_{t+1} \rangle \nonumber \\
&= (1 - \eta \tau) \KL (\bar{z}_{t+1} \| z_t ) + \eta \tau \KL (\bar{z}_{t+1} \| z^\star) + \KL (z_{t+1} \| \bar{z}_{t+1})   - \langle \log \bar{z}_{t+1} - \log z_{t+1}, \bar{z}_{t+1} - z_{t+1} \rangle.
\end{align}
Now, substituting $z = \bar{z}_{t+1}$ in Lemma \ref{Lemma_1} we have:
\begin{align}
\label{eq:sub3}
\langle \log z_{t+1} - (1 - \eta \tau) \log z_t - \eta \tau \log z^\star, \bar{z}_{t+1} - z^\star  \rangle \leq 0. 
\end{align}

Substituting Equations \eqref{eq:sub1} and \eqref{eq:sub2} in \eqref{eq:sub3}, we get the Inequality \eqref{eq:common_1}.

Inequality \eqref{eq:common_2} follows from the properties of KL divergence.
\end{proof}

From the ONPG updates in Algorithm \eqref{alg:ogda_monotone}, we have:
\begin{align}
\log \bar{g}_i^{t+1} - \log g_i^{t+1} = - \eta \left( \nabla_{g_i} f_i(\bar{g}_i^{t}, \bar{g}_{-i}^{t}) - \nabla_{g_i} f_i(\bar{g}_i^{t+1}, \bar{g}_{-i}^{t+1}) \right) + c\cdot \mathbbm{1}.
\end{align}
This implies:
\begin{align}
&\langle \log \bar{g}_i^{t+1} - \log g_i^{t+1}, \bar{g}_i^{t+1} - g_i^{t+1} \rangle \nonumber \\
&= -\eta \langle   \nabla_{g_i} f_i(\bar{g}_i^{t}, \bar{g}_{-i}^{t}) - \nabla_{g_i} f_i({g}_i^{t}, {g}_{-i}^{t}) + \nabla_{g_i} f_i({g}_i^{t}, {g}_{-i}^{t}) - \nabla_{g_i} f_i(\bar{g}_i^{t+1}, \bar{g}_{-i}^{t+1}) , \bar{g}_i^{t+1} - g_i^{t+1} \rangle \nonumber \\
&\leq^{*1} \eta \left(\|   \nabla_{g_i} f_i(\bar{g}_i^{t}, \bar{g}_{-i}^{t}) - \nabla_{g_i} f_i({g}_i^{t}, {g}_{-i}^{t})  \| +   \|  \nabla_{g_i} f_i({g}_i^{t}, {g}_{-i}^{t}) - \nabla_{g_i} f_i(\bar{g}_i^{t+1}, \bar{g}_{-i}^{t+1}) \| \right) \|\bar{g}_i^{t+1} - g_i^{t+1} \| \nonumber \\
&\leq^{*2} \eta L \left( \|\bar{g}_i^{t} - g_i^t \| + \| \bar{g}_{-i}^{t} - g_{-i}^t \| +  \|\bar{g}_i^{t+1} - g_i^t \| + \| \bar{g}_{-i}^{t+1} - g_{-i}^t \| \right) \| \bar{g}_i^{t+1} - g_i^{t+1}  \| \nonumber \\
&\leq^{*3} \frac{1}{2} \eta L \left(  \|\bar{g}_i^{t} - g_i^t \|^2 + \| \bar{g}_{-i}^{t} - g_{-i}^t \|^2 +  \|\bar{g}_i^{t+1} - g_i^t \|^2 + \| \bar{g}_{-i}^{t+1} - g_{-i}^t \|^2 + 4 \|\bar{g}_i^{t+1} - g_i^{t+1} \|^2  \right),
\end{align}
where $(*1)$ follows from the fact that $a^{\top}b \leq \| a \| \|b \|$, $(*2)$ follows from Assumption \ref{ass:monotone_smooth_F}, and $(*3)$ follows from $x \cdot y \leq \frac{1}{2} (x^2 + y^2)$. Since this is true for all players $i$, we have:
\begin{align}
\langle \log \bar{z}_{t+1} - \log z_{t+1},  \bar{z}_{t+1} - z_{t+1} \rangle  &\leq \frac{1}{2} \eta L \sum_{i=1}^{N} \left(  \|\bar{g}_i^{t} - g_i^t \|^2 + \| \bar{g}_{-i}^{t} - g_{-i}^t \|^2 +  \|\bar{g}_i^{t+1} - g_i^t \|^2 + \| \bar{g}_{-i}^{t+1} - g_{-i}^t \|^2 + 4 \|\bar{g}_i^{t+1} - g_i^{t+1} \|^2  \right)\nonumber \\
&\leq \frac{1}{2} \eta L \left( N \|\bar{z}_{t} - z_t \|^2  +  N \|\bar{z}_{t+1} - z_t \|^2 + 4 \|\bar{z}_{t+1} - z_{t+1} \|^2  \right) \nonumber \\
&\leq^{*1}  \eta L \left(N \KL ( z_t \| \bar{z}_{t} )  + N  \KL (\bar{z}_{t+1} \| z_t) + 4 \KL(z_{t+1} \| \bar{z}_{t+1}) \right),
\label{eq:similar_OGDA_ref}
\end{align}
where $(*1)$ follows from Pinsker's Inequality and the fact that the $l_1$ norm is an upper bound for the $l_2$ norm.
Substituting this in Equation \eqref{eq:common_1}, we have
\begin{align}
\label{eq:OGDA_call_back}
\KL (z^\star \| {z}_{t+1} ) \leq (1 - \eta \tau) \KL (z^\star \| z_t) & - (1 - \eta \tau - N \eta L)\KL (\bar{z}_{t+1} \| z_t )  - \eta \tau \KL(\bar{z}_{t+1} \| z^\star) \nonumber \\
& \qquad - (1 - 4\eta L) \KL(z_{t+1} \| \bar{z}_{t+1}) + N\eta L \KL ( z_t \| \bar{z}_{t} ).
\end{align}
For $\eta < \frac{1}{2(N+4)L + 2\tau}$, we have: $
N\eta L \leq (1 - \eta \tau)(1 - 4\eta L).$ 
This gives us:
\begin{align}
 \KL (z^\star \| {z}_{t+1} ) +  (1 - 4\eta L) KL(z_{t+1} \| \bar{z}_{t+1}) &\leq (1 - \eta \tau) \KL (z^\star \| z_t) + N\eta L \KL ( z_t \| \bar{z}_{t} ) \nonumber \\
 &\leq  (1 - \eta \tau) \left( \KL (z^\star \| z_t) +  (1 - 4\eta L)  \KL ( z_t \| \bar{z}_{t} ) \right). \nonumber \\
 \end{align}
 Define:
 \begin{align}
 V_t = \KL (z^\star \| z_t) +  (1 - 4\eta L)  \KL ( z_t \| \bar{z}_{t} ).
 \end{align}
 Then we have:
 \begin{align}
 V_{t+1} \leq (1 - \eta \tau) V_t,
 \end{align}
 and therefore:
 \begin{align}
  \KL (z^\star \| {z}_{t+1} )  \leq V_{t+1} \leq (1 - \eta \tau)^t V_0 = (1 - \eta \tau)^t  \KL (z^\star \| {z}_{0} ),
 \end{align}
 which shows convergence of $\KL (z^\star \| {z}_{t+1} )$. Next we show convergence of $\KL (z^\star \| \bar{z}_{t+1} )$. 

Similar to derivation of Equation \eqref{eq:similar_OGDA_ref}, we have
\begin{align}
- \langle z^\star - \bar{z}_{t+1}, \log \bar{z}_{t+1} - \log z_{t+1} \rangle &\leq  \eta L \left(N \KL ( z_t \| \bar{z}_{t} )  + N  \KL (\bar{z}_{t+1} \| z_t) + 4 \KL(z^\star \| \bar{z}_{t+1}) \right).
\end{align}
Substituting this in Equation \eqref{eq:common_2}, we have:
\begin{align}
(1 - 4\eta L) \KL (z^\star \| \bar{z}_{t+1}) \leq \KL(z^\star \| {z}_{t+1}) +  \eta L \left(N \KL ( z_t \| \bar{z}_{t} )  + N  \KL (\bar{z}_{t+1} \| z_t) \right).
\end{align}
Now, using Equation \eqref{eq:OGDA_call_back}, we have:
\begin{align}
(1 - 4\eta L) \KL (z^\star \| \bar{z}_{t+1}) &\leq (1 - \eta \tau) \KL (z^\star \| z_t) - (1 - \eta \tau - 2N \eta L) \KL (\bar{z}_{t+1} \| z_t )  - \eta \tau \KL(\bar{z}_{t+1} \| z^\star) \nonumber \\
& \qquad - (1 - 4\eta L) \KL(z_{t+1} \| \bar{z}_{t+1}) + 2N\eta L \KL ( z_t \| \bar{z}_{t} )\nonumber \\
&\leq  (1 - \eta \tau) \KL (z^\star \| z_t) +  2N\eta L \KL ( z_t \| \bar{z}_{t} ) \nonumber \\
&\leq \KL (z^\star \| z_t) +  (1 - 4\eta L)  \KL ( z_t \| \bar{z}_{t} ) := V_t,
\end{align}
which gives us:
\begin{align}
\KL (z^\star \| \bar{z}_{t+1})  \leq 2V_t \leq 2(1-\eta\tau)^t V_0 = 2(1-\eta\tau)^t \KL(z^\star \| z_0).
\end{align}
This completes the proof of the first part of the theorem.

Next, we prove parameter convergence. We have from Equation \eqref{eq:call_back_para_conv}:
\begin{align}
\| \theta^{t+1}_i - \theta^\star_i \|^2  = (1 - \eta \tau + \eta^2 \tau^2)  \| \theta^t_i - \theta^\star_i\|^2 + \left(1 + \frac{1}{\eta\tau} (1 - \eta \tau)^2 \right)\| \varepsilon_t \|^2,
\end{align}
where $\theta^\star_i = \frac{- \nabla_{g_i} f_i(g_i^\star, g_{-i}^\star)}{\tau}$ (Note that the term $\varepsilon_t$ however is different here from definition of $\varepsilon_t$ in Equation \eqref{eq:call_back_para_conv}).  Now, for ONPG, we have:
\begin{align}
\| \varepsilon_t \|^2 &= \eta^2 \| \nabla_{g_i} f_i(g_i^\star, g_{-i}^\star) - \nabla_{g_i} f_i(\bar{g}_i^{t+1}, \bar{g}_{-i}^{t+1}) \|^2 \leq \eta^2 L^2 \| \bar{z}_{t+1} - z^\star \|^2 \leq 2\eta^2L^2 \KL(z^\star \| z_{t+1}).
\end{align}

This gives us:
\begin{align}
\| \theta^{t+1}_i - \theta^\star_i \|^2 \leq (1 - \eta \tau + \eta^2 \tau^2)  \| \theta^t_i - \theta^\star_i\|^2 + C(1 - \eta \tau)^{t},
\end{align}
where 
\begin{align}
C = 4\eta^2L^2\left(1 + \frac{1}{\eta\tau} (1 - \eta \tau)^2 \right)\KL(z^\star \| z_{0}),
\end{align}
from the first part of the Theorem proved above.
For $\eta\tau < 1/2$, this reduces to:
\begin{align}
\| \theta^{t+1}_i - \theta^\star_i \|^2 \leq (1 - \eta \tau/2)  \| \theta^t_i - \theta^\star_i\|^2 + C(1 - \eta \tau)^{t}.
\end{align}

Now, consider the Lyapunov function:
\begin{align}
V_{t+1} = \| \theta^{t+1}_i - \theta^\star_i \|^2 + \frac{2C}{\eta\tau} (1 - \eta \tau)^{t+1}.
\end{align}
We have:
\begin{align}
V_{t+1} &\leq (1 - \eta \tau/2)  \| \theta^t_i - \theta^\star_i\|^2 + C(1 - \eta \tau)^{t} + \frac{2C}{\eta\tau} (1 - \eta \tau) (1 - \eta \tau)^{t} \nonumber \\
&= (1 - \eta \tau/2)  \| \theta^t_i - \theta^\star_i\|^2   +  \frac{2C}{\eta\tau} (1 - \eta \tau)^{t} (1 - \eta \tau + \eta \tau /2)  = (1 - \eta \tau/2)  \| \theta^t_i - \theta^\star_i\|^2   +  \frac{2C}{\eta\tau}(1 - \eta \tau)^{t} (1 - \eta \tau/2) \nonumber \\
&= (1 - \eta \tau/2) \left( \| \theta^{t}_i - \theta^\star_i \|^2 + \frac{2C}{\eta\tau} (1 - \eta \tau)^{t} \right) = (1 - \eta\tau/2) V_t.
\end{align}
This shows linear convergence of the parameter $\theta_i$ to $\theta^\star_i$ since: 
$\| \theta^{t+1}_i - \theta^\star_i \|^2 \leq V_{t+1} \leq (1 - \eta \tau/2)^t V_0.$ 
Merging these inequalities for all players $i$ completes the proof of the Theorem.
\hfil \qed

\subsection{Proof of Corollary \ref{cor:small_tau_N_player}}

We define $g_i := g_{\theta_i}$. The duality gap for the regularized game is given by:
\begin{align}
\label{eq:ineq_1_DG}
&\text{DG}_\tau(g_1, g_2, \cdots, g_N) = \sum_{i=1}^N \left[ f_i^\tau(g_i,g_{-i}) - \min_{\tilde{g}_i} f_i^\tau(\tilde{g}_i, g_{-i})  \right] = \max_{\tilde{g}_1, \tilde{g}_2, \cdots, \tilde{g}_N} \sum_{i=1}^N \left[ f_i^\tau(g_i,g_{-i}) -  f_i^\tau(\tilde{g}_i, g_{-i})  \right] \nonumber \\
&\quad= \max_{\tilde{g}_1, \tilde{g}_2, \cdots, \tilde{g}_N} \sum_{i=1}^N [ f_i^\tau(g_i,g_{-i}) - f_i^\tau(g_i,g_{-i}^\star) + f_i^\tau(g_i,g_{-i}^\star) - f_i^\tau(\tilde{g}_i,g_{-i}^\star)    +  f_i^\tau(\tilde{g}_i,g_{-i}^\star) - f_i^\tau(\tilde{g}_i, g_{-i})  ] \nonumber \\
&\quad\leq \max_{\tilde{g}_1, \tilde{g}_2, \cdots, \tilde{g}_N} \sum_{i=1}^N [ f_i^\tau(g_i,g_{-i}) - f_i^\tau(g_i,g_{-i}^\star) + f_i^\tau(g_i,g_{-i}^\star) - f_i^\tau(g_i^\star ,g_{-i}^\star)    +  f_i^\tau(\tilde{g}_i,g_{-i}^\star) - f_i^\tau(\tilde{g}_i, g_{-i})  ] .
\end{align}

Next, we note that:
\begin{align}
\label{eq:ineq_2_DG}
\sum_{i=1}^N  f_i^\tau(g_i,g_{-i}) &- f_i^\tau(g_i,g_{-i}^\star)  + f_i^\tau(g_i,g_{-i}^\star) - f_i^\tau(g_i^\star ,g_{-i}^\star) +  f_i^\tau(\tilde{g}_i,g_{-i}^\star) - f_i^\tau(\tilde{g}_i, g_{-i})  \nonumber \\
&\leq \sum_{i=1}^N \left[ \| f_i^\tau(g_i,g_{-i}) - f_i^\tau(g_i,g_{-i}^\star)  \| + \| f_i^\tau(g_i,g_{-i}^\star) - f_i^\tau(g_i^\star ,g_{-i}^\star) \| + \|  f_i^\tau(\tilde{g}_i,g_{-i}^\star) - f_i^\tau(\tilde{g}_i, g_{-i}) \|  \right] \nonumber \\
&\leq \sum_{i=1}^N \left[ \| f_i^\tau(g_i,g_{-i}) - f_i^\tau(g_i,g_{-i}^\star)  \| + \| f_i^\tau(g_i,g_{-i}^\star) - f_i^\tau(g_i^\star ,g_{-i}^\star) \| + \|  f_i^\tau(\tilde{g}_i,g_{-i}^\star) - f_i^\tau(\tilde{g}_i, g_{-i}) \|  \right] \nonumber \\
&\leq C_1 \sum_{i=1}^N \left[ \| g_i - g_i^\star \| + \| g_{-i} - g_{-i}^\star \| \right] \leq C_2N \sqrt{ \KL (z^\star \| z_t) }.
\end{align}
This follows by noting that the functions $f_i^\tau$ are Lipschitz since they are continuous  functions defined on a compact domain. The last step follows from Pinsker's inequality and the fact that the $l_1$ norm is an upper bound for the $l_2$ norm.

Combining the two inequalities \eqref{eq:ineq_1_DG} and \eqref{eq:ineq_2_DG} we have:
\begin{align}
\text{DG}_\tau(g_1, g_2, \cdots, g_N) \leq C_2N \sqrt{ \KL (z^\star \| z_t) }.
\end{align}
Let DG denote the Duality gap of the unregularized problem. Then we have:
\begin{align}
\text{DG}(g_1, g_2, \cdots, g_N) \leq \text{DG}_\tau(g_1, g_2, \cdots, g_N) + 2N \tau \log n.
\end{align}
Therefore, setting $\tau = \frac{\epsilon}{4N \log n}$ and solving the regularized problem to an accuracy of $\frac{\epsilon^2}{4C_2^2N^2}$ in terms of KL divergence, we have that:
\begin{align}
\text{DG}(g_1, g_2, \cdots, g_N) \leq \epsilon,
\end{align} 
completing the proof. \hfil \qed

\subsection{Proximal point and Extragradient methods for multi-player monotone games}
\label{app:sec_EG_PP_proofs}

We define $g_i := g_{\theta_i}$ and $g_i^t := g_{\theta_i^t}$ for simplicity. 

\subsubsection{Proximal-point updates} 

In this subsection, we show the convergence of the Proximal Point (PP) updates to the NE of the regularized N-player monotone game. The PP algorithm is presented in Algorithm \ref{alg:pp_monotone}

\begin{algorithm}[tb]
   \caption{Proximal Point Method}
   \label{alg:pp_monotone}
\begin{algorithmic}
   \STATE {\bfseries Initialize:} $\theta^0_i = 0$ for all players $i$.
   \FOR{$t=1, 2, \cdots$}
     \STATE $\theta_i^{t+1} = (1 - \eta \tau)\theta^t_i  -\eta \nabla_{g_{\theta_i}} f_i(g_{\theta_i^{t+1}}, g_{\theta^{t+1}_i}) \ \ \forall i \in [N].$ 
   \ENDFOR
\end{algorithmic}
\end{algorithm}

\begin{theorem}
\label{prop:pp_conv_mono}
Let $z^\star = (g_i^\star)_{i=1}^N$ be the Nash equilibrium of Problem \eqref{eq:N-payer_monotone_game_reg}. Also, we denote $z_t = (g_{i}^t)_{i=1}^N$. Define $\theta_i^\star := \frac{-\nabla_{g_i} f_i (g_i^\star, g_{-i}^\star)}{\tau}$. Then for updates in Algorithm \ref{alg:pp_monotone}, we have for $\eta\tau < 1/2$:
\vspace{-2mm}
\begin{itemize}
\item $\KL(z^\star \| z_{t+1} )  \leq (1 - \eta \tau)^t \KL (z^\star \| z_0).$
\item $\| \theta_{t+1} - \theta^\star \|^2  \leq (1 - \eta \tau/2)^t V_0$, where, 
\begin{align}
V_0 =\| \theta^{t} - \theta^\star \|^2 + \frac{2NC}{\eta\tau}, \qquad \qquad 
C = 2\eta^2L^2\left(1 + \frac{4}{\eta\tau} (1 - \eta \tau)^2 \right)\KL(z^\star \| z_{0}). \nonumber
\end{align}
\end{itemize}
\end{theorem}

\begin{proof}

From the definiton of the KL divergence we have:
\begin{align}
- \langle \log z_{t+1} - (1 - \eta \tau) \log z_t - \eta \tau \log z^\star, z^\star  \rangle = - (1 - \eta \tau) \KL (z^\star \| z_t) + \KL(z^\star \| z_{t+1} ),
\end{align}
and 
\begin{align}
\langle \log z_{t+1} - (1 - \eta \tau) \log z_t - \eta \tau \log z^\star, z_{t+1}  \rangle = (1 - \eta \tau) \KL (z_{t+1} \| z_t) + \KL(z_{t+1} \| z^\star ).
\end{align}
Substituting in Lemma \ref{Lemma_1} with $z = z_{t+1}$, we have:
\begin{align}
\langle \log z_{t+1} - (1 - \eta \tau) \log z_t - \eta \tau \log z^\star, z_{t+1} - z^\star  \rangle \leq 0,
\end{align}
and using the two equalities, we get:
\begin{align}
- (1 - \eta \tau) \KL (z^\star \| z_t) + \KL(z^\star \| z_{t+1} ) + (1 - \eta \tau) \KL (z_{t+1} \| z_t) + \KL(z_{t+1} \| z^\star ) \leq 0.
\end{align}
This implies:
\begin{align}
\KL(z^\star \| z_{t+1} )  \leq (1 - \eta \tau) \KL (z^\star \| z_t).
\end{align}
This shows linear convergence of the KL divergence to the Nash equilibrium for the proximal point method,  which completes the proof of the first part of the Theorem.

When the agents are playing the NE strategy, the updates reduce to
\begin{align}
\theta^{t+1}_i &= (1-\eta \tau) \theta^t_i - \eta \nabla_{g_i} f_i(g_i^\star, g_{-i}^\star). 
\end{align}

From Lemma \ref{Lemma_1}, we know that the NE satisfy:
\begin{align}
g_i^\star(a) = \frac{e^{\frac{- [\nabla_{g_i} f_i(g_i^\star, g_{-i}^\star)]_a}{\tau}} }{K},
\end{align} 
where $K = \sum_{a' \in \mathcal{A}_i} e^{\frac{- [\nabla_{g_i} f_i(g_i^\star, g_{-i}^\star)]_{a'}}{\tau}}$. Taking $\log$ on both sides, we have:
\begin{align}
 -\eta \nabla_{g_i} f_i(g_i^\star, g_{-i}^\star) = \eta \tau \log g_i^\star + \eta \tau \log K.
\end{align}
Substituting this back into the $\theta$ update, we have:
\begin{align}
\theta^{t+1}_i &= (1-\eta \tau) \theta^t_i +  \eta \tau \log g_i^\star + \eta \tau \log K =  \theta^t_i - \eta \tau \theta^t_i +  \eta \tau \log g_i^\star +  \eta \tau \log K- \eta \tau \log Z_{\theta^t_i} + \eta \tau \log Z_{\theta^t_i} \nonumber \\
&= \theta^t_i   +  \eta \tau \log g_i^\star - \eta \tau \log \left( \frac{e^{\theta^t_i}}{Z_{\theta^t_i}} \right) +  \eta \tau \log \left(  \frac{K}{Z_{\theta^t_i}} \right) = \theta^t_i   +  \eta \tau \log g_i^\star - \eta \tau \log g_{\theta^t_i}  +  \eta \tau \log \left(  \frac{K}{Z_{\theta^t_i}} \right) \nonumber \\
&= \theta^t_i   + \eta \tau \log \left( \frac{g_i^\star}{g_{\theta^t_i}} \right)  +  \eta \tau \log \left(  \frac{K}{Z_{\theta^t_i}} \right) =  \theta^t_i   +  \eta \tau \log \left(  \frac{ g_i^\star K }{ g_{\theta^t_i} Z_{\theta^t_i} } \right).
\end{align}
We can further simplify this as:
\begin{align}
\theta^{t+1}_i &= \theta^t_i + \eta \tau \log \left(  \frac{e^{- [\nabla_{g_i} f_i(g_i^\star, g_{-i}^\star)]/\tau}}{e^{ \theta^t_i}} \right). 
\end{align}
This is nothing but:
\begin{align}
\theta^{t+1}_i &= \theta^t_i + \eta \tau \left( \frac{- [\nabla_{g_i} f_i(g_i^\star, g_{-i}^\star)]}{\tau} - \theta^t_i \right).
\end{align}

Now the Proximal Point updates can be written as:

\begin{align}
\theta^{t+1}_i &= \theta^t_i + \eta \tau \left( \frac{- \nabla_{g_i} f_i(g_i^\star, g_{-i}^\star)}{\tau} - \theta^t_i \right) + \varepsilon_t,
\end{align}
where $\varepsilon_t = \eta \nabla_{g_i} f_i(g_i^\star, g_{-i}^\star) - \eta \nabla_{g_i} f_i(g_i^{t+1}, g_{-i}^{t+1})$.

Defining $\theta^\star_i = \frac{- \nabla_{g_i} f_i(g_i^\star, g_{-i}^\star)}{\tau}$, we have:
\begin{align}
\| \theta^{t+1}_i - \theta^\star_i \|^2 &= \| \theta^t_i + \eta \tau \left( \theta^\star_i - \theta^t_i \right) + \varepsilon_t  - \theta^\star_i \|^2 = \| (1 - \eta \tau) (\theta^t_i - \theta^\star_i) + \varepsilon_t \|^2 \nonumber \\
&= (1 - \eta \tau)^2 \| \theta^t_i - \theta^\star_i \|^2 + 2(1 - \eta \tau) (\theta^t_i - \theta^\star_i)^{\top} \varepsilon_t + \| \varepsilon_t \|^2 \nonumber \\
&\leq^{*1}  (1 - \eta \tau)^2 \| \theta^t_i - \theta^\star_i\|^2 + \eta \tau \| \theta^t_i - \theta^\star_i\|^2 + \frac{4}{\eta\tau} (1 - \eta \tau)^2  \| \varepsilon_t \|^2 +  \| \varepsilon_t \|^2 \nonumber \\
&= (1 - \eta \tau + \eta^2 \tau^2)  \| \theta^t_i - \theta^\star_i\|^2 + \left(1 + \frac{4}{\eta\tau} (1 - \eta \tau)^2 \right)\| \varepsilon_t \|^2,
\end{align}
where $*1$ follows from Young's inequality. \\

Now, for the proximal point methods, we have:
\begin{align}
\|\varepsilon_t\|^2 &= \eta^2\| \nabla_{g_i} f_i(g_i^\star, g_{-i}^\star) - \nabla_{g_i} f_i(g_i^{t+1}, g_{-i}^{t+1}) \|^2 \leq \eta^2 L^2 \| z_{t+1} - z^\star \|^2 \leq 2\eta^2L^2 \KL(z^\star \| z_{t+1}),
\end{align}
where it follows from Pinsker's inequality and the fact that the $l_1$ norm is an upper bound for the $l_2$ norm.

This gives us:
\begin{align}
\| \theta^{t+1}_i - \theta^\star_i \|^2 \leq (1 - \eta \tau + \eta^2 \tau^2)  \| \theta^t_i - \theta^\star_i\|^2 + C(1 - \eta \tau)^{t},
\end{align}
where 
\begin{align}
C = 2\eta^2L^2\left(1 + \frac{4}{\eta\tau} (1 - \eta \tau)^2 \right)\KL(z^\star \| z_{0}).
\end{align}
For $\eta\tau < 1/2$, this reduces to:
\begin{align}
\| \theta^{t+1}_i - \theta^\star_i \|^2 \leq (1 - \eta \tau/2)  \| \theta^t_i - \theta^\star_i\|^2 + C(1 - \eta \tau)^{t}.
\end{align}

Now, consider the Lyapunov function:
\begin{align}
V_{t+1} = \| \theta^{t+1}_i - \theta^\star_i \|^2 + \frac{2C}{\eta\tau} (1 - \eta \tau)^{t+1}.
\end{align}
We have:
\begin{align}
V_{t+1} &\leq (1 - \eta \tau/2)  \| \theta^t_i - \theta^\star_i\|^2 + C(1 - \eta \tau)^{t} + \frac{2C}{\eta\tau} (1 - \eta \tau) (1 - \eta \tau)^{t} = (1 - \eta \tau/2)  \| \theta^t_i - \theta^\star_i\|^2   +  \frac{2C}{\eta\tau} (1 - \eta \tau)^{t} (1 - \eta \tau + \eta \tau /2) \nonumber \\
& = (1 - \eta \tau/2)  \| \theta^t_i - \theta^\star_i\|^2   +  \frac{2C}{\eta\tau}(1 - \eta \tau)^{t} (1 - \eta \tau/2) = (1 - \eta \tau/2) \left( \| \theta^{t}_i - \theta^\star_i \|^2 + \frac{2C}{\eta\tau} (1 - \eta \tau)^{t} \right) \nonumber \\
&= (1 - \eta\tau/2) V_t.
\end{align}
This shows linear convergence of the parameter $\theta_i$ to $\theta^\star_i$ since:
\begin{align}
\| \theta^{t+1}_i - \theta^\star_i \|^2 \leq V_{t+1} \leq (1 - \eta \tau/2)^t V_0.
\end{align}
Merging these inequalities for all players $i$ completes the proof of the theorem.
\end{proof}

\subsubsection{Extragradient updates} \label{subsec:eg_mono_app}
In this subsection, we show the convergence of the Extragradient   method to the NE of the regularized N-player monotone game. The EG algorithm is presented in Algorithm \ref{alg:eg_monotone}. 

\begin{algorithm}[tb]
   \caption{Extragradient Method}
   \label{alg:eg_monotone}
\begin{algorithmic}
   \STATE {\bfseries Initialize:} $g_0$ and $h_0$.
   \FOR{$t=1, 2, \cdots$}
   \STATE $\bar{\theta}_i^{t+1} = (1 - \eta \tau)\theta_i^t -\eta \nabla_{g_{\theta_i}} f_i(g_{{\theta}^t_i}, g_{{\theta}^t_{-i}}) \ \ \ \ \ \ \forall i \in [N]$.
   \STATE $\theta_i^{t+1} = (1 - \eta \tau)\theta_i^t -\eta \nabla_{g_{\theta_i}} f_i(g_{\bar{\theta}^{t+1}_i}, g_{\bar{\theta}^{t+1}_{-i}}) \ \ \forall i \in [N].$
   \ENDFOR
\end{algorithmic}
\end{algorithm}

\begin{theorem}
\label{thm:eg_conv_mono}
Let $z^\star = (g_i^\star)_{i=1}^N$ be the unique Nash equilibrium of Problem \eqref{eq:N-payer_monotone_game_reg}. 
Also, we denote $z_t = (g_{\theta^t_i})_{i=1}^N$. Then for updates in Algorithm \ref{alg:eg_monotone}, we have for stepsize satisfying $0 < \eta <  \frac{1}{2NL + \tau}$:
\vspace{-2mm}
\begin{itemize}
\item Convergence of distributions:
\begin{align}
\max \left\{ \KL(z^\star \| z_{t} ), \KL(z^\star \| \bar{z}_{t+1} ) \right\} \leq (1 - \eta \tau)^t 2 \KL (z^\star \| z_0).
\end{align}
\item Convergence of parameters:
\begin{align}
\| \theta_{t+1} - \theta^\star \|^2  \leq (1 - \eta \tau/2)^t V_0,
\end{align} 
where, 
\begin{align}
V_0 =\| \theta^{t} - \theta^\star \|^2 + \frac{2NC}{\eta\tau}, \qquad\qquad \qquad  C = 4\eta^2L^2\left(1 + \frac{4}{\eta\tau} (1 - \eta \tau)^2 \right) \KL(z^\star \| z_{0}). \nonumber
\end{align}
\end{itemize}
\end{theorem}

\begin{proof}
From the EG updates in Algorithm \ref{alg:eg_monotone}, we have:
\begin{align}
\log \bar{g}_i^{t+1} - \log g_i^{t+1} = -\eta \left(\nabla_{g_i} f_i({g}_i^{t}, {g}_{-i}^{t}) - \nabla_{g_i} f_i(\bar{g}_i^{t+1}, \bar{g}_{-i}^{t+1}) \right) + c\cdot  \mathbbm{1}.
\end{align}
This implies:
\begin{align}
\langle \log \bar{g}_i^{t+1} - \log g_i^{t+1}, \bar{g}^{t+1}_i - g_{t+1}^i  \rangle &= - \eta \langle  \nabla_{g_i} f_i({g}_i^{t}, {g}_{-i}^{t}) - \nabla_{g_i} f_i(\bar{g}_i^{t+1}, \bar{g}_{-i}^{t+1}), \bar{g}^{t+1}_i - g^{t+1}_i \rangle \nonumber \\
&\leq^{*1} \eta \| \nabla_{g_i} f_i({g}_i^{t}, {g}_{-i}^{t}) - \nabla_{g_i} f_i(\bar{g}_i^{t+1}, \bar{g}_{-i}^{t+1}) \|\cdot \| \bar{g}^{t+1}_i - g^{t+1}_i  \| \nonumber \\
&\leq^{*2} \eta L \left( \|\bar{g}^{t+1}_i - g^{t}_i  \| + \| \bar{g}^{t+1}_{-i} - g^{t}_{-i} \| \right) \|\bar{g}^{t+1}_i - g^{t+1}_i \| \nonumber \\
&\leq^{*3} \frac{1}{2} \eta L \left( \|\bar{g}^{t+1}_i - g^{t}_i \|^2 + 2 \|\bar{g}^{t+1}_i - g^{t+1}_i \|^2 + \| \bar{g}^{t+1}_{-i} - g^{t}_{-i} \|^2 \right),
\end{align}
where $(*1)$ follows from the fact that $a^{\top}b \leq \| a \| \|b \|$, $(*2)$ follows from Assumption \ref{ass:monotone_smooth_F}, and $(*3)$ follows from $x \cdot y \leq \frac{1}{2} (x^2 + y^2)$. Since this is true for all players, we have:
\begin{align}
\label{eq:similar_eg_ref}
\langle \log \bar{z}_{t+1} - \log z_{t+1}, \bar{z}_{t+1} - z_{t+1} \rangle &\leq  \frac{1}{2} \eta L \sum_{i=1}^N \left( \|\bar{g}^{t+1}_i - g^{t}_i \|^2 + 2 \|\bar{g}^{t+1}_i - g^{t+1}_i \|^2 + \| \bar{g}^{t+1}_{-i} - g^{t}_{-i} \|^2 \right) \nonumber \\
&\leq \frac{1}{2} \eta L \left( N \|\bar{z}_{t+1} - z_{t} \|^2 + 2 \|\bar{z}_{t+1} - z_{t+1} \|^2 \right) \leq \frac{N}{2} \eta L \left( \|\bar{z}_{t+1} - z_{t} \|^2 +  \|\bar{z}_{t+1} - z_{t+1} \|^2 \right) \nonumber \\
&\leq^{*1} N\eta L \left( \KL (\bar{z}_{t+1} \| z_t) + \KL(z_{t+1} \| \bar{z}_{t+1}) \right),
\end{align}
where $(*1)$ follows from Pinsker's Inequality and the fact that the $l_1$ norm is an upper bound for the $l_2$ norm.
Substituting this in Equation \eqref{eq:common_1}, we have
\begin{align}
\label{eq:callback_eg_almost}
\KL (z^\star \| {z}_{t+1} ) &\leq (1 - \eta \tau) \KL (z^\star \| z_t) - (1 - \eta \tau - N \eta L) \KL (\bar{z}_{t+1} \| z_t )  \nonumber \\
& \qquad - \eta \tau \KL(\bar{z}_{t+1} \| z^\star) - (1 - N\eta L) \KL(z_{t+1} \| \bar{z}_{t+1}).
\end{align}
If $\eta \leq \frac{1}{\tau + NL}$, we have:
\begin{align}
\KL (z^\star \| {z}_{t+1} ) &\leq (1 - \eta \tau) \KL (z^\star \| z_t).
\end{align}

Similar to the derivation of Equation \eqref{eq:similar_eg_ref}, we have:
\begin{align}
- \langle \log \bar{z}_{t+1} - \log z_{t+1},z^\star -  \bar{z}_{t+1} \rangle \leq N\eta L \left( \KL (\bar{z}_{t+1} \| z_t) + \KL(z^\star \| \bar{z}_{t+1}) \right).
\end{align}
Substituting this in Equation \eqref{eq:common_2} we have:
\begin{align}
\label{eq:ughh_eg}
(1 - N\eta L) \KL(z^\star \| \bar{z}_{t+1}) \leq \KL (z^\star \| z_{t+1} ) +  N\eta L\KL (\bar{z}_{t+1} \| z_t). 
\end{align}
Now, plugging Inequality \eqref{eq:callback_eg_almost} in \eqref{eq:ughh_eg} we have:
\begin{align}
&(1 - N\eta L) \KL(z^\star \| \bar{z}_{t+1}) \leq (1 - \eta \tau)  \KL (z^\star \| z_t) \nonumber \\
&\qquad- (1 - \eta \tau - 2N \eta L) \KL (\bar{z}_{t+1} \| z_t )  - \eta \tau \KL(\bar{z}_{t+1} \| z^\star) - (1 - N\eta L) \KL(z_{t+1} \| \bar{z}_{t+1}).
\end{align}
With stepsize $\eta < \frac{1}{\tau + 2N L}$ we have:
\begin{align}
\KL(z^\star \| \bar{z}_{t+1}) \leq 2  \KL (z^\star \| z_t) \leq 2(1 - \eta \tau)^t \KL(z^\star \| z_0),
\end{align}
which completes the first part of the proof.

The proof of parameter convergence follows exactly from the proof of Theorem  \ref{thm:eg_og_conv_mono} and we avoid rewriting it here.
\end{proof}

\subsection{Parameterization with function approximation}
\label{appendix:monotone_FA}


In this section, we discuss the monotone game setting with function approximation for policy parameterization, as discussed for matrix games in \S\ref{sec:ONPG_Matrix_FA}. 
In this setting, the regularized problem each player $i$ faces is:
\begin{align} 
\label{eq:mono_func_main}
\min_{\theta_i \in \mathbb{R}^d}~~~f_i(g_{\theta_i}, g_{\theta_{-i}})  - \tau \mathcal{H}(g_{\theta_i}),
\end{align}
where $g_{\theta_i}$ is a log-linear policy parametrization. In the next lemma, we show the existence of a NE in this setting as well as the equivalence of this problem to one on the entire simplex, in the same spirit as  in \S\ref{sec:ONPG_Matrix_FA} for matrix games:

\begin{lemma}
\label{lemma:func_approx_monotone}
The in-class NE (Definition \ref{def:para_NE_mono_in_class}) for the {unregularized (and regularized)} monotone game under the log-linear policy parametrization exists. Also, under Assumption \ref{ass:full_rank}, solving Problem \eqref{eq:mono_func_main} for all $i$ is equivalent to:
\begin{align} 
\min_{g_{\theta_i} \in {\Delta}}~~~f_i(\Psi g_{\theta_i}, \Psi g_{\theta_{-i}})  - \tau \mathcal{H}(g_{\theta_i}),
\end{align}
where $\Psi$ is defined in Equation \eqref{eq:psi_def}.
\end{lemma}
Now, using Lemma \ref{lemma:func_approx_monotone}, and Proposition \ref{eq:lemma:algo_lin_func_approx}, we describe an algorithm to solve this problem in the following corollary. 
\begin{corollary}
The update rule: 
\begin{align}
\bar{\theta}_i^{t+1} &= (1 - \eta \tau)\theta_i^t -[(M^\top)^{-1} | 0] \eta  \tilde{P} \nabla_{g_{\theta_i}} f_i(g_{\bar{\theta}^t_i}, g_{\bar{\theta}^t_{-i}}) \nonumber \\
{\theta}_i^{t+1} &= (1 - \eta \tau)\theta_i^t - [(M^\top)^{-1} | 0] \eta  \tilde{P}  \nabla_{g_{\theta_i}} f_i(g_{\bar{\theta}^{t+1}_i}, g_{\bar{\theta}^{t+1}_{-i}}), \nonumber
\end{align}
solves Problem \eqref{eq:mono_func_main} with similar guarantees given by Theorem \ref{thm:eg_og_conv_mono}. Here, the NE parameter value to which the algorithm converges to is given by $\theta^\star_i = \frac{- [(M^\top)^{-1} | 0]  \tilde{P} \nabla_{g_{\theta_i}} f_i(g_{i}^\star, g_{{-i}}^\star)}{\tau} $. Furthermore, by choosing the regularization parameter $\tau$ small enough, like in Corollary \ref{cor:small_tau_N_player}, we reach an $\epsilon$-in-class NE (Definition \ref{def:para_eps_NE_mono}) of the unregularized monotone game under the function approximation setting. 
\end{corollary}

\subsection{Proof of Lemma \ref{lemma:func_approx_monotone}}
The proof of this Lemma follows along the lines of Lemma \ref{lemma:char_func_approx} and Theorem \ref{thm:Nash_char_func_approx}. 
The key here is to notice that:
\begin{align}\label{equ:lemma_func_approx_monotone_1}
\nabla_{g_{\theta_i}} f_i ( \Psi g_{\theta_i}, \Psi g_{\theta_{-1}} ) = \Psi^{\top} \nabla_{\Psi g_{\theta_i}} f_i ( \Psi g_{\theta_i}, \Psi g_{\theta_{-1}} ), 
\end{align}
and that: $\Psi \mu = \mu,~~\forall \mu \in \tilde{\Delta}$. 
The rest of the proof is identical to the proof of Lemma \ref{lemma:char_func_approx}.
\hfil \qed

\section{Missing Definitions and Proofs in \S\ref{sec:ONPG_Markov}}\label{sec:append_missing_ONPG_Markov}

\begin{remark}
We note that all results presented  in this section also follow for the case where the cardinality of the  action spaces for both players are asymmetric. However, we stick to the case where the number of action is the same for both players in all states for ease of exposition.
\end{remark}


\subsection{Proof of Theorem \ref{thm:Markov_game_convergence}}
\label{appendix:thm:Markov_game_convergence}

\begin{algorithm}[tb]
   \caption{Optimistic NPG for Markov Games}
   \label{alg:markov_game}
\begin{algorithmic}
   \STATE {\bfseries Initialize:} $Q_0 = 0$
   \FOR{$t=1, 2, \cdots, T_{outer}$}
   \FOR{$s = 1, 2, \cdots, |\cS|$}
     \STATE Let $Q_t(s, a, b) := r(s, a, b) + \gamma \mathbb{E}_{s' \sim \mathbb{P}(\cdot|s, a, b)} [ V_t (s')]$
     \STATE Solve $\min_{{\theta} \in \mathbb{R}^n} \ \max_{{\nu}  \in \mathbb{R}^n } f_{\tau} (Q(s); g_\theta,h_\nu)$ by running the Optimistic NPG algorithm (Algorithm \ref{alg:optimistic_npg}) for $T_{inner}$ iterations and return the last iterates $(\theta_{T_{inner}}, \nu_{T_{inner}})$.
     \STATE Set $V_{t+1}(s) = f_{\tau} (Q_t(s); g_{\theta_{T_{inner}}}(\cdot\given s), h_{\nu_{T_{inner}}}(\cdot\given s))$
     \ENDFOR
   \ENDFOR
\end{algorithmic}
\end{algorithm}

We first have the following lemma which shows the smoothness property of the NE policy with respect to the game matrix $Q$.
\begin{lemma}
\label{lemma:smooth_wrtQ}
Consider the following entropy regularized game:
\begin{align}\label{equ:append_tilde_delta_minimax}
\min_{x \in \tilde{\Delta}} \ \max_{y \in \tilde{\Delta}} ~~~\ x^\top Q y - \tau \cH (x) + \tau \cH(y),
\end{align}
where $\tilde{\Delta} \subseteq \Delta$ is a convex compact subset of the simplex given by Equation \eqref{eq:Delta_tilde_def}. Let $(x^\star_{Q}, y^\star_Q)$ denote the unique solution to this problem (note that this is unique since we have a  strongly convex-strongly concave objective over a compact convex set). 
Then, we have: 
\begin{align}
\max \big\{ \| x^\star_{Q_1} - x^\star_{Q_2} \|, \| y^\star_{Q_1} - y^\star_{Q_2} \| \big\} \leq C \cdot\|Q_1 - Q_2\|_F,
\end{align}
for some constant $C > 0$ and for any  $Q_1, Q_2 \in \mathbb{R}^{n \times n}$.
\end{lemma}
\begin{proof}
First notice that by the proof of Theorem \ref{thm:Nash_char_func_approx}, solving \eqref{equ:append_tilde_delta_minimax}  is equivalent to solving 
\#\label{equ:append_delta_minimax}
\min_{x\in {\Delta}} \ \max_{y  \in {\Delta} } ~~~\ x^\top \Psi^\top Q \Psi y - \tau \cH (x) + \tau \cH(y),
\#
with $\Psi$ being defined in Equation \eqref{eq:psi_def}, which admits a unique solution. In other words, the solution  $(x^\star_{Q}, y^\star_Q)$ also solves \eqref{equ:append_delta_minimax}.  
Also, notice that the solution to \eqref{equ:append_delta_minimax} always lies in the relative interior of $\Delta$, given by \eqref{equ:optimal_solution_linear_FA}. In other words, $[x^\star_{Q}]_a>0$ and $[y^\star_{Q}]_b>0$ for all $a$ and $b$. Due to the simplex constraint, the free variable is of dimension $n-1$, and the last dimension of $x$ can be represented as $1-\sum_{a=1}^{n-1}x_a>0$ (similarly for $y$). Let $\tilde x=(x_1,x_2,\cdots,x_{n-1})^\top$, $\tilde y=(y_1,y_2,\cdots,y_{n-1})^\top$, and $f(x,y):=x^\top \Psi^\top Q\Psi y - \tau \cH (x) + \tau \cH(y)$. Recall that $f(x,y)$ is strongly convex in $x$ and strongly concave in $y$.  Note that $f(x,y)=f\big(\Lambda(\tilde x),\Lambda(\tilde y)\big)=\tilde f(\tilde x,\tilde y):=\Lambda(\tilde x)^\top \Psi^\top Q\Psi  \Lambda(\tilde y) - \tau \cH \big(\Lambda (\tilde x)\big) + \tau \cH\big(\Lambda (\tilde y)\big)$, where $x=\Lambda(\tilde x)=\begin{bmatrix}
	I \\ -\mathbbm{1}^\top
\end{bmatrix}\tilde x+\begin{bmatrix}
0\\0\\\vdots\\1
\end{bmatrix}$, and $\mathbbm{1}$ denotes an all-one vector of proper dimension. Note that $\tilde f(\tilde x,\tilde y)$ is also strongly convex in $\tilde x$ and strongly concave in $\tilde y$, as for any $\tilde y$, the Hessian $\nabla^2_{\tilde x} \tilde f(\tilde x,\tilde y)=\begin{bmatrix}
	I \ \big| \ -\mathbbm{1}
\end{bmatrix}\nabla^2_x  f(\Lambda(\tilde x),\Lambda(\tilde y))\begin{bmatrix}
	I \\ -\mathbbm{1}^\top
\end{bmatrix}\succ 0$, so is the Hessian with respect to $\tilde y$ for any $\tilde x$. Hence, the solution to the  minimax problem 
\#\label{equ:surrogate_low_dim}
\min_{\big\{\tilde x\biggiven\tilde x_a\geq 0,1-\sum_{a=1}^{n-1}\tilde x_a\geq 0\big\}}\max_{\big\{\tilde y\biggiven\tilde y_b\geq 0,1-\sum_{b=1}^{n-1}\tilde y_b\geq 0\big\}}\quad\tilde f(\tilde x,\tilde y) 
\# 
is given by $(\tilde x^\star_Q,\tilde y^\star_Q)$, where $\tilde x^\star_Q$ and $\tilde y^\star_Q$ are just the first $n-1$ dimensions of $x^\star_Q$ and $y^\star_Q$, satisfying $[\tilde x^\star_{Q}]_a>0$, $[\tilde y^\star_{Q}]_b>0$, and $1-\sum_{a=1}^{n-1}[\tilde x^\star_{Q}]_a>0$, $1-\sum_{b=1}^{n-1}[\tilde y^\star_{Q}]_b>0$, i.e., the constraints in \eqref{equ:surrogate_low_dim} are not violated at $(\tilde x^\star_Q,\tilde y^\star_Q)$. By KKT conditions, it holds that at $(\tilde x^\star_Q,\tilde y^\star_Q)$ 
\#\label{equ:KKT_low_dim}
\tau\nabla_{\tilde x}\cH(\Lambda(\tilde x^\star_Q))-\begin{bmatrix}
	I \ \big| \ -\mathbbm{1}
\end{bmatrix}\Psi^\top Q\Psi  \Lambda(\tilde y^\star_Q)&=0\\
\tau\nabla_{\tilde y}\cH(\Lambda(\tilde y^\star_Q))+\begin{bmatrix}
	I \ \big| \ -\mathbbm{1}
\end{bmatrix}\Psi^\top Q^\top\Psi  \Lambda(\tilde x^\star_Q)&=0. 
\#
Define operator $\cG\big(\tilde x,\tilde y,\text{vec}(Q)\big)$ as
\$
\cG\big(\tilde x,\tilde y,\text{vec}(Q)\big):=\begin{bmatrix}
	\tau\nabla_{\tilde x}\cH(\Lambda(\tilde x))-\begin{bmatrix}
	I \ \big| \ -\mathbbm{1}
\end{bmatrix}\Psi^\top Q\Psi  \Lambda(\tilde y)\\
\tau\nabla_{\tilde y}\cH(\Lambda(\tilde y))+\begin{bmatrix}
	I \ \big| \ -\mathbbm{1}
\end{bmatrix}\Psi^\top Q^\top\Psi  \Lambda(\tilde x)
\end{bmatrix},
\$
where $\text{vec}(Q)$ is the vectorization of $Q$.   
Then, $(\tilde x^\star_Q,\tilde y^\star_Q)$ is given by the solution to $\cG\big(\tilde x,\tilde y,\text{vec}(Q)\big)=0$. Note that the Jacobian of $\cG$ with respect to $[\tilde x^\top,\tilde y^\top]^\top$ is 
\$
\cM\big(\tilde x,\tilde y,\text{vec}(Q)\big):=\begin{bmatrix}
	\frac{\partial \cG}{\partial \tilde x}\ \big| \ \frac{\partial \cG}{\partial \tilde y}
\end{bmatrix}=\begin{bmatrix}
	\tau\nabla^2_{\tilde x} \cH(\Lambda(\tilde x)) & -\begin{bmatrix}
	I \ \big| \ -\mathbbm{1}
\end{bmatrix}\Psi^\top Q\Psi  \begin{bmatrix}
	I \\ -\mathbbm{1}^\top
\end{bmatrix}\\
\begin{bmatrix}
	I \ \big| \ -\mathbbm{1}
\end{bmatrix}\Psi^\top Q^\top\Psi  \begin{bmatrix}
	I \\ -\mathbbm{1}^\top
\end{bmatrix} & \tau\nabla^2_{\tilde y} \cH(\Lambda(\tilde y))
	\end{bmatrix},
\$
which is always invertible for any $\tilde x$ and $\tilde y$ belonging to the constraints in \eqref{equ:surrogate_low_dim}, due to the fact that $\tau\nabla^2_{\tilde x} \cH(\Lambda(\tilde x)),\tau\nabla^2_{\tilde y} \cH(\Lambda(\tilde y))\succ 0$, and $\cM\big(\tilde x,\tilde y,\text{vec}(Q)\big)$ is skew-symmetric, yielding the fact that the real parts of the eigenvalues  of $\cM\big(\tilde x,\tilde y,\text{vec}(Q)\big)$, which are the eigenvalues of  $(\cM^\top+\cM)/2$,  are
always positive. In fact, the real parts are  uniformly lower bounded by some constant $\eta>0$ for any $\tilde x$ and $\tilde y$ belong to the constraints in \eqref{equ:surrogate_low_dim}, due to the strong convexity of $\cH(\Lambda(\tilde x))$ and $\cH(\Lambda(\tilde y))$. Hence, we have
\$
\big\|\cM\big(\tilde x,\tilde y,\text{vec}(Q)\big)^{-1}\big\|_2\leq \frac{2}{\lambda_{\min}\Big(\cM\big(\tilde x,\tilde y,\text{vec}(Q)\big)+\cM\big(\tilde x,\tilde y,\text{vec}(Q)\big)^\top\Big)}\leq \frac{1}{\eta}.
\$
Due to the invertibility of $\cM\big(\tilde x,\tilde y,\text{vec}(Q)\big)$, we can apply implicit function theorem \citep{krantz2012implicit} for any solution to $\cG\big(\tilde x,\tilde y,\text{vec}(Q)\big)=0$, and obtain that for any  such a solution $(\tilde x^\star_{Q},\tilde y^\star_Q,\text{vec}(Q))$, there  exists a neighborhood of it such that for any $(\tilde x,\tilde y,\text{vec}(\tilde Q))$ in the neighborhood
\$
\frac{\partial[\tilde x^\top,\tilde y^\top]^\top}{\partial \text{vec}(\tilde Q)}=-\cM\big(\tilde x,\tilde y,\text{vec}(\tilde Q)\big)^{-1}\cdot \frac{\partial \cG\big(\tilde x,\tilde y,\text{vec}(\tilde Q)\big)}{\partial \text{vec}(\tilde Q)}.  
\$
Notice that $\frac{\partial \cG(\tilde x,\tilde y,\text{vec}(\tilde Q))}{\partial \text{vec}(\tilde Q)}$ is uniformly bounded in norm on the constrained sets in \eqref{equ:surrogate_low_dim}, due to the boundedness of the sets. Hence, there exists a uniform constant $C'>0$ such that 
\$
\bigg\|\frac{\partial[(\tilde x^\star_{Q})^\top,~ (\tilde y^\star_{Q})^\top]^\top}{\partial \text{vec}(Q)}\bigg\|_2\leq \Big\|\cM\big(\tilde x^\star_Q,\tilde y^\star_Q,\text{vec}(Q)\big)^{-1}\Big\|_2\cdot \bigg\| \frac{\partial \cG\big(\tilde x^\star_Q,\tilde y^\star_Q,\text{vec}(Q)\big)}{\partial \text{vec}(\tilde Q)}\bigg\|_2\leq C'
\$ 
for any $(\tilde x^\star_{Q},\tilde y^\star_Q,\text{vec}(Q))$. By the mean-value theorem, we know that
\$
\big\|[(\tilde x^\star_{Q_1})^\top,~ (\tilde y^\star_{Q_1})^\top]-[(\tilde x^\star_{Q_2})^\top,~ (\tilde y^\star_{Q_2})^\top]\big\|_2 \leq C' \cdot \big\|\text{vec}(Q_1)-\text{vec}(Q_2)\big\|_2. 
\$
Finally, notice that 
\$
\big\|[(x^\star_{Q_1})^\top,~ (y^\star_{Q_1})^\top]-[(x^\star_{Q_2})^\top,~ (y^\star_{Q_2})^\top]\big\|_2\leq \Bigg\|\begin{bmatrix}
	I \\ -\mathbbm{1}^\top
\end{bmatrix}
\Bigg\|_2\cdot \big\|[(\tilde x^\star_{Q_1})^\top,~ (\tilde y^\star_{Q_1})^\top]-[(\tilde x^\star_{Q_2})^\top,~ (\tilde y^\star_{Q_2})^\top]\big\|_2,
\$
which completes the proof by the equivalence of norms.
\end{proof}

\subsubsection{Tabular case and proof of Theorem \ref{thm:Markov_game_convergence}}

We set the stepsize to be:
\begin{align}
\eta = \frac{1-\gamma}{2(1 + \tau(\log n+ 1 - \gamma))}. 
\end{align}

The convergence of the $Q$-function follows from Lemma \ref{lemma:distribution_con_cen}, along with Theorem 2 in \cite{cen2021fast}. In particular, it is shown that when the inner problem is solved upto an accuracy of $\epsilon$, we have:
\begin{align}
\label{eq:ref_Q_conv_cen}
\| Q_t - Q^\star \|_\infty \leq \epsilon + \gamma^t \| Q_0 - Q^\star \|_\infty
\end{align}


We show the convergence of the parameter next. We define:
\begin{align}
\label{def:optimal_para_MG_tab}
\theta^\star(s) = \frac{- Q^\star_\tau h_{\nu^\star}(\cdot\given s)}{\tau},\qquad\qquad 
\nu^\star(s) = \frac{{Q^\star_\tau}^\top g_{\theta^\star}(\cdot\given s)}{\tau} .
\end{align}
Similarly, we also define $\theta^\star_{Q_t}$ and $\nu^\star_{Q_t}$ as:
\begin{align}
\theta^\star_{Q_t} = \frac{- Q_t h_{\nu^\star_{Q_t}}(\cdot\given s)}{\tau},\qquad\qquad 
\nu^\star(s) = \frac{{Q_t}^\top g_{\theta^\star_{Q_t}}(\cdot\given s)}{\tau} .
\end{align}
 We have: 
\begin{align}
\label{eq:Markov_1}
\| \theta_t - \theta^\star_{Q^\star} \|^2 = \| \theta_t - \theta^\star_{Q_t} + \theta^\star_{Q_t} -  \theta^\star \|^2 \leq  2\| \theta_t - \theta^\star_{Q_t} \|^2 + 2\| \theta^\star_{Q_t} -  \theta^\star \|^2.
\end{align}
Now, the first term converges approximately after the inner loop terminates.  We can analyze the second term as follows: 
\begin{align}
\| \theta^\star_{Q_t} -  \theta^\star \|^2 &=^{*1} \frac{1}{\tau} \| Q_t h_{\nu^\star_{Q_t}} -  Q^\star h_{\nu^\star_{Q^\star}} \|^2 =  \frac{1}{\tau} \| Q_t h_{\nu^\star_{Q_t}} -Q^\star h_{\nu^\star_{Q_t}} + Q^\star h_{\nu^\star_{Q_t}} - Q^\star h_{\nu^\star_{Q^\star}} \|^2 \nonumber \\
&\leq   \frac{2}{\tau} \left(  \| Q_t h_{\nu^\star_{Q_t}} -Q^\star h_{\nu^\star_{Q_t}}  \|^2 + \| Q^\star h_{\nu^\star_{Q_t}} - Q^\star h_{\nu^\star_{Q^\star}} \|^2 \right) \nonumber \\
&\leq  \frac{2}{\tau} \left(  \| Q_t -Q^\star  \|^2_F + \|Q^\star\|^2_F \|  h_{\nu^\star_{Q_t}} - h_{\nu^\star_{Q^\star}} \|^2 \right) \nonumber \\
&\leq^{*2}  \frac{2}{\tau} \left(  \| Q_t -Q^\star  \|^2_F + C \|Q^\star\|^2_F\| Q_t -Q^\star  \|^2_F \right) =  \frac{2}{\tau} \left(1 + C^2 \|Q^\star\|^2_F \right)  \| Q_t -Q^\star  \|^2_F,
\end{align}
where $(*1)$ follows from the definition of $\theta^\star$ and $(*2)$ follows from Lemma \ref{lemma:smooth_wrtQ}.
Substituting this in Equation \eqref{eq:Markov_1} we have: 
\begin{align}
\| \theta_t - \theta^\star_{Q^\star} \|^2 &\leq 2\| \theta_t - \theta^\star_{Q_t} \|^2 + \frac{4}{\tau}  \left(1 + C^2 \|Q^\star\|^2_F \right)  \| Q_t -Q^\star  \|^2_F.
\end{align}
Therefore, we have $\|\theta_t - \theta^\star_{Q^\star} \|^2 \leq \epsilon$ if 
\begin{align}
\| \theta_t - \theta^\star_{Q_t} \|^2 \leq \frac{\epsilon}{4} , \qquad  \| Q_t -Q^\star  \|^2_F \leq \frac{\epsilon}{\frac{8}{\tau}  \left(1 + C^2 \|Q^\star\|^2_F \right) }.
\end{align}
Note that the first term can be achieved by setting the inner-loop iterations $T_{inner}$ to be the following (from Theorem \ref{thm:param_con}):
\begin{align}
T_{inner} = \mathcal{O} \left( \frac{1}{\eta \tau} \left( \log \frac{1}{\epsilon} + \log \frac{1}{1 - \gamma} + \log \log n + \log \frac{1}{\eta} \right) \right),
\end{align}
and the second term can be achieved by noting that:
\begin{align}
 \| Q_t -Q^\star  \|_F \leq d \| Q_t -Q^\star  \|_\infty,
 \end{align}
 for $d = |\cS| \times |A|$. Now, using Inequality \eqref{eq:ref_Q_conv_cen}, we can set the outer-loop iterations $T_{outer}$ to be: 
\begin{align}
T_{outer} = \mathcal{O} \left( \frac{1}{1 - \gamma} \left( \log \frac{d}{\epsilon} + \log \left( \frac{8}{\tau}  \left(1 + C^2 \|Q^\star\|^2_F \right) \right)  +\log \frac{1 + \tau \log n}{1 - \gamma} \right) \right),
\end{align}
to get the desired convergence result. This completes the proof. \hfil \qed

\subsection{Function approximation setting}
\label{sec:appendix_func_approx_markov_results}


In this subsection, we discuss Markov games where the policies have a log-linear parametrization. The basic idea is to follow the tabular setting, but only for those states for which there is an action for which the feature vector corresponding to the state-action pair is non-zero. We first make the following assumption on the feature matrix $\phi$.

\begin{assumption}
\label{ass:phi_mg}
The feature matrix $\Phi$ is full rank, Moreover, it is of the form $\Phi=[\phi_1, \phi_2, \cdots , \phi_{|\cS|\times|\cA|}]=[I \ | \ 0]$.
\end{assumption}

Note that this assumption is similar to Assumption \ref{ass:full_rank} for the matrix game. This The full rank assumption is standard in the literature. Furthermore, this  particular structure of the feature matrix, though being restrictive, ensures that the constraint set of policies is convex (similar to the case of matrix games), otherwise the minimax theorem of $\min\max=\max\min$  might not hold, i.e., the Nash equilibrium for the parameterized game does not exist. See the paragraph below Assumption \ref{ass:full_rank} for further discussion on the structure of the feature matrix.
We next describe the detailed description of the setup. 

\subsubsection{Setup}

Each column of $\Phi$ is a feature vector corresponding to some state action pair $(s, a)$. Note that for each state, there could be  $0$ to $\min \{n, d \}$ actions for which the feature vector is non-zero.

Now, consider a state $s\in S$. Define $A_{s} = \{ a \in \mathcal{A}: \Phi_{s, a} \neq 0 \}$ 
where $\Phi_{s,a}$ corresponds to the column in the feature matrix for state $s$, and action $a$, and here $0$ denotes the zero vector. Therefore $A_s$ is the set of actions in state $s$ for which the feature vector is non-zero. For sake of notational simplicity, let these be the actions $1, 2, \cdots, | A_{s} |$. Note that $A_{s}$ can be an empty set. We further assume that the first $|A_1|$ columns of $\Phi$ are corresponding to state $1$, the next $|A_2|$ columns correspond to state $2$ and so on. Note that we have $\sum_{s \in \cS} |A_s | = d$.

For state $s$, if $A_{s}$ is nonempty, define the feature matrix $\Phi_{s} = [I_{|A_{s}|} \ |\ 0] \in \mathbb{R}^{|A_{s}| \times n}$.
Note that this would be the feature matrix corresponding to each state for the original feature matrix $\Phi$.  
Now, define $\tilde{\Delta}_s$ corresponding to each state $s$ using the  feature matrix $\Phi_s$, as in Equation \eqref{eq:Delta_tilde_def}. This corresponds to the set of admissible  distributions under the function approximation setting for state $s$. If the set $A_s$ is empty, we take $\tilde{\Delta}_s$ to be the singleton set with the uniform distribution. Furthermore, we define $\tilde{\Delta} = \times_{s \in \cS} \tilde{\Delta}_s$.

Next, for notational convenience, we let the first $d_1$ columns of the Matrix $\Phi$ correspond to state $s_{1}$, i.e., $d_1= |A_{s_1}|$, the next $d_2$ columns correspond to actions in state $s_{2}$ and so on till finally the columns from $d - d_D +1$ to column $d$ correspond to state $s_{D}$, i.e., we partition the columns for which the feature vector is non-zero into the different states they correspond to. Therefore, $D$ corresponds to the number of states for which there is at least one action for which the feature vector corresponding to the state action pair is non-zero.
This means that the states $s_{1}, s_{2}, \cdots, s_{D}$ are the only states for which there is at least one action with a nonzero feature vector, and therefore $\tilde{\Delta}_s$ is not a singleton set for these states. For all other states $s \in \cS \backslash \{ s_{1}, s_{2}, \cdots, s_{D} \}$, we have that  $\tilde{\Delta}_s$ is a singleton set containing the uniform distribution. We will also separate the parameters as follows: $\theta_{s_1}$ denotes the first $d_1$ elements of $\theta$, $\theta_{s_2}$ denotes the next $d_2$ elements of $\theta$ and so on. Similarly for $\nu$.

Now, the algorithm used to solve the Markov game in this function approximation setting, is similar to the tabular setting, except that we only have to run the inner iteration on the states $s_{i}, \ i = \{ 1, 2, \cdots, D \}$. We describe the algorithm in detail in Algorithm \ref{alg:markov_game_func} in \S\ref{sec:append_missing_ONPG_Markov}.




\begin{theorem}
\label{thm:Markov_game_convergence_func}
Let $Q^\star_{\tau}$ be the in-class NE (Definition \ref{def:para_eps_NE}) $Q$-value defined in Equation \eqref{eq:def_optimal_Q_func} of the regularized problem, under the log-linear parametrization satisfying Assumption \ref{ass:phi_mg}. Note that we have existence of the in-class NE from Lemma \ref{lemma:exist_para_NE_FA}. Choose the stepsize $\eta = \frac{1-\gamma}{2(1 + \tau(\log n + 1 - \gamma))}$ for the inner loop in Algorithm \ref{alg:markov_game_func}.
Then, after running Algorithm \ref{alg:markov_game_func} for 
\begin{align}
T_{inner} &= \mathcal{O} \left( \frac{1}{\eta \tau} \left( \log \frac{1}{\epsilon} + \log \frac{1}{1 - \gamma} + \log \log n + \log \frac{1}{\eta} \right) \right), \nonumber \\
T_{outer} &= \mathcal{O} \left( \frac{1}{1 - \gamma} \left( \log \frac{D}{\epsilon} + \log \left( \frac{8}{\tau}  \left(1 + C^2 \|Q^\star\|_F^2 \right) \right)  +\log \frac{1 + \tau \log n}{1 - \gamma} \right) \right),
\end{align}
iterations, we have $\| Q_T - Q^\star_{\tau} \|_\infty \leq \epsilon$ and $\max\{ \| \theta_T - \theta^\star \| , \| \nu_T - \nu^\star \| \} \leq \epsilon$
where $(Q_T, \theta_T, \nu_T)$ is the output of Algorithm \ref{alg:markov_game_func}  after $T$ iterations and $(\theta^\star, \nu^\star)$ are defined in Equation \eqref{eq:def_optimal_sol}.
\end{theorem}

 \begin{remark}
 Note that Theorem \ref{thm:Markov_game_convergence_func} provides the convergence rate for a two player Markov game under the function approximation setting. This covers the tabular case by setting the feature matrix, $\Phi$ (See Equation \eqref{eq:Phi_def} in \S\ref{sec:ONPG_Matrix_FA}), to be equal to the identity matrix. In particular, making this substitution recovers the results of Theorem \ref{thm:Markov_game_convergence} . 
 \end{remark}

\begin{remark}
We remark that Theorem \ref{thm:Markov_game_convergence_func}, to the best of our knowledge, provides the first symmetric algorithm with convergence rate  guarantees for Markov games under the function approximation setting. {The only other existing result in this setting is \cite{zhao2021provably}, where the update is asymmetric, and one of the players performs multiple updates while the other player updates once. An asymmetric update-rule also appears in \cite{daskalakis2020independent}, without function approximation.  
Our results also improve over  \cite{wei2021last,cen2021fast} by generalizing the results to the case of certain function approximation, as well as showing parameter convergence.}
\end{remark}

\subsubsection{Proof of Theorem \ref{thm:Markov_game_convergence_func}}
\label{sec:appendix_func_approx_markov}


\begin{algorithm}[tb]
   \caption{Optimistic NPG for Markov Games with Function Approximation}
   \label{alg:markov_game_func}
\begin{algorithmic}
   \STATE {\bfseries Initialize:} $Q_0 = 0$
   \FOR{$t=1, 2, \cdots, T_{outer}$}
     \STATE Let $Q_t(s, a, b) := r(s, a, b) + \gamma \mathbb{E}_{s' \sim \mathbb{P}(\cdot|s, a, b)} [- f_{\tau} (Q_t(s'); g_{\theta_{T_{inner}}}(\cdot\given s'), h_{\nu_{T_{inner}}}(\cdot\given s'))]$
     \FOR{$i = 1, 2, \cdots, D$}
     \STATE Solve $\min_{{\theta} \in \mathbb{R}^{|A_{s_{i}}|}} \ \max_{{\nu}  \in \mathbb{R}^{|B_{s_{i}}|} } f_{\tau} (Q(s_{i}); g_{\theta},h_{\nu})$  by running the Optimistic NPG algorithm (Algorithm \ref{alg:optimistic_npg_func}) with feature matrix $\Phi_{s_{i}}$ for $T_{inner}$ iterations and return the last iterates $(\theta_{s_i}^{T_{inner}}, \nu_{s_i}^{T_{inner}})$.
     \ENDFOR
     \STATE Set $(\theta_{T_{inner}} , \nu_{T_{inner}} ) = \left( [\theta_{s_1}^{T_{inner}}, \theta_{s_2}^{T_{inner}}, \cdots, \theta_{s_D}^{T_{inner}}], [\nu_{s_1}^{T_{inner}}, \nu_{s_2}^{T_{inner}}, \cdots, \nu_{s_D}^{T_{inner}}] \right)$.
   \ENDFOR
\end{algorithmic}
\end{algorithm}

In order to characterize the point where the parameter converges  to, we define: 
\begin{align} 
\label{eq:def_optimal_sol}
\theta^\star = [(\theta_{s_{1}}^\star)^\top, (\theta_{s_{2}}^\star)^\top, \cdots, (\theta_{s_{D}}^\star)^\top]^\top,\qquad \qquad 
\nu^\star = [(\nu_{s_{1}}^\star)^\top, (\nu_{s_{2}}^\star)^\top, \cdots, (\nu_{s_{D}}^\star)^\top]^\top,
\end{align} 
where 
\begin{align}
\theta_{s_{i}}^\star = - [I_{|A_{s_{i}}|} | 0]\frac{Q_\tau^\star(s_{i}) h_{\nu^\star}(\cdot\given s_{i})}{\tau} \in \mathbb{R}^{A_{s_{i}}},\qquad \qquad 
\nu_{s_{i}}^\star = [I_{|A_{s_{i}}|} | 0]\frac{Q_\tau^\star(s_{i})^{\top} g_{\theta^\star}(\cdot\given s_{i})}{\tau} \in \mathbb{R}^{A_{s_{i}}} \nonumber,
\end{align}
and $Q_\tau^\star$ is the in-class NE Q-value (see Definition \ref{def:para_eps_NE}). The existence of this in-class NE follows from Lemma \ref{lemma:exist_para_NE_FA}.

The {in-class} Nash equilibrium $(\theta^\star, \nu^\star)$ under the function approximation setting satisfies:
\begin{align}
V^{\theta^\star, \nu}(s) \geq V^{\theta^\star, \nu^\star}(s) \geq V^{\theta, \nu^\star}(s) \qquad \qquad \forall \theta, \nu \in \mathbb{R}^d, \ \forall s \in \cS.
\end{align}
Note that this is equivalent to (from Lemma \ref{lemma:char_func_approx}): 
\begin{align}
V^{g_{\theta^\star}, h_{\nu}}(s) \geq V^{g_{\theta^\star}, h_{\nu^\star}}(s) \geq V^{g_{\theta}, h_{\nu^\star}}(s) \qquad \qquad \forall g_\theta, h_\nu \in \tilde{\Delta}, \ \forall s \in \cS.
\end{align}
We denote the NE V (and Q) as $V^\star$ (and $Q^\star$) (We use this notation for the general regularized version with $\tau \geq 0$, i.e., we do not explicitly state the dependence on $\tau$)
, i.e.,
\begin{align}
Q^\star (s, a, b) = r(s, a, b) + \gamma \mathbb{E}_{s' \sim \mathbb{P}(\cdot|s, a, b)} [V^\star (s')].
\end{align}

Next, we define the soft-Bellman operator as \footnote{Note that we use the structure of the feature matrix $\Phi$, along with Theorem \ref{thm:Nash_char_func_approx}, to show that $\min_{g_\theta(\cdot\given s)\in\tilde\Delta}\max_{h_\nu(\cdot\given s)\in\tilde\Delta}$ is equivalent to $\min_{{\theta}_s \in \mathbb{R}^{|A_s|}} \ \max_{{\nu}_s \in \mathbb{R}^{|A_s|}}$}: 
\begin{align}
\label{eq:soft_bellman_func}
\mathcal{T}_{\tau}(Q)(s, a, b) := r(s, a, b) + \gamma \mathbb{E}_{s' \sim \mathbb{P}(\cdot|s, a, b)} [ - \min_{{\theta} \in \mathbb{R}^{|A_s|}} \ \max_{{\nu} \in \mathbb{R}^{|A_s|}} f_{\tau} (Q(s'); g_\theta(\cdot\given s'),h_\nu(\cdot\given s'))],
\end{align}
where
\begin{align}
f_{\tau} (Q(s); g_\theta(\cdot\given s),h_\nu(\cdot\given s) &:= -\ g_\theta(\cdot\given s) ^{\top} Q(s) h_\nu(\cdot\given s) - \tau \mathcal{H}(g_\theta(\cdot\given s)) + \tau \mathcal{H}(h_\nu(\cdot\given s)), 
\end{align}
and $Q(s)$ is the Q-value matrix at state $s$. Let $\theta_s, \nu_s$ be the NE parameters at state $s$.
Note that by $NE$ we mean the solution to the min-max problem (which is the in-class NE in the function approximation setting). Then  the concatenation $\theta = [\theta_{s_1}^\top, \theta_{s_2}^\top, \cdots, \theta_{s_D}^\top]^\top$ (and similarly for $\nu$) denotes the parameters. 

Note that the inner $\min \max$ problem is equivalent to (from \S \ref{sec:ONPG_Matrix_FA}): 
\begin{align}
\min_{g_{\theta}(\cdot|s') \in \tilde{\Delta}_{s'}} \ \max_{h_{\nu}(\cdot|s') \in\tilde{\Delta}_{s'} } ~~~~f_{\tau} (Q(s'); g_\theta(\cdot\given s'),h_\nu(\cdot\given s')).
\end{align}

Consider the value iteration:
\begin{align}
Q_{t+1} = \mathcal{T}_{\tau}(Q_t).
\end{align}
We have $\|Q_t - Q^\star \|_\infty \leq \gamma^t \| Q_0 - Q^\star \|_\infty,$
due to the non-expansiveness  property of the $\min \ \max$ operator and the contracting factor $\gamma  < 1$. Note from Lemma \ref{lemma:exist_para_NE_FA},
this fixed point in fact corresponds to the in-class NE $Q-$value matrix of the regularized Markov game. 


The inner problem, which solves the saddle-point problem, is solved for each state with the input feature matrix $\Phi_s$ using Algorithm \ref{alg:optimistic_npg_func}. The iteration complexity follows from a similar analysis to the tabular case. 
This completes the proof. \hfil \qed

\subsection{Proof of Lemma \ref{lemma:exist_para_NE}}\label{sec:proof_lemma_exist_para_NE}



Consider the operator \footnote{Note that we here we use the structure of the tabular parametrization, by noting that $\max_{g_\theta(\cdot\given s)\in \Delta}\min_{h_\nu(\cdot\given s)\in \Delta}$ is equivalent to $\max_{{\theta}_s \in \mathbb{R}^{|A|}} \ \min_{{\nu}_s \in \mathbb{R}^{|A|}}$ using Theorem \ref{thm:con_MWU}.} 
\begin{align}
\label{eq:V_contrac_operator}
\mathcal{T}(V)(s) :=\max_{{\theta_s} \in \mathbb{R}^{|A|}} \ \min_{{\nu_s} \in \mathbb{R}^{|A|} } \ \mathbb{E}_{a \sim g_{\theta}(\cdot | s), b \sim h_{\nu}(\cdot | s)} \left[ r(s, a, b) + \gamma \mathbb{E}_{s' \sim \mathbb{P}(\cdot|s, a, b)} [ V(s')] \right]. 
\end{align}
The proof here is for the regularized game, i.e., with $\tau \geq 0$\footnote{We suppress the dependency on $\tau$ in the notation, by dropping the subscript $\tau$ for $Q$, $V$, and the operator $\cT$.}.
Note that this is the operator for the value function $V$ corresponding to the $Q$-value operator in Equation \eqref{eq:soft_bellman_func} with $\tau = 0$. We have, from the nonexpansive property of the min-max operator (see for example \cite{filar2012competitive}): 
\begin{align}
\| \mathcal{T}(V_1) - \mathcal{T}(V_2) \|_\infty \leq \gamma \| V_1 - V_2 \|_\infty,
\end{align}
which shows that this is a contracting operator and therefore has a unique fixed point by the Banach Fixed Point theorem. We show that this fixed point will lead to the in-class Nash equilibrium policy defined in Definition \ref{def:para_eps_NE}. Let the fixed point be denoted by $V^\star$, and let $(\theta^\star, \nu^\star)$ be the maxmin policy parameters in Equation \eqref{eq:V_contrac_operator} when plugging in $V^\star$. Note that $\theta^\star = (\theta_{s}^\star)_{s=1}^{|\cS|}$\footnote{Since we are in the tabular setting, each element $\theta^\star_s$ is of length $|A|$.} and similarly $\nu^\star$. We will show that $(\theta^\star, \nu^\star)$ is in fact the NE policy parameters.
We have:
\begin{align}
V^\star(s) &= \mathbb{E}_{a \sim g_{\theta^\star}(\cdot | s), b \sim h_{\nu^\star}(\cdot | s)} \left[ r(s, a, b) + \gamma \mathbb{E}_{s' \sim \mathbb{P}(\cdot|s, a, b)} [ V^{\star}(s')] \right] \nonumber \\
&=  \min_{{\nu}_s \in \mathbb{R}^{|A|} } \ \mathbb{E}_{a \sim g_{\theta^\star}(\cdot | s), b \sim h_{\nu}(\cdot | s)} \left[ r(s, a, b) + \gamma \mathbb{E}_{s' \sim \mathbb{P}(\cdot|s, a, b)} [ V^{\star}(s')] \right] \nonumber \\
&\leq \mathbb{E}_{a \sim g_{\theta^\star}(\cdot | s), b \sim h_{\nu}(\cdot | s)} \left[ r(s, a, b) + \gamma \mathbb{E}_{s' \sim \mathbb{P}(\cdot|s, a, b)} [ V^{\star}(s')] \right]  \qquad \forall \ {\nu} \in \mathbb{R}^{d}.\label{ineq:V_lower_bound_Filar}
\end{align}

Now applying Lemma 4.3.3 in \cite{filar2012competitive} 
to Inequality  \eqref{ineq:V_lower_bound_Filar} for all states $s$, we have:
\begin{align}
V^{\star} \leq V^{\theta^\star, \nu} \qquad \forall \ {\nu} \in \mathbb{R}^{d},  
\end{align}
where $d = |\cS | \times |A|$. Applying the same inequality for $\nu^\star$, we have: 
\begin{align}
\label{eq:last_lemma_saddle}
V^{\theta^\star, \nu} \geq V^{\star} \geq V^{\theta, \nu^\star} \qquad \forall\ \theta, {\nu} \in \mathbb{R}^{d}.
\end{align}
Furthermore, since we have: 
\begin{align}
V^\star &= \mathbb{E}_{a \sim g_{\theta{^\star} (\cdot|s)}, b \sim h_{\nu^\star} (\cdot|s)} \left[ r(s, a, b) + \gamma \mathbb{E}_{s' \sim \mathbb{P}(\cdot|s, a, b)} [ V^{\star}(s')] \right]. 
\end{align}
Lemma 4.3.3 in \cite{filar2012competitive} gives us that $V^\star = V^{\theta^\star, \nu^\star}$. Combining this with Inequality \eqref{eq:last_lemma_saddle}, we have:
\begin{align}
V^{\theta^\star, \nu} \geq V^{\theta^\star, \nu^\star} \geq V^{\theta, \nu^\star} \qquad \ \forall \theta, {\nu} \in \mathbb{R}^{d},
\end{align}
which shows that $(\theta^\star, \nu^\star)$ is the required NE.
Finally, using Theorem 4.3.2 (iii) in \cite{filar2012competitive},  we can find the NE $(\theta^\star, \nu^\star)$ for each state $s\in\cS$ by solving the following matrix game (for which the solution is guaranteed to exist, from Theorem \ref{thm:con_MWU}) : 
\$
\max_{{\theta_s} \in \mathbb{R}^{|A|}} \ \min_{{\nu_s} \in \mathbb{R}^{|A|} } \ \mathbb{E}_{a \sim g_{\theta}(\cdot|s), b \sim h_\nu (\cdot|s)} \left[ r(s, a, b) + \gamma \mathbb{E}_{s' \sim \mathbb{P}(\cdot|s, a, b)} [ V^\star(s')] \right].
\$
This completes the proof. 

\hfil \qed

\begin{lemma}[Existence of parameterized/in-class NE under Linear FA]\label{lemma:exist_para_NE_FA}
	Under policy parameterization \eqref{equ:policy_para_markov} with log linear policy parametrization and Assumption \ref{ass:phi_mg}, the  in-class NE defined in Definition \ref{def:para_eps_NE} exists.
\end{lemma}
\begin{proof}
The proof follows exactly along the lines of Lemma \ref{lemma:exist_para_NE}, along with the fact that the matrix game under linear function approximation and Assumption \ref{ass:phi_mg} has a solution, as shown in Theorem \ref{thm:Nash_char_func_approx}.
\end{proof}


%% file: NPG_background.tex

\section{Missing Definitions and Proofs in \S\ref{sec:prelim}}\label{sec:append_background_NPG}





\subsection{Proof of Lemma \ref{lemma:ncnc}}

We show that the problem 
\begin{align}
\min_{\theta \in \mathbb{R}^n} \ \max_{\nu \in  \mathbb{R}^n}  \ g_{\theta} ^{\top} Q h_{\nu},
\end{align}
is nonconvex-nonconcave.

Let
\begin{align}
Q = \begin{bmatrix}
1 & 0 \\
0 & 1
\end{bmatrix}.
\end{align}

Consider $\theta_1 = (0,0)$ and $\theta_2 = (\log 4, \log 9)$. This implies $g_{\theta_1} = (1/2, 1/2)^\top$ and $g_{\theta_2} = (4/13, 9/13)^\top$. 
Also, from the form of $Q$, we have $Qh_{\nu} = [h_{\nu}(1),  h_{\nu}(2)]^\top,~ \ \forall \nu$. Now, for $[h_{\nu}(1),  h_{\nu}(2)] = [1/3, 2/3]$, we have
\begin{align}
\frac{1}{2} (g_{\theta_1} ^{\top} Q h_{\nu} + g_{\theta_2} ^{\top} Q h_{\nu}) < g_{(\theta_1 + \theta_2)/2}Q h_{\nu},
\end{align}
which implies nonconvexity in $\theta$.

Similarly, taking $\nu_1 = (0,0)$ and $\nu_2 = (\log 4, \log 9)$ (which implies $h_{\nu_1} = (1/2, 1/2)^\top$ and $h_{\nu_2} = (4/13, 9/13)^\top$, and taking $g_{\theta} = (2/3, 1/3)^\top$, we have
\begin{align}  
\frac{1}{2} (g_{\theta} ^{\top} Q h_{\nu_1} + g_{\theta} ^{\top} Q h_{\nu_2}) > g_{\theta}Q h_{(\nu_1 + \nu_2)/2},
\end{align}
which implies nonconcavity in $\nu$. 

Note that adding regularization does not get rid of this convexity. For example consider the specific case when $h_\nu$ is a constant policy and the matrix $Q$ is 0. We show the nonconvexity of the function $- \tau \mathcal{H}(g_{\theta})$ in $\theta$ next. Consider $\theta_1 = (0,0)$ and $\theta_2 = (\log 4, \log 9)$. We have $g_{\theta_1} = (1/2, 1/2)$ and $g_{\theta_2} = (4/13, 9/13)$. Furthermore, we have $g_{(\theta_1 + \theta_2)/2} = (1/3, 2/3)$, and we can see that:
\begin{align}
\frac{-\tau}{2} ( \mathcal{H}(g_{\theta_1}) + \mathcal{H}(g_{\theta_2}) ) < -\tau \mathcal{H}(g_{(\theta_1 + \theta_2)/2})
\end{align}
which shows nonconvexity of $- \tau \mathcal{H}(g_{\theta})$. This completes the proof of the lemma.\hfil \qed

\subsection{Vanilla NPG for matrix games}\label{sec:vanilla_NPG_deriv}

Next, we compute the Fisher Information Matrix $F_\theta(\theta) = \EE_{a\sim g_\theta}\big[\big(\nabla_\theta\log g_\theta(a)\big)\big(\nabla_\theta\log g_\theta(a)\big)^\top\big]$. For the softmax parametrization, we have:
\begin{align} 
\nabla_\theta\log g_\theta(a) = \nabla_\theta \left( \theta(a) - \log\Big(\sum_{a' \in \mathcal{A}} e^{\theta(a')} \Big)\right) = [-g_{\theta}(1), -g_\theta(2), \cdots, 1 - g_\theta(a), \cdots, -g_\theta(n) ]^{\top}.
\end{align}
Now, consider the $(i,i)^{th}$ element of the Fisher information matrix. We have:
\begin{align}
[F_\theta(\theta)]_{ii} = g_\theta(i) (1 - g_\theta(i) ) (1 - g_\theta(i)) + \sum_{j \neq i} g_{\theta}(j) g_\theta(i)^2 = g_\theta(i) (1 - g_\theta(i) ).
\end{align}
Similarly, we have the $(i,j)^{th}$ element, where $i \neq j$ is given by:
\begin{align}
[F_\theta(\theta)]_{ij} &= (1 - g_\theta(i) - g_\theta(j)) g_\theta(i) g_\theta(j) - g_\theta(i)(1 - g_\theta(i))g_\theta(j) - g_\theta(j) (1 - g_\theta(j)) g_\theta(i) = -g_\theta(i)g_\theta(j).
\end{align}
Therefore, the matrix $F_\theta(\theta)$ can be succinctly written as:
\begin{align}
F_\theta(\theta) = \text{diag} (g_\theta) - g_\theta g_\theta^{\top},
\end{align}
where $\text{diag} (g_\theta)$ is a diagonal matrix with entries $g_\theta$. Note that this is in fact $\nabla_\theta g_\theta$ (see \cite{mei2020global}), i.e., 
\begin{align}
\nabla_\theta g_\theta = \text{diag} (g_\theta) - g_\theta g_\theta^{\top}.
\end{align}
Therefore, we have:
\begin{align}
F^{\dagger}_\theta(\theta) \nabla_\theta g_\theta = I.
\end{align}
The update of the vanilla NPG thus simplifies to the following: 
\begin{align}
{\theta}_{t+1} &={\theta}_{t}-\eta \cdot F^{\dagger}_\theta(\theta_t)\cdot\frac{\partial f_\tau(\theta_t,\nu_t)}{\partial\theta} ={\theta}_{t}-\eta\frac{\partial f_\tau(\theta_t,\nu_t)}{\partial g_\theta} \nonumber \\
&= \theta_t - \eta \left( Qh_{\nu_t} + \tau (\mathbbm{1} + \log g_{\theta_t}) \right).
\end{align}
However, since $g_\theta(a) = \frac{e^{\theta(a)}}{\sum_{a' \in \mathcal{A}} e^{\theta(a')}}$, we have
\begin{align}
\theta_{t+1}(a) = (1 - \eta \tau) \theta_t(a) - \eta \bigg( [Q h_{{\nu}_t}]_a + \tau - \tau \log \sum_{a' \in\cA} e^{{\theta_t}(a')}\bigg).
\end{align}
A similar update for $\nu$ leads to the updates in Equations \eqref{equ:vanilla_NPG_1}-\eqref{equ:vanilla_NPG_2}. Note that when we write a constant in the update, we mean a constant vector with all elements being the same.

\subsection{Proof of  Lemma \ref{lemma:pitfall_NPG}} 

We restate the lemma here first for convenience:

\begin{lemma}[Pitfall of vanilla NPG]
	There exists a game \eqref{prob:entr_reg_matrix} with $\tau \geq 0$ (we allow for unregularized games as well) and a dummy player 2, i.e., $|\cB|=1$, for which the updates \eqref{equ:vanilla_NPG_1}-\eqref{equ:vanilla_NPG_2} do not converge for any $\eta>0$. 
\end{lemma}
\begin{proof}
Consider the $\theta$ update under NPG:
\begin{align}
\theta_{t+1}(a) = (1 - \eta \tau) \theta_t(a) - \eta \bigg( [Q h_{{\nu}_t}]_a + \tau - \tau \log \sum_{a' \in\cA} e^{{\theta}_t(a')}\bigg).
\end{align}
From here, it is easy to see that it need not converge for the case where $\tau = 0$, since this would require $Qh_{\nu_t} = 0$ which need not be the case (For example consider $Q = [1 \ ~|~ \ 1]^\top$, and $h_{\nu} = [1]$. In this case, $h_{\nu_t} = 1$ for any parameter $\nu_t$). 

Next, we consider the case where $\tau > 0$.
Suppose $\theta$ converges to some point $\theta^\star$. Since $|\mathcal{B}| = 1$, we have $h_{\nu_t} = [1]$. 
Substituting the point $\theta^\star$ into the update we have:
\begin{align}
\theta^\star(a) = (1 - \eta \tau) \theta^\star(a) - \eta \bigg( [Q]_a + \tau - \tau \log \sum_{a' \in\cA} e^{{\theta^\star}(a')}\bigg).
\end{align}
This implies:
\begin{align}
\eta \tau \theta^\star(a) = -\eta \tau \left( \frac{[Q]_a}{\tau} + 1\right) + \eta \tau \log \sum_{a' \in\cA} e^{{\theta^\star}(a')}.
\end{align}
This leads to:
\begin{align}
\label{eq:contra_almost}
\log e^{\theta^\star(a)} -  \log \sum_{a' \in\cA} e^{{\theta^\star}(a')} = -  \frac{[Q]_a}{\tau} - 1.
\end{align}
However, 
\begin{align}
\log e^{\theta^\star(a)} -  \log \sum_{a' \in\cA} e^{{\theta^\star}(a')} = \log \frac{e^{\theta^\star(a)}}{ \sum_{a' \in\cA} e^{{\theta^\star}(a')}} = \log g_{\theta^\star}(a).
\end{align}
Substituting this in Equation \eqref{eq:contra_almost}, we have:
\begin{align}
g_{\theta^\star}(a) = \exp \left( -  \frac{[Q]_a}{\tau} - 1 \right).
\end{align}
This need not be a valid probability measure. For example, consider $Q = [-2 \ ~|~ \-2]^{\top}$ 
and $\tau = 1$, we have:
\begin{align}
g_{\theta^\star}(1) = e > 1,
\end{align}
which contradicts the fact that $g_\theta$ is a probability measure. This implies that the original NPG updates cannot have a fixed point, and therefore does not converge for any stepsize $\eta>0$.
\end{proof}
